\documentclass[11pt,a4paper]{amsart}

\usepackage{amsmath}
\usepackage{amssymb}
\usepackage{amsthm}
\usepackage{geometry}
\usepackage{hyperref}
\usepackage[foot]{amsaddr}
\usepackage[usenames]{color}

\numberwithin{equation}{section}

\usepackage{mathrsfs}
\usepackage{mathtools}

\theoremstyle{plain}
\newtheorem{theorem}{Theorem}[section]
\newtheorem{lemma}[theorem]{Lemma}
\newtheorem{proposition}[theorem]{Proposition}
\newtheorem{corollary}[theorem]{Corollary}

\theoremstyle{definition}
\newtheorem{definition}[theorem]{Definition}
\newtheorem{remark}[theorem]{Remark}

\let\temp\phi
\let\phi\varphi
\let\varphi\temp
\let\temp\epsilon
\let\epsilon\varepsilon
\let\varepsilon\temp

\newcommand{\R}{\mathbb{R}}
\newcommand{\N}{\mathbb{N}}
\newcommand{\B}{\mathcal{B}}

\newcommand{\D}{\mathcal{D}}
\newcommand{\Lip}{\mathrm{Lip}}
\newcommand{\haus}{\mathcal{H}}
\newcommand{\cI}{\mathcal{I}}
\newcommand{\T}{\mathcal{T}}
\newcommand{\A}{\mathcal{A}}
\newcommand{\Leb}{\mathcal{L}}
\newcommand{\RCD}{\mathsf{RCD}}
\newcommand{\CD}{\mathsf{CD}}
\newcommand{\PP}{\mathsf{P}}
\newcommand{\MCP}{\mathsf{MCP}}

\newcommand{\M}{\mathcal{M}}

\DeclareMathOperator*{\argmax}{arg\,max}

\DeclareMathOperator*{\essinf}{ess\,inf}
\DeclareMathOperator*{\esssup}{ess\,sup}
\DeclareMathOperator{\Dom}{Dom}
\DeclareMathOperator{\Geo}{Geo}
\DeclareMathOperator{\Ric}{Ric}
\DeclareMathOperator{\diam}{diam}
\DeclareMathOperator{\dom}{Dom}
\DeclareMathOperator{\id}{Id}
\DeclareMathOperator{\supp}{supp}
\DeclareMathOperator{\Res}{Res}
\DeclareMathOperator{\vol}{Vol}

\newcommand{\norm}[1]{\left\lVert#1\right\rVert}

\newcommand{\I}{\mathcal{I}}
\newcommand{\m}{\mathfrak{m}}
\newcommand{\q}{\mathfrak{q}}
\newcommand{\ProbMeas}{\mathcal{P}}
\newcommand{\mm}{\mathfrak m}
\newcommand{\qq}{\mathfrak q}
\newcommand{\QQ}{\mathfrak Q}
\newcommand{\sfd}{\mathsf d}
\newcommand{\AVR}{\mathsf{AVR}}
\newcommand{\indicator}{\mathbf{1}}

\newcommand{\relation}{\mathcal{R}}

\begin{document}

\author{Davide Manini}
\email{dmanini@campus.technion.ac.il}
\address{Department of Mathematics, Technion --- Israel Institute of
  Technology, Haifa, Israel}
\title{Isoperimetric inequality for Finsler manifolds with
  non-negative Ricci curvature}
\subjclass[2010]{58B20,51Fxx}
\begin{abstract}
We prove a sharp isoperimetric inequality for measured Finsler manifolds
having non-negative Ricci curvature and Euclidean volume growth.
We also prove a rigidity result for this inequality, under the
additional hypotheses of boundedness of the isoperimetric set and
the finite reversibility of the space.

As a consequence, we deduce the rigidity of the weighted anisotropic
isoperimetric inequality for cones in the Euclidean space, in the
irreversible setting.
\end{abstract}
\maketitle
\tableofcontents

\section{Introduction}
The classical isoperimetric inequality in the Eucludean space 
states that 
if $E$ is a (sufficiently regular) subset of $\R^d$, then 
\begin{equation}\label{eq:isop-classic}
  \PP(E)
  \geq
  d \omega_{d}^{\frac{1}{d}} \Leb^d(E)^{1-\frac{1}{d}}
  ,
\end{equation}
where $\PP(E)$ denotes the perimeter of the set $E$, $\omega_d$ the
measure of the unit ball in $\R^d$, and $\Leb^d$ the Lebesgue measure.
Moreover, if the equality is attained in~\eqref{eq:isop-classic} by a certain
set $E$ with positive measure, then the set $E$ coincides with a ball
of radius $(\frac{\Leb^d(E)}{\omega_d})^{1/d}$.
This inequality has been successfully extended in more general
settings where the space is not the Euclidean one.
Indeed, it turns out that two are the relevant properties of the
Euclidean space needed for such generalizations: 1) the fact that
$\R^d$ has non-negative Ricci curvature; 2) its Euclidean volume
growth, i.e.\ a constraint on the growth of the measure of large
balls.

If $(X,g)$ is an $n$-dimensional Riemannian manifold, one can
consider different measures to the canonical volume $\vol_{g}$.
In this case, the Ricci tensor has to be replaced with the
generalised $N$-Ricci tensor:
if $h:X\to(0,\infty)$ is a weight for the volume $\vol_g$, the generalised
$N$-Ricci tensor, with $N> n$, is defined by
\begin{equation}
\label{eq:n-ricci}
\Ric_{N} : =  \Ric - (N-n) \frac{\nabla^{2} h^{\frac{1}{N-n}}}{h^{\frac{1}{N-n}}}.  
\end{equation}
We say that the weighted manifold $(X,g,h \vol_g)$ verifies the so-called
Cur\-va\-tu\-re-Dimen\-sion condition $\CD(K,N)$ (see~\cite{BakryEmery85})
whenever $\Ric_{N} \geq K g$.
%

In their seminal works Lott--Villani \cite{LoVi09} and Sturm \cite{St06, St06-2} introduced a synthetic definition of $\CD(K,N)$  for complete and separable metric spaces $(X,\sfd)$ endowed with a (locally-finite Borel) reference measure 
$\mm$ (``metric-measure space", or m.m.s.).
This $\CD(K,N)$ condition is formulated using the theory Optimal Transport
and it coincides with the Bakry--\'Emery one in the smooth Riemannian
setting (and in particular in the classical non-weighted one).

A measured Finsler manifold is a triple $(X,F,\mm)$, such that $X$ is a
differential manifold (possibly with boundary), $\mm$ a Borel measure, and $F$ a Finsler
structure, that is a real-valued function $F:TX\to[0,\infty)$, which
is convex, positively homogeneous, and $F(v)=0$ if and only if $v=0$
(see Section~\ref{Ss:finsler} for the precise definition).
In general $F(v)\neq F(-v)$.
This feature, know as irreversibility, is what prevents to apply the
techniques developed for m.m.s.'s\ to measured Finsler manifolds.
Recently, Ohta successfully extended the theory of the
Curvature--Dimension condition for possibly-irreversible Finsler
manifolds (see~\cite{Ohta09a,Ohta17,Ohta21}).
Namely, a notion of $N$-Ricci curvature (compatible with the
riemannian one) was introduced and it was proven that a measured Finsler
manifold satisfies the $\CD(K,N)$ condition if and only if $\Ric_N\geq
K$.
More recently, the notion of irreversible metric measure space has
been introduced~\cite{KristalyZhao22}.

\smallskip

The isoperimetric problem in the reversible setting has been
extensively investigated.
E.\ Milman~\cite{Mi15} gave a sharp isoperimetric inequality for weighted
riemannian manifolds satisfying the $\CD(K,N)$ condition (for any
$K\in\R$, $N> 1$), with an additional constraint on the diameter.
In particular, given $K\in\R$, $N>1$, $D\in(0,\infty]$, he gave an
explicit description of the so-called \emph{isoperimetric profile}
function $\I_{K,N,D}:[0,1]\to\R$.
The isoperimetric profile has the property that, given a weighted
riemannian manifold satisfying the $\CD(K,N)$ condition with diameter
at most $D$, whose total measure is $1$, it holds that
$\PP(E)\geq\I_{K,N,D}(v)$, for any subset $E$ of measure $v\in[0,1]$;
moreover, Milman's result is sharp.
Cavalletti and Mondino~\cite{CaMo17} extended Milman's result to the
non-smooth setting finding the same lower bound.
Their proof makes use of the localisation method (also known as needle
decomposition), a powerful dimensional reduction tool, initially
developed by Klartag~\cite{Kl17} for riemannian manifolds and later
extended to $\CD(K,N)$ spaces~\cite{CaMo17}.

In the setting of measured Finsler manifold, less is known.
Following the line traced in~\cite{CaMo17}, Ohta~\cite{Ohta18}
extended the localization method to measured Finsler manifolds, obtaining
a lower bound for the perimeter for measured Finsler manifolds with finite
reversibility constant.
The reversibility constant (introduced in~\cite{Rademacher04}) of a Finsler structure
$F$ on the manifold $X$ is defined as the least constant (possibly
infinite) $\Lambda_{F}\geq 1$ such that $F(-v)\leq \Lambda_{F} F(v)$
for all vectors $v\in TX$.
Ohta proved~\cite{Ohta18} that given a measured Finsler manifold $(X,F,\mm)$
having finite reversibility constant and $\mm(X)=1$ satisfying the
$\CD(K,N)$ condition, with diameter bounded from above by $D$,
it holds that
\begin{equation}
  \PP(E)\geq \Lambda_{F}^{-1}\, \I_{K,N,D}(\mm(E))
  ,\qquad
  \forall E\subset X
  ,
\end{equation}
where $\I_{K,N,D}$ is the isoperimetric profile function described
by E.\ Milman.
The presence of the factor $\Lambda_{F}^{-1}$ suggests that the
inequality above is not sharp.
Indeed, in the case $N=D=\infty$, this factor can be eliminated
obtaining the Barky--Ledoux isoperimetric inequality for Finsler
manifolds~\cite{Ohta22}.

\smallskip

Regarding the case $K=0$, in order to generalize the classical
inequality~\eqref{eq:isop-classic} we must drop the assumption that
the space has measure $1$ and consider the case when the measure is
infinite.
However, it is well known that without an additional condition on the
geometry of the space no non-trivial isoperimetric inequality holds in
the case $K=0$.
A way to create an Euclidean-like environment is to impose an
additional constraint on the growth of the measure of the balls that
mimics the Euclidean one.
Letting $B^{+}(x,r)=\{y:\sfd(x,y)<r\}$ denoting the \emph{forward metric ball} with center $x\in X$ and
radius $r> 0$, by the Bishop--Gromov inequality (see~\eqref{eq:bishop-gromov})
the map
$r \mapsto \frac{\mm(B^{+}(r,x))}{r^{N}}$ is nonincreasing over
$(0,\infty)$ for any $x \in X$.
The \emph{asymptotic volume ratio} is then naturally defined by
\begin{equation}
  \label{eq:avr}
\mathsf{AVR}_{X} = \lim_{r\to\infty} \frac{\mm(B^{+}(x,r))}{\omega_{N} r^{N}}.
\end{equation}
It is easy to see that it is indeed independent of the choice of $x
\in X$.
When $\AVR_{X} > 0$, we say that space has
\emph{Euclidean volume growth}.

\smallskip

In the riemannian setting, the isoperimetric inequality for spaces
with Euclidean volume growth has been obtained in increasing
generality (see, e.g.~\cite{AgFoMa20,FogagnoloMazzieri22,Brendle22,Johne21}).
The most general result is due to Balogh and Krist\'{a}ly~\cite{BaKr22}
and it is valid for (non-smooth) $\CD(0,N)$ spaces; it was proven
exploiting the Brunn--Minkowski inequality for $\CD(0,N)$ spaces.

\begin{theorem}[{\cite[Theorem 1.1]{BaKr22}}]
\label{T:BalKri}
Let $(X,\sfd,\mm)$ be a m.m.s.\ satisfying the $\CD(0,N)$ condition for some $N > 1$
and having Euclidean volume growth. 
Then for every bounded Borel subset $E \subset X$ it holds
\begin{equation}\label{E:balogh-kristaly}
\mm^{+}(E) \geq N \omega_{N}^{\frac{1}{N}} \AVR_{X}^{\frac{1}{N}} 
\mm(E)^{\frac{N-1}{N}}
\end{equation}
Moreover, inequality \eqref{E:balogh-kristaly} is sharp.
\end{theorem}

In inequality~\eqref{E:balogh-kristaly}, $\mm^+$ denotes the Min\-kow\-ski
content.
In Appendix~\ref{S:minkowski} we will discuss the relation between the Minkowski content and the perimeter; for mildly regular sets these two notions coincide.
Using the $\Gamma$-function, one can naturally define the constant
$\omega_N$ for all $N>1$.

The rigidity of the inequality has been obtained, under two mild
assumptions, by the author and Cavalletti~\cite{CavallettiManini22a}.
These assumption are: 1) the essentially-non-branching-ness of the
space, that excludes pathological cases; 2) the fact that the set
attaining equality in~\eqref{E:balogh-kristaly} is bounded.

\begin{theorem}[{\cite[Theorem 1.4]{CavallettiManini22a}}]
  \label{T:cavalletti-manini}
Let $(X,\sfd,\mm)$ be an essentially non-branching m.m.s.\ satisfying the $\CD(0,N)$ condition for some $N > 1$, 
and having Euclidean volume growth. 
Let $E \subset X$ be a bounded Borel set that saturates \eqref{E:balogh-kristaly}. 

Then there exists (a unique) $o\in X$ such that, up to a negligible set,
$E=B^+(o,\rho)$, with
$\rho=(\frac{\mm(E)}{\AVR_{X}\omega_N})^{\frac{1}{N}}$.
Moreover,
the measure $\mm$ can be represented by the disintegration formula
\begin{equation}
\mm = \int_{\partial B^{+}(o,\rho)} \mm_{\alpha} 
\, \qq(d\alpha), \qquad \qq \in \mathcal{P}(\partial B^{+}(o,\rho)),
\quad \mm_{\alpha} \in \mathcal{M}_{+}(X),
\end{equation}
where $\mm_{\alpha}$ is concentrated on the geodesic ray from $o$ through $\alpha$ and $\mm_{\alpha}$  
can be identified (via the unitary speed parametrisation of the ray) with 
$N \omega_{N} \AVR_{X} t^{N-1} \mathcal{L}^{1}\llcorner_{[0,\infty)}$.
\end{theorem}

As a consequence of this result, having in mind the fact that ``volume
cone implies metric cone''~\cite{DePhilippisGigli16}, we obtain that
when the space is $\RCD(0,N)$, we also have a rigidity of the
metric, i.e., the space is isometric to a cone over an $\RCD(N-1,N-2)$
space.
In the case of non-collapsed $\RCD$ spaces, the hypothesis on the
boundedness of the set can be
lifted~\cite{AntonelliPasqualettoPozzettaSemola23,AntonelliPasqualettoPozzettaSemola22a}
(see also~\cite{Br21}).

The scope of the present paper is to extend Theorems~\ref{T:BalKri}
and~\ref{T:cavalletti-manini} to the setting of irreversible
measurable Finsler
manifolds.

\subsection{The result}
The first result of this paper is the following.

\begin{theorem}
\label{T:inequality}
Let $(X,F,\mm)$ be a forward-complete measured Finsler manifold (possibly with boundary) satisfying the $\CD(0,N)$
condition for some $N > 1$, 
having Euclidean volume growth.
Then
for every bounded Borel subset $E \subset X$ it holds
\begin{equation}\label{E:inequality}
\PP(E) \geq N \omega_{N}^{\frac{1}{N}} \AVR_{X}^{\frac{1}{N}} 
\mm(E)^{\frac{N-1}{N}}
\end{equation}
Moreover, inequality~\eqref{E:inequality} is sharp.
\end{theorem}

As we already said, the possible irreversibility of the manifolds does
not permit to simply apply Theorem~\ref{T:BalKri} to Finsler
manifolds.
In order to prove Theorem~\ref{T:inequality}, we will exploit the
Brunn--Minkowski inequality which holds true also for Finsler
manifolds.
This strategy was used by Balogh and
Krist\'{a}ly~\cite{BaKr22} for proving
Theorem~\ref{T:BalKri} and here we introduce no real new idea.
Indeed, in the light of~\cite{KristalyZhao22}, it seems that this
inequality holds true also for irreversible metric measure spaces;
here we confine our-self to the setting of measured Finsler manifolds.
To the best of the author knowledge, besides the Barky--Ledoux
inequality~\cite{Ohta18}, there is no other isoperimetric inequality
for measured Finsler manifolds that does not involve the reversibility
constant.

The main result of this paper concern the rigidity property of the
isoperimetric inequality~\eqref{E:inequality}.
To characterise its minima, we will have to additionally require the
reversibility constant $\Lambda_F$ to be finite.
This hypothesis is quite expected since in Finsler manifolds with
infinite reversibility certain pathological behaviors may arise (e.g.,\ the
Sobolev spaces may not be vector
spaces~\cite{KristalyRudas15,FarkasKristalyVarga15}).

\begin{theorem}
  \label{T:main1}
  Let $(X,F,\mm)$ be a forward-complete measured Finsler manifold (possibly with boundary) satisfying the $\CD(0,N)$
  condition for some $N > 1$, having reversibility constant
  $\Lambda_{F}<\infty$.
  %
  Assume that $\partial X$ is locally forward convex (see
  Definition~\ref{defn:convex-boundary}).
  Let
  $E \subset X$ be a bounded Borel set that saturates~\eqref{E:inequality}.

  Then there exists (a unique) $o\in X$ such that, up to a negligible
  set, $E=B^{+}(o,\rho)$, with
  $\rho=(\frac{\mm(E)}{\AVR_{X}\omega_N})^{\frac{1}{N}}$.
  Moreover, the measure $\mm$ can be represented by the disintegration formula
\begin{equation}\label{E:disintegrationintro}
\mm = \int_{\partial B^{+}(o,\rho)} \mm_{\alpha} 
\, \qq(d\alpha),
\qquad
\qq \in \mathcal{P}(\partial B^{+}(o,\rho)),
\qquad
\mm_{\alpha} \in \mathcal{M}_{+}(X),
\end{equation}
where $\mm_{\alpha}$ is concentrated on the geodesic ray from $o$ through $\alpha$ and $\mm_{\alpha}$  
can be identified (via the unitary speed parametrisation of the ray) with 
$N \omega_{N} \AVR_{X} t^{N-1} \mathcal{L}^{1}\llcorner_{[0,\infty)}$.
\end{theorem}

As an application of Theorem~\ref{T:main1}, we deduce the rigidity for
the weighted anisotropic isoperimetric problem in Euclidean cones, in
the irreversible case (the reversible case was already
investigated~\cite{CavallettiManini22a}).
We postpone this discussion to Section~\ref{Ss:euclidian}; now we
briefly present the proof strategy of Theorem~\ref{T:main1} and the
structure of the paper.

The classical approach for proving a rigidity results consists in
exploiting properties depending on the saturation of inequalities.
In this paper, following the line of~\cite{CavallettiManini22a}, we
adopt a different approach that starts from the proof of the
isoperimetric inequality for non-compact $\MCP(0,N)$
spaces~\cite{CavallettiManini22}.
In~\cite{CavallettiManini22a} it is used the localisation given by the
$L^1$-Optimal Transport problem between the renormalized measures
restricted on the set $E$ and $B_R$, a large ball of radius $R$
containing $E$.
The localization gives a family of one-dimensional, disjoint transport
rays together with a disintegration of the restriction to $B_R$ of
reference measure $\mm$.
At this point it is natural to analyze the well-known isoperimetric
inequality for the traces of $E$ along the one-dimensional transport
rays.
As Ohta pointed out~\cite{Ohta18}, differently from the reversible
case, the irreversibility of the space introduces the reversibility
constant, obtaining a non-sharp inequality.
Indeed, if one tries to prove Theorem~\ref{T:inequality} using the
localization of $E$ inside $B_R$ and taking the limit as $R\to\infty$
(as it was done in~\cite[Theorem~4.3]{CavallettiManini22a}) one would
obtain a factor $\Lambda_{F}^{-1}$ in the lower bound.
However, quite surprisingly, when studying the equality case, this
factor will disappear.

In order to capture the equality it is therefore necessary to deal
with this limit procedure and to get rid of the reversibility
constant.
The intuition suggests that, if $E$ saturates
inequality~\eqref{E:inequality}, then for large values of $R$ the
one-dimensional traces should be almost optimal.
We intend the almost optimality in many respects: for example, the
length of each transport ray has to be almost optimal; the
disintegration measures has to have density $t^{N-1}$; the traces of set $E$ has to almost
coincide with the interval starting from the starting point of the ray
having as length the expected radius of ball saturating the
inequality.
This last observation will be the key-point for solving the issues
arising by the irreversibility.

Indeed, we will see that the transport rays naturally come with a
unitary vector field that ``points outward'' from the set $E$, and
that, in the parametrization of the rays, this vector field points ``to
the right''.
When one computes the perimeter in the transport ray, one must compute
the measure of the boundary of the trace of $E$; we divide the
boundary in two parts: the part with outward normal vector pointing
``to the right'' and ``to the left'', respectively.
For the former part one computes the measure as usual, while for the
latter part one has to take into account the Finsler structure (here
appears the reversibility constant).
At this point the almost rigidity of the traces of $E$ is used: the
fact that any trace of $E$ almost coincides with the interval
$[0,\rho]$ permits us to deduce that the part of the boundary
``pointing to the left'' contributes little to the perimeter, therefore we
can get rid of the reversibility constant.
%
%

Having in mind these estimates we take the limit as $R\to \infty$.
There is no general procedure for taking the limit of a
disintegration.
However, following the procedure first employed
in~\cite{CavallettiManini22a}, we will exploit the almost optimality
of the traces of $E$ and the densities deduced in
Section~\ref{S:one-dim}; these properties permit to obtain a well
behaved limit disintegration for the reference measure restricted to
the set $E$, as it described in
Corollary~\ref{cor:disintegration-classical}.

Finally using the properties of the disintegration, we will deduce
that the set $E$ is a ball and the disintegration of the measure in
the whole space (see Theorems~\ref{T:Ball}
and~\ref{th:disintegration-final}, respectively).

The paper is organized as follows.
Section~\ref{S:preliminaries} recalls a few facts on Finsler
manifolds, the Curvature-Dimension condition and the localization
technique.
In Section~\ref{S:main-inequality} we prove
Theorem~\ref{T:inequality}.
In Sections~\ref{sec:localization} and~\ref{S:dimensiononespace} we
localize the reference measure and the perimeter and we present the
one-dimensional reductions.
In Section~\ref{S:one-dim} the one-dimensional estimates are carried
out.
In Section~\ref{S:limit} we deal with the limiting procedure, while
Section~\ref{eq:is-a-ball} concludes proof of Theorem~\ref{T:main1}.
We added two appendixes to this paper, containing the proof of the
fact that the relative perimeter is a measure and that the perimeter
is the relaxation of the Minkowski content.

\subsection{Applications in the Euclidean setting}
\label{Ss:euclidian}

As a consequence of Theorem~\ref{T:main1}, we present a
characterization of minima for the weighted anisotropic isoperimetric
problem in Euclidean cones.

The setting is the following: let $\Sigma\subset \R^{n}$ be an open
convex cone with vertex at the origin; let $H :\R^{n}\to [0,\infty)$
be a \emph{gauge}, i.e.,\ a non-negative, convex and positively
$1$-homogeneous function; let $w:\overline{\Sigma}\to(0,\infty)$ be a
continuous weight for the Lebesgue measure.

If $E\subset \R^{n}$ is a set with smooth boundary, we define the
\emph{weighted anisotropic perimeter} relative to the cone $\Sigma$ as
\begin{equation*}
\PP_{w,H}(E;\Sigma) = \int_{\partial E} H(\nu(x))w (x) \, dS
,  
\end{equation*}
(here $\nu$ and $dS$ denote the unit outward normal vector and the
surface measure, respectively).
Under the assumptions that $w^\frac{1}{\alpha}$ is concave and $w$ is
positively $\alpha$-homogeneous, it has been
proven~\cite{MilmanRotem14,CabreRosOtonSerra16} the following sharp isoperimetric
inequality for the weighted anisotropic perimeter
\begin{equation}\label{E:weightedanistropic}
\frac{\PP_{w,H}(E;\Sigma)}{w(E\cap \Sigma)^{\frac{N-1}{N}}} \geq \frac{\PP_{w,H}(W;\Sigma)}{w(W\cap \Sigma)^{\frac{N-1}{N}}}, 
\end{equation}
where $N = n + \alpha$, $W$ is the Wulff shape associated to $H$, and
the expression $w(A)$ with $A \subset \R^{n}$ is a short-hand notation
$\int_A w\,dx$.

If we take $w = 1$, $\Sigma= \R^{n}$, and $H =\|\cdot \|_{2}$,
inequality~\eqref{E:weightedanistropic} becomes the classical sharp
isoperimetric inequality.

As observed in~\cite{CabreRosOtonSerra16}, Wulff balls $W$ centered at
the origin are always minimizers of~\eqref{E:weightedanistropic}.
However in~\cite{CabreRosOtonSerra16} the characterization of the
equality case, is not present.
Many efforts have been done for solving this problem.
We now briefly recall the known results.
For the unweighted case ($w=1$), Ciraolo et al.~\cite{CiraoloLi22}
solved the problem under the assumption of $H$ to be a uniformly
elliptic positive gauge (i.e. a not necessarily reversible norm).
The characterisation in weighted setting has been solved
in~\cite{CintiGlaudoPratelliRosOtonSerra22} but only in the isotropic
case ($H = \|\cdot \|_{2}$).
The author, together with Cavalletti, solved the weighted
problem~\cite{CavallettiManini22a}, with the assumption of $H$ to be a
norm (i.e.\ reversible) with strictly convex balls, knowing that the
isoperimetric set is bounded.
This paper improves~\cite{CavallettiManini22a} by dropping the
reversibility assumption.

As it was observed in~\cite{CabreRosOtonSerra16}, the assumption that 
$w^{1/\alpha}$  is concave has a natural interpretation as the
$\CD(0,N)$ condition, where $N = n + \alpha$.
To be precise, if $H$ is a gauge then its dual function $F$ is a
Finsler structure (with finite reversibility constant), provided that
$F$ is smooth outside the origin and $F^2$ is strictly convex
(this two requirements can be equivalently required for the gauge $H$).
One can associate to the triple $\Sigma$, $H$ and $w$ the measured Finsler manifold
$(\overline{\Sigma}, F, w \Leb^n )$.
One can check that $(\overline{\Sigma},F,w\Leb^d)$ satisfies the $\CD(0,d+\alpha)$
condition and that the unitary ball, if $w\in C^\infty$: in Chapter~10
of~\cite{Ohta21} and in~\cite{Ohta09a} this is done in the case
$\Sigma=\R^n$; clearly the proof extends to the case of convex
subsets.
The fact that this manifold has finite reversibility, the convexity,
and the forward-completeness are trivial checks.
The perimeter associated to this space will indeed coincide with
$\PP_{w,H}$.
Moreover, by the homogeneity properties of $H$ and $w$, 
one can check that 
\begin{equation*}
\AVR_{\Sigma}
=
\lim_{R\to\infty}
\frac{w(B_{F}^+(0,R) \cap \Sigma  )}{\omega_{N}R^N}
=
\frac{w(W \cap \Sigma)}{\omega_{N}} >
0
.
\end{equation*}
Indeed, recall that the Wulff shape $W$ of $H$ is the unitary ball of
the Finsler structure $F$, hence the measure scales with power
$N=n+\alpha$.
Conversely, the perimeter of the rescaled Wulff shape turns out to be the
derivative w.r.t.\ the scaling factor of the measure, therefore the
perimeter of the Wulff shape is $N$ times its measure.
This consideration shows that~\eqref{E:weightedanistropic} follows
from \eqref{E:inequality}, thus we can apply Theorem \ref{T:main1} to
$(\Sigma,F,w \mathcal{L}^{n})$, obtaining the following result.

\begin{theorem}\label{T:Euclid-application}
  Let $\Sigma\subset \R^{n}$ be an open convex cone with vertex at the
  origin, and $H :\R^{n}\to [0,\infty)$ be a gauge.
  Assume $H$ to have strictly convex balls and to be smooth outside the origin.
Consider moreover the $\alpha$-homogeneous smooth weight
$w : \overline \Sigma \to [0,\infty)$ such that $w^{1/\alpha}$ is concave.

Then the equality in \eqref{E:weightedanistropic} is attained if and
only if $E=W\cap \Sigma$, where $W$ is a rescaled Wulff shape.
\end{theorem}

\noindent
\textbf{Acknowledgement.} The author is thankful to Fabio Cavalletti,
Alexandru Krist\'{a}ly, and Shin-ichi Ohta for their comments to this
manuscript.
The authors are also grateful to the referees for reading the
manuscript and suggesting a number of improvements.

\section{Preliminaries}\label{S:preliminaries} 
In this section we recall the main constructions needed in the
paper.
In Section~\ref{Ss:finsler} we review the geometry of measured Finsler
manifolds; in Section~\ref{Ss:isoperimetric} the perimeter in
measured Finsler
manifolds; in Section~\ref{s:W2} the Curvature-Dimension condition;
finally in Section~\ref{S:needle} the localization method.

\subsection{Finsler geometry}
\label{Ss:finsler}

We quickly recall the basic notions regarding measured Finsler manifolds.
The reader should refer to~\cite{Ohta21} for more details.
We adopt the convention that a manifold may have a boundary, unless
otherwise stated.
We require the boundary to be Lipschitz.

\begin{definition}
  \label{defn:finsler}
  Let $X$ be a connected $n$-dimensional manifold.
  We say that a function $F:TX\to [0,\infty)$ is a Finsler structure
  on $X$ if
  \begin{enumerate}
  \item (Regularity)
    $F$ is $C^{\infty}$ on $TX\backslash0$, where $0$ denotes the null
    section;
  \item (Positive $1$-homogeneity)
    For all $c>0$, $v\in TX$, it holds that $F(cv)=cF(v)$;
  \item (Strong convexity)
    On each tangent space $T_{x}X$, the function $F^2$ is strictly
    convex.
  \end{enumerate}
\end{definition}

The reader should notice that in general $F(v)\neq F(-v)$.
This feature, known as irreversibility, is what precludes us from
applying the theory of m.m.s.'s.
We define the re\-ver\-si\-bi\-li\-ty constant of a Finsler structure as
\begin{equation}
  \Lambda_{X,F}:=
  \sup_{v\in TX: v\neq 0}
  \frac{F(v)}{F(-v)}
  \in [1,\infty]
  ,
\end{equation}
or, in other words, $\Lambda_{X,F}\in[1,\infty]$ is the least
constant $\Lambda_{X,F}\geq 1$
such that for all $v\in TX$, $F(v)\leq\Lambda_{X,F} F(-v)$.
Later we will restrict ourself to the family of Finsler structures with finite
reversibility.
If no confusion arises, we shall write $\Lambda_F=\Lambda_{X,F}$.
If $X$ is compact, then $\Lambda_{X,F}<\infty$.

We define the speed of a $C^{1}$ curve $\eta$ as $F(\dot{\eta})$.
The notion of speed naturally induces a length functional
\begin{equation}
  \mathrm{Length}(\eta)
  :=
  \int_{0}^{1} F(\dot{\eta}) \,dt
  ,
\end{equation}
and thus we have a natural notion of distance between two points given
by
\begin{equation}
  \sfd_{X,F}(x,y):=
  \inf_{\eta}\{\mathrm{Length}(\eta):
  \eta_0=x,\text{ and }\eta_1=y\}.
\end{equation}
Whenever no confusion arises, we shall write $\sfd=\sfd_{X,F}$.
The distance $\sfd$ satisfies the usual properties of a distance, with
the exception of the symmetry
\begin{equation}
  \sfd(x,y)
  \leq
  \sfd(x,z)
  +\sfd(z,y)
  ,
  \quad
  \forall x,y,z\in X,
  \quad
  \text{ and }
  \quad
  \sfd(x,y)=0
  \quad
  \Leftrightarrow
  \quad
  x=y
  .
\end{equation}

\begin{remark}
  \label{rmrk:reassurement}
  We reassure the reader on the fact that the lack of symmetry of the
  distance does not harm most of the classical theory of metric spaces.
  Indeed, one can build $g_1$ and $g_2$, two riemannian metric on
  $TX$, such that
  \begin{equation}
    \sqrt{g_1(v,v)}\leq F(v) \leq \sqrt{g_2(v,v)}
    ,
    \qquad
    \forall v\in TX
    .
  \end{equation}
  Such metrics can be built on local charts and
  then glued together using a partition of the unity.
  Furthermore, such metrics can be chosen so that $g_2\leq f g_1$, for
  some continuous function $f:X\to[1,\infty)$.

  Using these metrics one can reobtain many classical results for
  free.
  In particular, we will make use of the Ascoli--Arzel\`{a} theorem,
  the fact that locally Lipschitz functions (as will be later
  introduced) are differentiable almost everywhere, and that locally
  Lipschitz functions with compact support are globally Lipschitz.
\end{remark}

We define the forward and backward balls, respectively, as
\begin{align}
  &
    B^{+}(x,r):=
    \{y\in X: \sfd(x,y)<r\}
    ,
    \qquad
    B^{-}(x,r):=
    \{y\in M: \sfd(y,x)<r\}
    .
\end{align}
If $A\subset X$, we define its (forward) $\epsilon$-enlargement to
$B^{+}(A,\epsilon):=\bigcup_{x\in A} B^{+}(x,\epsilon)$.
We say that a set $A$ is forward (resp.\ backward) bounded, if for
some (hence any) $x_0\in X$, there exists $r>0$ such that
$A\subset B^+(x_0,r)$ (resp.\ $A\subset B^-(x_0,r)$).
A we say that a set is bounded if it both backward and forward
bounded.
We denote by $\diam A:=\sup_{x,y\in A} \sfd(x,y)$ the diameter of a
set; a set has finite diameter if and only if it is bounded.

\begin{definition}
  Let $(X,F)$ be a Finsler manifold, possibly with boundary.
  We say that it is forward-complete, if and only if, for all
  sequences $(x_k)_k\subset X$ satisfying the forward Cauchy condition
  \begin{equation}
    \forall \epsilon >0
    :
    \exists N>0
    :
    \forall n>m>N
    :
    \qquad
    \sfd(x_m,x_n)\leq \epsilon
  \end{equation}
  then $(x_k)_k$ is converging.
\end{definition}
In light of Hopf--Rinow theorem, forward-completeness of a Finsler
manifold is equivalent to the compactness of closed forward balls and
implies  that given two points there exists a minimal geodesic
(as defined in the next paragraph) joining these two points.
In case of manifolds without boundary, forward-completeness is
equivalent also to the definition of the exponential map on the whole
tangent bundle.

A curve $\gamma:[0,l]\to X$ is called minimal geodesic if it minimizes the
length and its speed is constant.
We point out that, if $\gamma:[0,l]\to X$ is a minimal geodesic, in general
the reverse curve $t\mapsto \gamma_{l-t}$ may fail to be a geodesic
due to the possible irreversibility of the manifold.
We will denote by $\Geo(X)$ the set of minimal geodesic with domain the
interval $[0,1]$.
Like in the reversible setting, if $\gamma\in\Geo(X)$ is a
minimal geodesic, it holds that
\begin{equation}
  \sfd(\gamma_t,\gamma_s)
  =
  (s-t)\,\sfd(\gamma_0,\gamma_1)
  ,
  \qquad
  \forall 0\leq t\leq s\leq 1
  ;
\end{equation}
in this case the condition $t\leq s$ cannot be lifted.
\begin{definition}
  \label{defn:convex-boundary}
  Let $(X,F)$ be a Finsler manifold with boundary.
  We say that $\partial X$ is locally forward convex, if and only if,
for all points $x,y\in X\backslash\partial X$,
and for all minimal geodesic $\gamma$ connecting $x$ to $y$, we have that
$\gamma$ does not touch the boundary.  
\end{definition}

Given a submanifold $Y\subset X$, we can identify the tangent bundle
$TY$ as a subset of $TX$ via the standard immersion.
With this notation, we can restrict the Finsler structure $F$ to $TY$;
clearly, $F|_{TY}$ satisfies
Definition~\ref{defn:finsler}.
Regarding the reversibility constant and the distance, one immediately
sees that
\begin{equation}
  \Lambda_{Y,F}
  \leq
  \Lambda_{X,F}
  ,
  \qquad
  \text{ and }
  \qquad
  \sfd_{X,F}(x,y)\leq
  \sfd_{Y,F}(x,y)
  ,
  \quad
  \forall x,y\in Y
  .
\end{equation}

We define the dual function $F^{*}:T^{*}X\to[0,\infty)$ as
\begin{equation}
  F^{*}(\omega)
  :=
  \sup \{
  \omega(v):
  v\in T_{x}X,\text{ and } F(v)\leq 1
  \}
  ,
  \qquad
  \text{ if }\omega\in T_{x}^{*}X
  .
\end{equation}
Notice that, while we have that $\omega(v)\leq
F^{*}(\omega) F(v)$, the ``reverse'' inequality may not hold:
$\omega(v)\ngeq -F^{*}(\omega) F(v)$.
We define the Legendre transform $\mathcal{L}:T^*_xX\to T_xX$ as
$\mathcal{L}(\omega)=v$, where $v\in T_xX$ is the unique vector such
that $F(v)=F^*(\omega)$ and $\omega(v)=F(v)^2$ (the uniqueness follows from
the fact that $F^2$ is smooth and strictly convex).
Given a differentiable function $f:X\to\R$, we define its gradient as
$\nabla f(x):=\mathcal{L}(df(x))$.
Please note that in general $\nabla(-f)\neq -\nabla f$.

We say that a function $f:X\to\R$ is $L$-Lipschitz (with $L\geq 0$),
if
\begin{equation}
  \label{eq:lipschitz-definition}
  -L\sfd(y,x)
  \leq
  f(x)
  -
  f(y)
  \leq
  L\sfd(x,y)
  ,
  \quad
  \forall x,y\in X
  .
  \footnote{Please note that  we have
    chosen a sign convention different from~\cite{Ohta18,Ohta21}.}
\end{equation}
We point out that the first inequality
in~\eqref{eq:lipschitz-definition} follows from the second by swapping
$x$ with $y$.
The family of $L$-Lipschitz functions is stable by pointwise limits;
the infimum or the supremum of $L$-Lipschitz functions is still
$L$-Lipschitz.
Moreover, by Ascoli--Arzel\`{a} theorem, the family of $L$-Lipschitz
functions forms a
compact set in the topology of local uniform convergence.
If $f$ is $L$-Lipschitz, then $-f$ is $(\Lambda_FL)$-Lipschitz.
Two examples of $1$-Lipschitz functions are given by
$f(x)=-\sfd(o,x)$ and $g(x)=\sfd(x,o)$, for some fixed $o$.

We define the slope of a locally Lipschitz function $f$ as
\begin{equation}
  \label{eq:slope}
  |\partial f|(x)
  :=
  \limsup_{y\to x}
  \frac{(f(x)-f(y))^{+}}{d(x,y)}
  .
\end{equation}
Obviously, if $f$ is $L$-Lipschitz, then $|\partial f|\leq L$.
If $Y\subset X$ is a submanifold, and $f:X\to \R$, then $|\partial_Y
f|\leq |\partial_X f|$ in $Y$, where these two expressions have the
meaning of the slope of $f$ as seen as a function defined in $Y$ and
$X$, respectively.
If $f$ is differentiable at $x\in X$, the slope can be computed as $|\partial
f|(x)=F^*(-df(x))=F(\nabla(-f)(x))$, hence for locally Lipschitz
functions $|\partial f|=F(\nabla(-f))$ almost everywhere.

Finally, we would like to endow a manifold with a measure.
Differently from the Riemannian case, there is no canonical measure
induced from the Finsler structure.
On the other hand the theory of m.m.s.'s does not require any strong
assumption on the reference measure and, a priori, this measure might have
nothing to do with the Hausdorff measure.
For this reason we will only require for the reference measure to have
a smooth density when expressed in coordinates.
We conclude this section with the definition of measured Finsler manifold.
\begin{definition}
  A triple $(X,F,\mm)$ is called measured Finsler manifold, provided that $X$
  is a connected differential manifold (possibly with boundary) $F$ is
  a Finsler structure on $X$ and $\mm$ is a positive smooth measure,
  i.e.,\ given $x_1,\dots,x_n$ local coordinates, we have that
  \begin{equation}
    \mm=f\, dx_1\dots dx_n
    ,
    \qquad
    \text{ with }
    f>0
    \text{ and }
    f\in C^\infty
    .
  \end{equation}
\end{definition}

\subsection{Perimeter}
\label{Ss:isoperimetric}

Following the line traced in~\cite{Mirandajr.03,Am02,AmDiGi17} we give
the definition of (relative) perimeter for measured Finsler manifold.
Given a Borel subset $E \subset X$ and $\Omega$ open, the perimeter of $E$
relative to $\Omega$ is denoted by $\mathsf{P}(E;\Omega)$ and is defined as
follows,
\begin{equation}
  \PP(E;\Omega)
  :=
  \inf\left\{
    \liminf_{n\to\infty}
    \int_{\Omega} |\partial u_n|\,d\mm
    ; u_n\in\Lip_{loc}(\Omega)
    , u_n\to\indicator_{E}
    \text{ in }
    L^1_{loc}(\Omega)
  \right\}
  .
  \label{eq:perimeter-defn}
\end{equation}
In the unweighted Riemannian setting, if $E$ has smooth boundary, it
is a standard fact that
$\PP(E;\Omega)=\haus^{n-1}(\Omega\cap \partial E)$.
We say that $E \subset X$ has finite perimeter in $X$ if $\PP(E;X) < \infty$.
We recall also a few elementary properties of the perimeter functions:
\begin{itemize}
\item[(a)] (locality) $\PP(E;\Omega) = \PP(F;\Omega)$, whenever $\mm((E\bigtriangleup F) \cap \Omega) = 0$;
\item[(b)] (lower semicontinuity)\ the map $E \mapsto \PP(E,\Omega)$ is l.s.c.\ with respect to the $L^{1}_{loc}(\Omega)$ convergence.
\end{itemize}
Please note that, due to the possible irreversibility of the Finsler
structure, in general the complementation property does not hold.
If $E$ is a set of finite perimeter, then the set function $A \to \PP(E;A)$ is the restriction to open sets of a finite
Borel measure $\PP(E;\cdot)$ in $X$ (see
Appendix~\ref{S:perimeter-measure}), defined by
\begin{equation*}
\PP(E;A)
: =
\inf \{ \PP(E;\Omega) : \Omega \supset A,\, \Omega \text{ open} \}.
\end{equation*}
Sometimes, for ease of notation, we will write $\PP(E)$ instead of $\PP(E;X)$.

Give a subset $E\subset X$, we define its (forward) Minkowski content
as
\begin{equation}
  \label{eq:minkowski}
  \mm^{+}(E)
  :=
  \liminf_{\epsilon\to 0^{+}}
  \frac{\mm(B^{+}(E,\epsilon))-\mm(E)}{\epsilon}
  .
\end{equation}
It can be shown that the perimeter is the l.s.c.\ relaxation of the
Minkowski content w.r.t.\ the $L^{1}$ distance of sets.
The proof of this fact can be found in Appendix~\ref{S:minkowski}.

\subsection{Wasserstein distance and the Curvature-Dimension condition}
\label{s:W2}
Given a forward-com\-ple\-te measured Finsler manifold $(X,F, \mm)$, by $\M^+(X)$ and $\ProbMeas(X)$ we
denote the space of non-negative Borel measures on $X$ and the space of probability measures respectively.
For $p\in[1,\infty)$ we will consider the space $\ProbMeas_p(X)\subset \ProbMeas(X)$
of the measures $\mu$ satisfying
\begin{equation}
  \int_{X} (\sfd(o,x)+\sfd(x,o))^{p}\,\mu(dx)
  ,
  \qquad
  \text{ for some (hence any) }o\in X,
\end{equation}
i.e.\ $\mu$ has finite $p$-th moment.
On the space $\ProbMeas_{p}(X)$ we define the $L^{p}$-Wasserstein distance $W_{p}$, by setting, for $\mu_0,\mu_1 \in \ProbMeas_{p}(X)$,
\begin{equation}\label{eq:Wdef}
  W_p(\mu_0,\mu_1)^p
  :=
  \inf_{ \pi} \int_{X\times X} \sfd^p(x,y) \,\pi(dxdy)
  <\infty
  .
\end{equation}
The infimum is taken over all probability measure
$\pi \in \mathcal{P}(X \times X)$ with $\mu_0$ and $\mu_1$ as the
first and the second marginal, i.e.,\
$(P_{1})_{\sharp} \pi= \mu_{0}, (P_{2})_{\sharp} \pi= \mu_{1}$, where
$P_{i}, i=1,2$ denote the projection on the first (resp.\ second)
factor.
The infimum is attained and this minimizing problem is called
Monge--Kantorovich problem.

We call  a geodesic in the Wasserstein space $(\ProbMeas_p(X),W_p)$
any curve $\mu:[0,1]\to\ProbMeas_p$ such that
\begin{equation}
  W_p(\mu_t,\mu_s)=(s-t)W_p(\mu_0,\mu_1)
  ,
  \qquad
  \forall 0\leq t\leq s\leq 1
  .
\end{equation}
It can be shown that if $\mu_0$ and $\mu_1$ are absolutely continuous,
there exists a unique geodesic connecting $\mu_0$ to $\mu_1$.

\medskip

The $\CD(K,N)$ for condition for m.m.s.'s has been introduced in the
seminal works of Sturm \cite{St06, St06-2} and Lott--Villani
\cite{LoVi09}, and later investigated in the realm of measured Finsler
manifolds~\cite{Ohta09} (see also the survey~\cite{Ohta10}); here we
briefly recall only the basics in the case $K = 0$, $1 <N<\infty$.

We define the $N$-R\'{e}nyi entropy as
\begin{equation}
  S_N(\mu|\mm):=
  -\int_{X}\rho^{1-\frac{1}{N}}
  \,d\mm
  ,\qquad
  \text{ where }\mu=\rho\mm+\mu_{s}
  \text{ and }\mm\perp\mu_{s}
  .
\end{equation}

\begin{definition}[$\CD(0,N)$] \label{def:CDKN-ENB}
  Let $(X,F,\mm)$ be a forward-complete measured Finsler manifold and
  let $N\in[\dim X,\infty)$.
  We say that $(X,F,\mm)$ satisfies the $\CD(0,N)$ condition if and
  only if the $N'$-R\'{e}nyi entropy is convex along the geodesic of
  the Wasserstein space $\forall N'\geq N$, that is, for all couples
  of absolutely continuous curves $\mu_0,
  \mu_1\in\ProbMeas_2(X)$, it holds that
  \begin{equation}
    S_{N}(\mu_t|\mm)\leq
    (1-t)S_{N}(\mu_0|\mm)
    +tS_{N}(\mu_1|\mm),
  \end{equation}
  where  $(\mu_t)_{t\in[0,1]}$ is the unique
  geodesic connecting $\mu_0$ to $\mu_1$.
\end{definition}

If $(X,g,h\vol_g)$ is a weighted riemannian manifold, one can
introduce the $N$-Ricci tensor (as defined in~\eqref{eq:n-ricci}).
It was proven~\cite{St06,St06-2,LoVi09} that a weighted complete
riemannian manifold without boundary, satisfies the $\CD(0,N)$
condition if and only if $\Ric_N\geq0$.

Similarly to the riemannian case, a notion of weighted $N$-Ricci
curvature, still denoted by $\Ric_N$, has been introduced.
Here we do not give the definition of $\Ric_N$, for it is quite
lenghty and useless for our purposes.
Ohta~\cite{Ohta09a} proved that a measured Finsler manifold without boundary
satisfies the $\CD(0,N)$ condition if and only if $\Ric_N\geq 0$.
The possible presence of the boundary in the manifolds the present
paper deals with does not harm the results of this paper; indeed, we
rely only on the curvature dimension condition given by
Definition~\ref{def:CDKN-ENB} and never on $\Ric_N$.

\smallskip

Among many consequences of the $\CD(0,N)$ condition, two are of
our interest.
One is the Brunn--Minkowski inequality (see e.g.~\cite[Theorem~18.8]{Ohta21}).
Given two measurable subsets $A$ and $B$ of a $\CD(0,N)$ measured Finsler manifold
$(X,F,\mm)$ and $t\in[0,1]$, we define
\begin{equation}
  \begin{aligned}
  Z_t(A,B)
  :=&
  \{ \gamma_t: \gamma
  \text{ is a minimal geodesic such that }
  \gamma_0\in A\text{ and } \gamma_1\in B
      \}
    \\
    =&
       \{
       z:\exists x\in A, y\in B:
       \sfd(x,z)=t\sfd(x,y)
       \text{ and }
       \sfd(z,y)=(1-t)\sfd(x,y)
       \}
  .
  \end{aligned}
\end{equation}
With this notation, we have the Brunn--Minkowski inequality
\begin{equation}
  \label{eq:brunn-minkowski}
  \mm(Z_t(A,B))^{\frac{1}{N}}
  \geq
  (1-t)\mm(A)^{\frac{1}{N}}
  +
  t\mm(B)^{\frac{1}{N}}
  ,
  \qquad
  t\in[0,1]
  .
\end{equation}
The other property we are interested in is the Bishop--Gromov
inequality that states
\begin{equation}
  \label{eq:bishop-gromov}
\frac{\mm(B^{+}(x,r))}{r^N}
\geq
\frac{\mm(B^{+}(x,R))}{R^N}
,
\qquad
\forall 0<r\leq R,
\end{equation}
for any fixed point $x\in X$.
This inequality guarantees that the definition of asymptotic volume
ratio (see~\eqref{eq:avr}) is well posed.


\subsection{Localization}
\label{S:needle}

The localization method, also known as needle decomposition, is now a well-established technique for
reducing high-dimensional problems to one-dimensional problems.

In the Euclidean setting it dates back to Payne and Weinberger \cite{PaWe60},
it has been later  developed and popularised by Gromov and
V.~Milman~\cite{GrMi87},  Lov\'asz--Simonovits~\cite{LoSi93}, and Kannan--Lovasz--Simonovits~\cite{KaLoSi95}. 
Klartag \cite{Kl17} reinterpreted the localization method as a measure disintegration adapted to 
$L^{1}$-Optimal-Transport, 
and extended it to weighted Riemannian manifolds satisfying $\CD(K,N)$. 
Cavalletti and Mondino \cite{CaMo17} have succeeded to 
generalise this technique to  essentially non-branching m.m.s.'s verifying the $\CD(K,N)$, condition 
with $N \in (1,\infty)$, and later Otha~\cite{Ohta18} developed this
method for the Finsler setting.
Here we only report the case $K = 0$.

In his work, Ohta considered manifolds without boundary.
However, his proof also work for manifolds with local forward convex
boundary.

Consider a measured Finsler manifold $(X,F,\mm)$ satisfying the $\CD(0,N)$
condition and a function $f\in L^{1}(\mm)$ with finite first moment
such that
\begin{equation}
  \label{eq:null-average}
  \int_{X}
  f\,d\mm
  =0
  .
\end{equation}
The function $f$ induces two absolutely continuous measures
$\mu_0=f^+\mm$ and $\mu_1=f^-\mm$.
The well-established theory of $L^1$-optimal transport~\cite{Villani09} specify that
the Monge--Kantorovich problem possess a strictly related dual
problem, the so-called Kantorovich--Rubinstein dual problem:
\begin{equation}
  \max_{\phi}
  \int_{X}f(x)\phi(x)\,\mm(dx)
  =
  \max_{\phi}
  \left\{
    \int_{X}\phi(x)\,\mu_0(dx)
    -
    \int_{X}\phi(x)\,\mu_1(dx)
  \right\}
  ,
\end{equation}
where the maximum is taken among all possible $1$-Lipschitz functions
$\phi$.
The problem clearly admits a (non-unique) solution $\phi$; we will
call $\phi$ \emph{Kantorovich potential} for the problem.
Using $\phi$ we can construct the set
\begin{equation*}
\Gamma : = \{ (x,y) \in X \times X  \colon \phi(x) - \phi(y) =
\sfd(x,y)\},
\footnote{Please, notice that we use a different sign convention
  from~\cite{Ohta18,Ohta21}.}
\end{equation*}
inducing a partial order relation.
The maximal chains of this order relation turns out to be the image of curves
of maximal slope for $\phi$ with unitary speed.
To be more precise, we introduce the concept of transport curve: we
say that a unitary speed geodesics $\gamma:\Dom(\gamma)\subset\R\to X$ is a
non-degenerate transport curve, if its domain has at least two points,
$\frac{d}{d t} \phi(\gamma(t))=-1$, and $\gamma$ cannot be extended to
a larger domain.

We distinguish three possible cases, according to the number of
non-degenerate transport curves passing through a given point $x\in X$.
\begin{itemize}
\item
  There is no non-degenerate transport curve passing through $x$.
  We denote by $\D$  the set of such points.
  The set $\D$ is generally large.
\item
  There is exactly $1$ non-degenerate transport curve passing through
  $x$.
  Such points form the so-called transport set that will be denoted by
  $\T$.
  A fundamental property of $\T$ is that the $f$ is constantly $0$
  a.e.\ in $X\backslash\T$.
\item There are $2$ or more non-degenerate transport curves passing
  through $x$.
  Such points are called branching points and the set that they form
  will be denoted by $\A$.
  The set $\A$ turns out to be negligible.
\end{itemize}
All these sets are $\sigma$-compact, hence Borel.
In the sequel, we will also refer to the sets of forward (resp.\ backward)
branching points, defined as
\begin{align*}
  &
  \A^{+}
  :=
    \{x\in \A:\exists y\neq x\text{ such that }(x,y)\in\Gamma\}
    ,
  \\
  &
  \A^{-}
  :=
    \{x\in \A:\exists y\neq x\text{ such that }(y,x)\in\Gamma\}
    .
\end{align*}

On the transport set, we define the relation $\relation$ given by
\begin{equation}
  \relation:=(\Gamma\cup\Gamma^{-1})\cap(\T\times\T)
  .
\end{equation}
By construction, $\relation$ is an equivalent relation on $\T$ and the
equivalence classes are precisely the images of the transport curves.
One can chose $Q\subset\T$ a Borel section of the equivalence
relation $\relation$ (this choice is possible as it was shown
in~\cite[Proposition 4.4]{BianchiniCavalletti13}).
Define the quotient map $\QQ:\T\to Q$ as $\QQ(x)=\alpha$, where
$\alpha$ is the unique element of $Q$ such that
$(x,\alpha)\in\relation$.
We shall denote by $(X_\alpha)_{\alpha\in Q}$, the equivalence classes
for relation $\relation$, and we will call them {\em transport rays}.

The existence of a measurable section permits to construct
$g : \dom(g) \subset  Q \times [0,+\infty) \to \T$,
a measurable pa\-ra\-me\-tri\-za\-tion of the transport rays.
Fix $\alpha\in Q$ and take $\gamma$, a transport curve, such that
$\inf(\Dom(\gamma))=0$.
Then define $g(\alpha,t):=\gamma_t$, whenever $t\in\Dom(\gamma)$.
We specify that this parametrization guarantees that
$f(g(\alpha,0))\geq0$.
By continuity of $g$ w.r.t.\ the variable $t$, we extend $g$, in order
to map also the end-points of the rays $X_\alpha$; the restriction of
$g$ to the set $\{(\alpha,t): t\in(0,|X_\alpha|)\}$ is injective,
where $|X_\alpha|:=\sup\{t:(\alpha,t)\in\Dom(g)\}$.
Notice that $|X_\alpha|$ is not the diameter of $X_\alpha$, for
it may happen that
$\sfd(g(\alpha,|X_\alpha|),g(\alpha,0))>|X_\alpha|$.

The transport rays naturally come with the structure of
one-dimensional oriented manifold, with the orientation given by
$\partial_t g(\alpha,t)$, the
velocity of the parametrization.
We endow $X_\alpha$ with the Finsler structure given by the
restriction of $F$ to $X_\alpha$; notice that
$F(\partial_t g(\alpha,t))=1$.
As we already pointed out, it holds that
\begin{equation}
  \sfd_{X,F}(x,y)
  \leq
  \sfd_{X_{\alpha},F}(x,y)
  ,
  \qquad
  \forall x,y\in X_{\alpha};
\end{equation}
if $(x,y)\in\Gamma$, the inequality above is saturated, hence
\begin{equation}
  \sfd(g(\alpha,t),g(\alpha,s))
  =s-t
  ,\qquad
  \forall 0\leq t\leq s\leq |X_\alpha|
  .
\end{equation}

Given a finite measure $\q\in\M^{+}(Q)$, such that
$\q\ll\QQ_{\#}(\mm\llcorner_{\T})$, the  Disintegration Theorem
applied to $(\T , \mathcal{B}(\T), \mm\llcorner_{\T})$, gives
a disintegration of $\mm\llcorner_{\T}$
consistent with the partition of $\T$ given by the equivalence
classes $\{ X_{\alpha}\}_{\alpha \in Q}$ of $\relation$:
\begin{equation}
  \label{eq:disintegration}
\mm\llcorner_{\T} = \int _{Q} \mm_{\alpha}\,\qq(d\alpha ).
\end{equation}
Note that such measure $\q$ can always be build, by taking the
push-forward via $\QQ$ of a suitable finite measure mutually
absolutely continuous w.r.t.\ $\mm_{\T}$.
We recall by disintegration, we mean a map $\mm:Q\times\B(X)\to\R$,
such that
\begin{enumerate}
\item For $\qq$-a.e. $\alpha \in Q$, $\mm_{\alpha}$ is concentrated on $X_{\alpha}$.
\item For all $B \in \mathcal{X}$ , the map $\alpha \mapsto \mm_{\alpha}(B)$ is 
$\qq$-measurable.
\item For all $B \in \B(X)$, $\mm(B) = \int_{Q} \mm_{\alpha}(B)\,\qq(d\alpha)$.
\end{enumerate}

\begin{remark}
We point out that the disintegration is unique for fixed $\q$.
That means that, if there is a family $(\tilde \mm_\alpha)_\alpha$
satisfying the conditions above, then for $\q$-a.e.\ $\alpha$,
$\mm_\alpha=\tilde\mm_\alpha$.
If we change $\q$ with a different measure $\widehat\q$, such that
$\widehat\q=\rho\q$, $\rho>0$, then the map
$\alpha\mapsto \rho(\alpha)^{-1}\mm_\alpha$ still satisfies the conditions
above, with $\widehat\q$ in place of $\q$.
\end{remark}

We endow the transport ray $X_\alpha$ with the measure $\mm_\alpha$,
making $(X_\alpha,F,\mm_\alpha)$ a one-dimensional oriented
measured Finsler manifold.

Differently from the reversible case, it might happen that the
transport rays fail to satisfy the $\CD(0,N)$ condition.
However, a bound from below on the Ricci curvature can be
given in a certain sense.
It can be proved that $\mm_\alpha=(g(\alpha,\,\cdot\,))_{\#}(h_\alpha
\Leb^{1}_{(0,|X_\alpha|)})$, for a certain non-negative function
$h_\alpha$.
The function $h_\alpha$ satisfies $(h_\alpha^{\frac{1}{N-1}})''\leq 0$
in the distributional sense, i.e., the function $h_\alpha^{\frac{1}{N-1}}$ is concave.
Here we can recognize the $\CD(0,N)$ for weighted riemannian
manifolds, namely that the space $([0,|X_\alpha|],|\,\cdot\,|,h_\alpha
\Leb^1_{[0,|X_\alpha|]})$ satisfies the $\CD(0,N)$ condition.
This fact leads us to the following definition.
\begin{definition}
  Let $(X,F,\mm)$ be a measured Finsler manifold diffeomorphic
  to an interval, endowed with an orientation given by a vector field
  $v$, such that $F(v)=1$.
  We say that $(X,F,\mm)$ satisfies the oriented $\CD(0,N)$ condition
  ($N>1$), if the following happens.
  There exists $g:\Dom(g)\subset\R\to X$  a parametrization of $X$ such that
  $\partial_t g(t)=v(g(t))$  and $h:\Dom(g)\to[0,\infty)$,  a function
  such that $g_{\#}(h\Leb^1)=\mm$ and $h^{\frac{1}{N-1}}$ is concave.
\end{definition}

With this definition, clearly holds that the transport rays satisfy
the oriented $\CD(0,N)$ condition.
For the reader used with the notion of $N$-Ricci curvature, we point
out that the oriented $\CD(0,N)$ condition is equivalent to the fact
$\Ric_{N}(\partial_t g(\alpha,t))\geq 0$.

Finally, we point out that, as a consequence of the properties of the
optimal transport, we can localize the constraint $\int_{X}
f\,d\mm=0$, i.e.\ it holds that $\int_{X}
f\,d\mm_\alpha=0$, for $\q$-a.e.\ $\alpha\in Q$.

We summarize this section in the following theorem.

\begin{theorem}\label{T:locMCP}
Let $(X,F,\mm)$ be a measured Finsler manifold satisfying
$\CD(0,N)$, for some $N\in (1,\infty)$.

Let $f \in L^1(\mm)$ with $\int_{X} f\,\mm = 0$ and
\begin{equation}
  \int_{X} (\sfd(o,x)+\sfd(x,o))|f(x)|\,\mm(dx) < \infty
  ,\qquad
  \text{ for some (hence any) }
  o \in X
  .
\end{equation}
Then there exists a measurable subset $\mathcal{T} \subset X$
(transport set), a family
$\{(X_{\alpha},F,\mm_\alpha)\}_{\alpha \in Q}$ of oriented
one-dimensional submanifolds of $X$ (transport rays), and a measurable
function $g:\Dom(g)\subset Q\times [0,\infty)$ such that the following
happens.

 The function $f$ is zero $\mm$-a.e.\ in $X\backslash\T$ and
 $\mm\llcorner_{\mathcal{T}}$ can be disintegrated in the following
 way,
 \begin{equation}
   \mm\llcorner_{\mathcal{T}}
   =
   \int_{Q} \mm_{\alpha}\,\qq(d\alpha)
   .
 \end{equation}
 Moreover, for $\qq$-a.e.\ $\alpha \in Q$, the transport ray
 $(X_{\alpha},F,\mm_\alpha)$ is parametrized by the unitary speed
 geodesic $g(\alpha,\,\cdot\,)$, it satisfies the oriented $\CD(0,N)$
 condition, and
 it holds that
 \begin{equation}
   \label{eq:localize-constraint}
   \int f\, d\mm_{\alpha} = 0,
 \end{equation}
 Furthermore, two distinct transport rays can only meet at their
 extremal points (having measure zero for $\mm_{\alpha}$).
\end{theorem}

\begin{remark}
  \label{R:disintegration}
  The construction of the needle decomposition depends only on the
  function $\phi$, rather than the function $f$.
  Indeed, given a $1$-Lipschitz function $\phi$ one can construct the needle
  decomposition and reobtain Theorem~\ref{T:locMCP}, without, of
  course, Equation~\eqref{eq:localize-constraint}, which now makes no
  sense.
\end{remark}


\section{Proof of the main inequality}
\label{S:main-inequality}
We devout this section in proving Theorem~\ref{T:inequality}.

\begin{proof}[Proof of Theorem~\ref{T:inequality}]
  We will first prove that
  \begin{equation}
    \mm^{+}(E)
    \geq
    N(\omega_N\AVR_X)^{\frac{1}{N}}
    \mm(E)^{1-\frac{1}{N}}
    ,
    \qquad
    \forall E\subset X
    \text{ bounded}.
  \end{equation}
  From the inequality above the thesis will immediately follow by
  Theorem~\ref{T:lsc-envelope-minkowski}.
  
  Fix $E\subset X$ bounded and $x_0\in E$; set $d=\diam E$.
  Fix $R>0$ so that $E\subset B^+(x_0,R)$.
  We claim that $Z_t(E,B^+(x_0,R))\subset B^+(E,t(d+R))$.
  Indeed, let $z\in Z_t(E,B^+(x_0,R))$, hence there exist $x\in E$ and
  $y\in B^+(x_0,R)$ so that $\sfd(x,z)=t\sfd(x,y)$.
  By triangular inequality we deduce that
  \begin{equation}
    \sfd(x,z)
    =
    t\sfd(x,y)
    \leq
    t(\sfd(x,x_0)+\sfd(x_0,y))
    \leq
    t(d+R)
    ,
  \end{equation}
  thus $z\in B^+(E,t(d+R))$, proving the claim.
  We are in position to compute the Minkowski content
  \begin{equation}
    \begin{aligned}
      \mm^+(E)
      &
      =
      \liminf_{\epsilon\to0}
      \frac{\mm(B^+(E,\epsilon))-\mm(E)}{\epsilon}
        =
        \liminf_{t\to0}
        \frac{\mm(B^+(E,t(d+R)))-\mm(E)}{t(d+R)}
      \\
      &
        \geq
        \liminf_{t\to0}
        \frac{\mm(Z_t(E,B^+(x_0,R))-\mm(E)}{t(d+R)}
      \\
      &
        \geq
        \liminf_{t\to0}
        \frac{((1-t)\mm(E)^{\frac{1}{N}} +
        t\mm(B^+(x_0,R))^{\frac{1}{N}})^{N}-\mm(E)}
        {t(d+R)}
      \\
      &
        \geq
        \liminf_{t\to0}
        \frac{\mm(E) + N\mm(E)^{1-\frac{1}{N}} t
        (\mm(B^+(x_0,R))^\frac{1}{N}-\mm(E)^{\frac{1}{N}})+o(t)
        -\mm(E)}
        {t(d+R)}
      \\
      &
        =N\mm(E)^{1-\frac{1}{N}}
        \frac{\mm(B^+(x_0,R))^\frac{1}{N}-\mm(E)^{\frac{1}{N}}}
        {d+R}.
    \end{aligned}
  \end{equation}
  By taking the limit as $R\to\infty$, recalling the definition of
  $\AVR_X$, we conclude.
\end{proof}

\section{Localization of the measure and the perimeter}
\label{sec:localization}
From now on we assume that every Finsler manifold is forward-complete,
that it has finite
reversibility constant, and
local forward convexity.
To prove Theorem~\ref{T:main1} we consider the isoperimetric problem
inside a ball with larger and larger radius.
In order to apply the needle decomposition given
by the Localization Theorem~\ref{T:locMCP}, one also needs in
principle the balls to be convex.
As in general balls fail to be convex, we will overcome this issue in
the following way.

Given a bounded set $E \subset X$ with $0< \mm(E) < \infty$, fix a point
 $x_{0} \in E$
 and  then consider $R > 0$ such that $E \subset B_{R}$
 (hereinafter we will adopt the notation $B_R:=B^{+}(x_0,R)$).
Consider then the following family of null-average
 functions,
 \begin{equation*}
f_{R} (x) = \chi_{E} - \frac{\mm(E)}{\mm(B_{R})}\chi_{B_{R}}.
 \end{equation*}
Clearly, $f_{R}$ falls in the hypothesis of Theorem \ref{T:locMCP}.
Thus we obtain a
measurable subset $\mathcal{T}_{R} \subset X$ (the transport set)
and a family $\{(X_{\alpha,R},F,\mm_{\alpha,R})\}_{\alpha \in Q_{R}}$
of transport rays,
so that the measure $\mm\llcorner_{\mathcal{T}_{R}}$ can be
disintegrated:
\begin{equation}\label{E:disintbasic}
\mm\llcorner_{\mathcal{T}_{R}}= \int_{Q_{R}} \mm_{\alpha,R}\,\qq_{R}(d\alpha),\qquad \qq_{R}(Q_{R})=\mm(\T_R),
\end{equation}
where $\mm_{\alpha,R}$ are probability densities supported on $X_{\alpha,R}$.
Let
$g_{R}(\alpha,\,\cdot\,):[0,|X_{\alpha,R}|]\to X_{\alpha,R}$ be the unit
speed parametrisation of the transport ray $X_{\alpha,R}$, in the
direction given by the natural orientation of the disintegration ray
$X_{\alpha,R}$.
With this notation, it holds
\begin{equation}
  \mm_{\alpha,R}
  =
  (g_R(\alpha,\,\cdot\,))_\#
  (h_{\alpha,R} \mathcal{L}^{1}\llcorner_{[0,|X_{\alpha,R}|]}),
\end{equation}
for some $\CD(0,N)$ density $h_{\alpha,R}$.
The localization of the zero mean implies that
(see~\eqref{eq:localize-constraint})
\begin{equation}\label{E:balancing}
\mm_{\alpha,R}(E) = \frac{\mm(E)}{\mm(B_{R})} \mm_{\alpha,R}(B_{R}), \qquad 
\qq_{R}\text{-a.e.} \ \alpha \in Q_{R}.
\end{equation}
Unfortunately, the presence of the factor $\mm_{\alpha,R}(B_{R})$ in
the r.h.s.\ of the equation does make the quantity $\mm_{\alpha,R}$
independent of $\alpha$, harming the localization approach.
To get rid of this factor we proceed as follows.

We define $T_{\alpha,R}$ to be the unique element of $[0,|X_{\alpha,R}|]$
such that 
\begin{equation}
  \mm_{\alpha,R}(g_R(\alpha,[0,T_{\alpha,R}]))
  =
  \int_{0}^{T_{\alpha,R}}
  h_{\alpha,R}(x)\,dx
  =
  \mm_{\alpha,R}(B_{R})
\end{equation}
The measurability in $\alpha$ of $\mm_{\alpha,R}$ implies the same
measurability for $T_{\alpha,R}$.

Notice that  $|X_{\alpha,R}| \leq R + \diam(E)$: 
since $g_R(\alpha,\,\cdot\,)$ is the unit speed parametrization of $X_{\alpha,R}$, then
\begin{equation}
   \sfd(g_R(\alpha,0),g_R(\alpha,|X_{\alpha,R}|))
   \leq
   \sfd(g_R(\alpha,0),x_{0}) + \sfd(x_{0},g_R(\alpha,|X_{\alpha,R}|)) 
   \leq
   \diam(E) + R
   ,
\end{equation}
and consequently we deduce $T_{\alpha,R}\leq R + \diam(E)$.
We restrict $\mm_{\alpha,R}$ to the ray
$\widehat{X}_{\alpha,R}:=g_R(\alpha,{[0,T_{\alpha,R}]})$, having the
disintegration formula
\begin{equation}\label{E:disintnormalized}
  \mm\llcorner_{\widehat{\mathcal{T}}_{R}}
  =
  \int_{Q_{R}} \widehat{\mm}_{\alpha,R}\, \widehat{\qq}_{R}(d\alpha), 
  \quad
  \widehat{\mm}_{\alpha,R} : =
  \frac{\mm_{\alpha,R}\llcorner_{\widehat{X}_{\alpha,R}}}{\mm_{\alpha,R}(B_{R})} 
  \in \mathcal{P}(X)
  ,
  \quad
  \widehat{\qq}_{R} = \mm_{\cdot,R}(B_{R})\, \qq_{R}
  ;
\end{equation}
where $\widehat{\mathcal{T}}_{R} := \cup_{\alpha \in Q_{R}}
\widehat{X}_{\alpha,R}$.
Using \eqref{E:disintbasic} and the fact that
$B_{R}\subset \mathcal{T}_{R}$, we get
$\widehat{\qq}_{R}(Q_{R}) = \mm(B_{R})$,

The disintegration~\eqref{E:disintnormalized} will be a useful localisation 
only if $(E \cap X_{\alpha,R}) \subset \widehat{X}_{\alpha,R}$; in
this case we have
\begin{equation*}
\widehat{\mm}_{\alpha,R}(E) = \frac{\mm(E)}{\mm(B_{R})}, 
\qquad 
\widehat{\qq}_{R}\text{-a.e.} \ \alpha \in Q_{R},
\end{equation*}
obtaining a localization constraint independent of $\alpha$.
To prove this inclusion we will impose that $E \subset
B_{R/(4\Lambda_F)}$.
Since $g_R(\alpha,\,\cdot\,) : [0,|X_{\alpha,R}|] \to X_{\alpha,R}$
has unitary speed, we notice that
\begin{equation*}
\sfd(x_{0},g_R(\alpha,t))
\leq
\sfd(x_{0},g_R(\alpha,0)) + \sfd(g_R(\alpha,0),g_R(\alpha,t))
\leq
\frac{R}{4\Lambda_F} + t
\leq
\frac{R}{2} + t
,
\end{equation*}
where in the second inequality we have used that each starting point
of the transport ray has to be inside $E\subset B_{R/(4\Lambda_F)}$,
being precisely where $f_{R} > 0$.
The inequality above yields $g_R(\alpha,t)\in B_{R}$ for all $t < R/2$,
hence $((g_R(\alpha,\,\cdot\,))^{-1}(B_{R}) \supset 
[0,\min\{R/2, |X_{\alpha,R}|\}]$, thus there are no ``holes'' inside $(g_R(\alpha,\,\cdot\,))^{-1}(B_{R})$
before $\min\{R/2, |X_{\alpha,R}|\}$, implying that
$|\widehat{X}_{\alpha,R}| \geq \min\{R/2, |X_{\alpha,R}|\}$.
Fix $x\in E\cap X_{\alpha,R}$ and let $t\in[0,|X_{\alpha,R}|]$ be such
that $x=g_{R}(\alpha,t)$.
It holds that
\begin{equation}
  t=
  \sfd(g_R(\alpha,0),x)
  \leq
\sfd(g_R(\alpha,0),x_0)
+
\sfd(x_0,x)
\leq
(\Lambda+1)\frac{R}{4\Lambda}
\leq
\frac{R}{2},
\end{equation}
where in the second inequality we used that $g_{R}(\alpha,0),x\in
E\subset B_{R/(4\Lambda)}$.
The inequality immediately implies
$(g_R(\alpha,\,\cdot\,))^{-1}(E) \subset
[0,\min\{R/2,|X_{\alpha,R}|\}]$, hence
$E\cap X_{R,\alpha}\subset \widehat{X}_{\alpha,R}$, as we desired.

\smallskip
We describe  explicitly  the measure $\widehat\q_R$ in term
of a push-forward via the quotient map $\QQ_R$ of the measure
$\mm\llcorner_E$
\begin{align*}
  \widehat\q_R(A)
  &
    =
    \int_{Q_R}
    \indicator_A(\alpha)\frac{\mm(B_R)}{\mm(E)}
    \widehat\mm_{\alpha,R}(E)
    \,\widehat{\q}_R(d\alpha)
  \\
  &
    =
    \int_{Q_R}
    \frac{\mm(B_R)}{\mm(E)}
    \widehat\mm_{\alpha,R}(E\cap\QQ_R^{-1}(A))
    \,\widehat{\q}_R(d\alpha)
    =
    \frac{\mm(B_R)}{\mm(E)}
    \mm(E\cap\QQ_R^{-1}(A)),
\end{align*}
hence
$\widehat{\q}_R= \frac{\mm(B_R)}{\mm(E)}(\QQ_R)_{\#}(\mm\llcorner_E)$.

\smallskip

We study now the relation between the perimeter and the disintegration
of the measure~\eqref{E:disintnormalized}.
Let $\Omega\subset X$ be an open set and consider the relative
perimeter $\PP(E;\Omega)$.
Let $u_n\in \Lip_{loc}(\Omega)$ be a sequence such that
$u_n\to \indicator_E$ in $L^1_{loc}(\Omega)$ and
$\lim_{n\to\infty}\int_\Omega |\partial u_n|\,d\mm=\PP(E;\Omega)$.
Using the Fatou Lemma, we can compute
\begin{align*}
  \PP(E;\Omega)
  &
    =
    \lim_{n\to\infty}
    \int_\Omega |\partial u_n|\,d\mm
    \geq
    \liminf_{n\to\infty}
    \int_{\Omega\cap \widehat{\mathcal{T}}_R} |\partial u_n|\,d\mm
  \\
  &
    =
    \liminf_{n\to\infty}
    \int_{Q_R}
    \int_{\Omega} |\partial u_n| \, \widehat\mm_{\alpha,R}(dx)
    \,\widehat\q_R(d\alpha)
  \\
  &
    \geq
    \int_{Q_R}
    \liminf_{n\to\infty}
    \int_{\Omega} |\partial u_n| \, \widehat\mm_{\alpha,R}(dx)
    \,\widehat\q_R(d\alpha)
  \\
  &
    \geq
    \int_{Q_R}
    \liminf_{n\to\infty}
    \int_{X_{\alpha,R}\cap\Omega} |\partial_{X_{R,\alpha}} u_n| \, \widehat\mm_{\alpha,R}(dx)
    \,\widehat\q_R(d\alpha)
  \\
  &
    \geq
    \int_{Q_R}
    \PP_{\widehat X_{\alpha,R}}(E;\Omega)
    \,\widehat\q_R(d\alpha)
    ,
\end{align*}
where $|\partial_{X_{\alpha,R}} u|$ denotes the slope of the
restriction of $u$ to the transport ray $\widehat{X}_{\alpha,R}$
and $\PP_{\widehat X_{\alpha,R}}$ the perimeter in the submanifold
$(\widehat X_{\alpha,R},F,\widehat\mm_{\alpha,R})$.

By arbitrariness of $\Omega$, we deduce the following disintegration inequality
\begin{equation}
  \PP(E;\,\cdot\,)
  \geq
  \int_{Q_R}
  \PP_{\widehat X_{\alpha,R}}(E;\,\cdot\,)
  \,\widehat\q_R(d\alpha)
  .
\end{equation}

Next proposition summarizes this construction.

\begin{proposition}\label{P:disintfinal}
  Let $(X,F,\mm)$ be a $\CD(0,N)$ measured Finsler manifold with
  $\Lambda_{F}<\infty$.
  Given any bounded set $E \subset X$ with $0< \mm(E) < \infty$, fix
  any point $x_{0} \in E$ and then fix $R > 0$ such that
  $E \subset B_{R/(4\Lambda_F)}(x_{0})$.

Then there exists a Borel set $\widehat{\mathcal{T}}_{R} \subset X$, with 
$E \subset \widehat{\mathcal{T}}_{R}$ and a disintegration formula 
\begin{align}
  \label{E:disintfinal}
  &
\mm\llcorner_{\widehat{\mathcal{T}}_{R}} 
= \int_{Q_{R}} \widehat{\mm}_{\alpha,R}\, \widehat{\qq}_{R}(d\alpha), \qquad 
\widehat{\mm}_{\alpha,R}(\widehat X_{\alpha,R}) = 1, \qquad \widehat{\qq}_{R}(Q_{R}) = \mm(B_{R}),
\end{align}
such that
\begin{align}
  &
  \label{E:disintfinal2}
    \widehat{\mm}_{\alpha,R}(E)
    =
    \frac{\mm(E)}{\mm(B_{R})},
    \quad
    \text{for $\widehat \qq_{R}$-a.e.\ }\alpha\in Q_{R}
    \quad\text{ and }\quad
    \widehat\q_R
    =
    \frac{\mm(B_R)}{\mm(E)}
    (\QQ_R)_\#(\mm\llcorner_E)
    ,
\end{align}
Moreover, the transport ray
$(\widehat X_{\alpha,R}, F,\widehat{\mm}_{\alpha,R})$ satisfies the
oriented $\CD(0,N)$ condition and $|X_\alpha|\leq R + \diam (E)$.
Furthermore, it holds true that
\begin{equation}
  \label{E:disintfinalper}
  \PP(E;\,\cdot\,)
  \geq
  \int_{Q_R}
  \PP_{\widehat X_{\alpha,R}}(E;\,\cdot\,)
  \,\widehat\q_R(d\alpha)
  .
\end{equation}
\end{proposition}

The rescaling introduced in Proposition \ref{P:disintfinal} will be crucially used to 
obtain non-trivial limit estimates as $R \to \infty$.

\section{One-dimensional analysis}\label{S:dimensiononespace}

Proposition \ref{P:disintfinal} is the first step to obtain from the optimality of a bounded set $E$ 
an almost optimality of $E \cap \widehat X_{\alpha,R}$. 
We now have to analyse in details the behaviour of the perimeter in
one-dimensional oriented measured Finsler manifolds.

\smallskip
We fix few notation and conventions.
A one-dimensional oriented measured Finsler manifold can be identified with the manifold
$(I,F,\mm)$,
where $I\subset\R$ is an interval.
Without loss of generality we assume that the orientation is given by
the coordinated vector field $\partial_t$ on $I$.
Since we are studying manifolds arising from the localization, we
shall consider only Finsler structures that satisfy $F(\partial_t)=1$.
Thus, it is clear that the Finsler structure is completely
determined by $F(-\partial_t)$; for this reason, with a slight abuse
of notation, we will denote by $F$, the real-valued function given by
$F(-\partial_t)$.
With this convention, the reversibility constant turns out to be
\begin{equation}
  \Lambda_{I,F}=\sup_{x\in
    I}\left\{\max\left\{F(x),\frac{1}{F(x)}\right\}\right\}
  .
\end{equation}
When the interval has finite diameter, we will always assume that
$I=[0,D]$.
Notice that $D$ in general is not the diameter, for it may happen that
$\sfd(D,0)>\sfd(0,D)=D$; however, it holds that $\diam (I,F) \leq
\Lambda_{I,F} D$.

If $(I, F, \mm)$ satisfies the oriented $\CD(0,N)$ condition, then it happens that $\mm$
is absolutely continuous w.r.t.\ the Lebesgue measure $\Leb^1$ and
\begin{equation}
  \label{eq:cd-forward}
  (h^{\frac{1}{N-1}})''\leq 0
  ,\quad
  \text{ in the sense of distributions, where }h=\frac{d\mm}{d\Leb^1}
  .
\end{equation}
We stress out that if $(I,F,h\Leb^1\llcorner_I)$ satisfies the oriented
$\CD(0,N)$ condition, then the reversible manifold
$(I,|\,\cdot\,|,h\Leb^1\llcorner_I)$ satisfies the $\CD(0,N)$ condition.
We will say that the function $h$ itself satisfies the $\CD(0,N)$
condition if~\eqref{eq:cd-forward} holds.

Given a function $h:I\to[0,\infty)$, we shall write
$\mm_h:= h\Leb^1\llcorner_{I}$.
If the interval $I$ is compact, we will assume also that
$\int_0^D h=1$, unless otherwise specified.
We also introduce the functions
$v_h:[0,D]\to[0,1]$ and $r_h:[0,1]\to[0,D]$ as
\begin{equation}\label{E:volumesNa}
v_{h}(r) : = \int_{0}^{r} h(s)\,ds, \qquad r_{h}(v) : = (v_{h})^{-1}(v);
\end{equation}
notice that from the $\CD(0,N)$ condition, $h >0$ over $(0,D)$ making
$v_{h}$ invertible and in turn the definition of $r_{h}$ well-posed.

\smallskip

We will denote by $\PP_{F,h}$ the perimeter in the measured Finsler manifold
$(I, F ,h\Leb^1\llcorner_{I})$.
If $E\subset[0,D]$ is a set of finite perimeter, then it can be
decomposed (up to a negligible set) in a family of disjoint intervals
\begin{equation}
  E=\bigcup_i (a_i,b_i),
\end{equation}
and the union is at most countable.
In this case, we have that the perimeter is given by the formula
\begin{equation}
  \PP_{F,h}(E)
  =
  \sum_{i:a_i\neq 0} F(a_i) h(a_i)
  +
  \sum_{i:b_i\neq D} h(b_i).
\end{equation}
From the equation above, we immediately deduce a lower
bound on the perimeter
\begin{equation}
  \PP_{F,h}(E)
  \geq
  \Lambda_{I,F}^{-1}
  \PP_{|\,\cdot\,|,h}(E)
  \label{eq:lower-bound-irreversible}
\end{equation}
%

\subsection{Isoperimetric profile function}

The isoperimetric inequality for $\CD(0,N)$ manifolds with bounded
diameter is given in terms of the isoperimetric problem in the
so-called model spaces.
Here we recall the basic notions.

For $N>1$, $D>0$, and, $\xi\geq0$, we consider the model space
$([0,D],|\,\cdot\,|,h_{N,D}(\xi,\,\cdot\,)\Leb^1)$,
where
\begin{equation}
  \label{eq:definition-model-density}
  h_{N,D}(\xi,x)
  :=
  \frac{N}{D^N}
  \,
  \frac{(x+\xi D)^{N-1}}{(\xi+1)^N-\xi^N}
  .
\end{equation}
For the model spaces, we can easily compute the functions
$v_{N,D}(\xi,\,\,\cdot\,):=v_{h_{N,D}(\xi,\cdot)}$ and
$r_{N,D}(\xi,\,\,\cdot\,):=r_{h_{N,D}(\xi,\cdot)}$
\begin{align}
  &
    v_{N,D}(\xi,r)
    =
    \frac{
    (r+\xi D)^N-(\xi D)^N
    }{
    D^N((1+\xi)^N-\xi^N)
    }
    ,
  \\
  \label{eq:model-ray}
  &
    r_{N,D}(\xi,v)
    =
    D
    \left(
    (
    v(1+\xi)^N
    +(1-v)\xi^N
    )^{\frac{1}{N}}
    -\xi
    \right).
\end{align}
The isoperimetric profile function for the model spaces is given by
the formula
\begin{equation}
  \begin{aligned}
  \cI_{N,D}(\xi,v)
      :=&\,
      h_{N,D}(\xi,r_{N,D}(\min\{v,1-v\}))
    \\
  =&\,
  \frac{N}{D}
  \frac{
    (\min\{v,1-v\} (\xi+1)^N + \max\{v,1-v\} \xi^N)^{\frac{N-1}{N}}
  }{
    (\xi+1)^N-\xi^N
  }
  .
  \end{aligned}
\end{equation}
%

\smallskip

The family of one-dimensional measured Finsler manifolds
satisfying the $\CD(0,N)$ condition and having $\Lambda_F=1$ coincides with the of family of
weighted Riemannian manifolds.
E.\ Milman~\cite{Mi15} gave an explicit lower bound for the perimeter of
subset of manifolds in this family with the additional constraint of
having diameter bounded by some constant $D>0$.
In other words, Milman proved that given $D\geq D'>0$ and $h:[0,D']\to\R$ a
$\CD(0,N)$ density, then for all $E\subset [0,D]$ it holds that
$\PP_{|\,\cdot\,|,h}(E)\geq\I_{N,D}(v)$, where
\begin{equation}\label{E:isoperi}
  \I_{N,D}(v):=\frac{N}{D}\inf_{\xi\geq 0}
  \frac{(\min\{v , 1-v\} (\xi+1)^N + \max\{v,1-v\} \xi^N)^{\frac{N-1}{N}}}
  {(\xi+1)^N-\xi^N}
  =
  \inf_{\xi\geq0}
  \cI_{N,D}(\xi,v).
\end{equation}
As immediate consequence, one obtains that if we drop the
reversibility hypothesis, the lower bound of the perimeter must be
divided by the reversibility constant.

\smallskip

The author and Cavalletti proved~\cite[Lemma~4.1]{CavallettiManini22a}
deduced the an expansion for the isoperimetric profile, as follows.

\begin{lemma}\label{lem:milman-estimate}
Fix $N>1$.
Then, we have the following estimate for $\I_{N,D}$
\begin{equation}
  \I_{N,D}(w)
  \geq
  \frac{N}{D}w^{1-\frac{1}{N}}(1-O(w^{\frac{1}{N}}))
  =
  \frac{N}{D}(w^{1-\frac{1}{N}}-O(w)),
  \qquad
  \text{ as }w\to0.
\end{equation}
\end{lemma}

The following corollary incorporates both the irreversible and reversible case.

\begin{corollary}
\label{cor:milman-isoperimetric-estimate}
Fix $N>1$.
Then for all $D\geq D'>0$ and for all one-dimensional oriented
measured Finsler manifolds
$([0,D'],F,h\Leb^1)$ satisfying the oriented $\CD(0,N)$ condition, it
holds that
\begin{equation}
  \label{eq:isoperimetri-1d-finsler}
  \begin{aligned}
  \PP_{F,h}(E)
  &\geq \frac{\I_{N,D}(\mm_h(E))}{\Lambda_F}
  \geq \frac{N}{\Lambda_FD'}\m_h(E)^{1-\frac{1}{N}} (1-O(\mm_{h}(E)
    ^{\frac{1}{N}})
    \\&
  \geq \frac{N}{\Lambda_F D}\m_h(E)^{1-\frac{1}{N}} (1-O(\mm_{h}(E) ^{\frac{1}{N}}),
  \end{aligned}
\end{equation}
for any Borel set $E\subset [0,D']$.
If $E$ is of the form $[0,r_h(v)]$, then it holds
\begin{equation}
  \label{eq:isoperimetric-1d}
  \begin{aligned}
    \PP_{F,h}([0,r_h(v)])
    =
    h(r_h(v)
  &
  \geq \frac{N}{D'}v^{1-\frac{1}{N}} (1-O(v
    ^{\frac{1}{N}}))
    \geq \frac{N}{D}v^{1-\frac{1}{N}} (1-O(v^{\frac{1}{N}}))
    .
  \end{aligned}
\end{equation}
\end{corollary}

\begin{remark}
  The lower bound in~\eqref{eq:isoperimetri-1d-finsler} is very
  rough for our purposes.
  If one attempted to prove the isoperimetric
  inequality~\eqref{E:inequality}, the inverse of the reversibility
  constant would appear in the lower bound.

  The only reason why the factor $\Lambda_{F}^{-1}$ appears
  in~\eqref{eq:isoperimetri-1d-finsler}, is that the part of the
  boundary where the external normal vector ``points to the left''
  might be non-empty.
  Indeed, if $E$ is of the form $[0,b]$, then
  $\PP_{F,h}(E)=\PP_{|\,\cdot\,|,h}(E)$.
  We will see that the part of the boundary ``pointing to the
  left'' contributes little to the perimeter.
\end{remark}

\smallskip

We give the definition of the residual of a set.
This object quantifies, in a way that will be detailed in
Section~\ref{S:one-dim}, how far away is a ray from the expected model
space.

\begin{definition}
  Let $D\geq D'>0$ and let $([0,D'],F,h\Leb^1)$ be a one-dimensional
  measured Finsler manifold satisfying the oriented $\CD(0,N)$ condition.
  If $E\subset [0,D']$ is Borel set, we define its $D$-residual as
  \begin{equation}
    \label{eq:residual-definition}
    \Res_{F,h}^D(E)
    :=
    \frac{D\PP_{F,h}(E)}{N(\mm_h(E))^{1-\frac{1}{N}}}-1
    .
  \end{equation}
 If $v\in(0,1/2)$, we define the $D$-residual of $v$ as
 \begin{equation}
   \label{eq:residual-reversible}
    \Res_h^D(v)
    :=
    \Res_{F,h}^D([0,r_h(v)])
    =
    \frac{Dh(r_h(v))}{Nv^{1-\frac{1}{N}}}-1
    .
  \end{equation}
\end{definition}

Notice that in the definition of $\Res_h^D(v)$ there is no dependence
on the Finsler structure $F$; indeed, the definition of $\Res_h^D(v)$
is given in terms of the perimeter of $[0,r_h(v)]$, and the perimeter
of this set in $[0,D']$ does not capture the possible irreversibility the
Finsler structure.
Using the residual, Inequality~\ref{eq:isoperimetri-1d-finsler}
can be restated as
\begin{equation}
  \label{eq:isoperimetric-inequality-residual}
  \Res_{F,h}^D(E)\geq \Lambda_{F}^{-1}-1- O(\mm_h(E)^{\frac{1}{N}}).
\end{equation}
On the other hand, whenever the set $E$ is of the form $[0,r]$ we
obtain a much refined estimate
\begin{equation}
  \label{eq:isoperimetric-inequality-residual-reversible}
  \Res_{h}^D(v)
  =
  \Res_{F,h}^D([0,r_h(v)])
  \geq - O(v^{\frac{1}{N}}).
\end{equation}

\subsection{One-dimensional reduction for the optimal region}

We are ready to apply the definition of residual to the disintegration rays.
In order to ease the notation, we let
$\PP_{\alpha,R}=\PP_{(\widehat X_{\alpha,R}, F,\widehat
  \mm_{\alpha,R})}$.
The measure $\widehat \mm_{\alpha,R}$ will be identified with the ray
map $g_R(\alpha,\cdot)$ to $h_{\alpha,R} \mathcal{L}^{1}$, thus we define
\begin{align*}
  &
    \Res_{\alpha,R}:=\Res_{F,h_{\alpha,R}}^{R+\diam(E)}(g(\alpha,\cdot)^{-1}(E \cap \widehat X_{\alpha,R})), \quad \textrm{for } \alpha\in
  Q_{R},
  \\
  &
\Res_{x,R}: =\Res_{\QQ_R(x),R}, \quad \textrm{for } x\in E.
\end{align*}
The good rays are those rays having small residual.
We quantify their abundance.

\begin{proposition}\label{P:goodrays2}
  Assume that $(X,F,\mm)$ is a $\CD(0,N)$ measured Finsler manifold, such that $\AVR_{X} > 0$.
If $E\subset X$ is a bounded set attaining the identity in the inequality 
\eqref{E:inequality},  then
\begin{equation}
  \label{eq:residual-is-O-R-N}
  \limsup_{R\to\infty}
  \frac{1}{\mm(B_R)}
    \int_{Q_{R}} \Res_{\alpha,R}\,\q_{R}(d\alpha)
    \leq 0
    .
\end{equation}
\end{proposition}

\begin{proof}
In order to check that the function $\alpha\to\Res_{\alpha,R}$ is
integrable, it is enough to check that $(\Res_{\alpha,R})^-$, is
integrable.
This last fact derives from the isoperimetric inequality
\begin{equation*}
\Res_{\alpha,R}\geq\Lambda_{F}^{-1}-1-O((\frac{\mm(E)}{\mm(B_R)})^{\frac{1}{N}})
\end{equation*}
as
stated in~\eqref{eq:isoperimetric-inequality-residual}.
We can now compute the integral
in~\eqref{eq:residual-is-O-R-N}
\begin{equation}
  \begin{aligned}
  \int_{Q_R}
  \Res_{\alpha,R}\,\widehat\q_R(d\alpha)
  &
  =
  \int_{Q_R}
  \left(
  \frac{(R+\diam(E))\PP_{\alpha,R}(E)}{N}
  \left(
  \frac{\mm(B_R)}{\mm(E)}
  \right)^{1-\frac{1}{N}}
  -1
  \right)
  \,\widehat\q_R(d\alpha)
  \\
  &
  =
  \frac{R+\diam(E)}{\mm(B_R)^{\frac{1}{N}-1}\,N\mm(E)^{1-\frac{1}{N}}}
  \,
  \int_{Q_R}
  \PP_{\alpha,R}(E)
  \,\widehat\q_R(d\alpha)-\mm(B_R)
  \\
  &
  \leq
  \frac{R+\diam(E)}{\mm(B_R)^{\frac{1}{N}-1}\,N\mm(E)^{1-\frac{1}{N}}}
  \,
  \PP(E)-\mm(B_R)
  \\
  &
  \leq
    \mm(B_R)
  \frac{R+\diam(E)}{\mm(B_R)^{\frac{1}{N}}}
  (\AVR_X\omega_N)^{\frac{1}{N}}-\mm(B_R),
  \end{aligned}
\end{equation}
yielding
\begin{equation}
  \begin{aligned}
\frac{1}{\mm(B_R)}
  \int_{Q_R}
  \Res_{\alpha,R}\,\q_R(d\alpha)
  \leq
  \frac{R+\diam(E)}{\mm(B_R)^{\frac{1}{N}}}
  (\AVR_X\omega_N)^{\frac{1}{N}}-1,
\end{aligned}
\end{equation}
and the r.h.s.\ goes to $0$, as $R\to\infty$.
\end{proof}

\begin{corollary}\label{cor:goodrays2}
  Let $(X,F,\mm)$ be a $\CD(0,N)$ measured Finsler manifold, having $\AVR_{X} > 0$.
Let $E\subset X$ be a set saturating the isoperimetric inequality
\eqref{E:inequality}, 
then it holds true that
\begin{equation}
  \label{eq:residual-is-infinitesimal}
  \limsup_{R\to\infty}
  \int_{E} \Res_{\QQ_R(x),R}\,\mm(dx)
  \leq 0.
\end{equation}
\end{corollary}

\begin{proof}
  A direct computation gives
  \begin{align*}
  \int_{E} \Res_{\QQ_R(x),R} \, \mm(dx)
    &
      =
      \int_{Q_R}
      \int_E
      \Res_{\QQ_R(x),R}
      \,
      \widehat\mm_{\alpha,R}(dx)
      \,
      \widehat\q_R(d\alpha)
    \\
    &
      =
      \int_{Q_R}
      \Res_{\alpha,R}
      \,
      \widehat\mm_{\alpha,R}(E)
      \,
      \widehat\q_R(d\alpha)
    \\
    &
      =
      \frac{\mm(E)}{\mm(B_R)}
    \int_{Q_{R}} \Res_{\alpha,R}\,\q_{R}(d\alpha)
      \to0.
      \qedhere
  \end{align*}
\end{proof}


\section{Analysis along the good rays}\label{S:one-dim}
The last theorem asserts (in a very weak sense) that the residual, in the limit for $R\to\infty$,
must be non-positive.
Moreover, the measure of the traces of $E$ is
$\frac{\mm(E)}{\mm(B_R)}$, hence infinitesimal.
For this reason, we now use the residual and the measure of the set to
control the density $h:[0,D']\to\R$, proving that in case of
small measure and
residual, $h$ is close to the model density $x\in[0,D]\mapsto
Nx^{N-1}/D$.
Similarly we prove that the traces of $E$ are closed to the optimal,
i.e.,\ a certain interval of the form $[0,r]$.

\begin{remark}
We will  extensively use the Landau's ``big-O'' and
``small-o'' notation.
If several variables appear, but only a few of them are converging,
either the ``big-O'' or ``small-o'' could in principle depend on the
non-converging variables.
However, this is not the case.

To be precise, in our setting, the converging variables will be
$w\to0$ and $\delta\to0$.
Conversely the ``free'' variables will be: 1) $D$, a bound on the
length of the ray; 2) $D'\in(0,D]$, the length of the ray; 3)
$([0,D'],F,h)$ a one-dimensional measured Finsler manifold satisfying the oriented
$\CD(0,N)$ condition (in practice, each transport ray); 4) $E\subset[0,D']$ a set with
measure $\mm_h(E)=w$ and residual $\Res_{F,h}^D(E)\leq \delta$.

The estimates we will prove are infinitesimal expansions as
$w\to 0$ and $\delta\to 0$
and whenever a ``big-O'' or
``small-O'' appears, it has to be understood that it is uniform
w.r.t.\ the ``free'' variable.
\end{remark}

\begin{remark}
  An important point to remark is the fact that we consider only the
  case when $E$ is ``on the left'', i.e., $E\subset [0,L]$, with the
  tacit understanding that $L\ll D'$.
  This is possible because the transport rays come from the Optimal
  Transport problem between the bounded set isoperimetric $E$ and the
  ball $B_R$
\end{remark}

\subsection{Almost rigidity of the set \texorpdfstring{$E$}{E} and of the
length of the ray}
\label{Ss:rigidity-non-convex}
We start considering the
special case when the set $E$ of the form $E=[0,r]$.
In this case the Finsler structure plays no role, for the outer normal
vector on the boundary of $E$ points to the right.
For this reason, we omit the proof of the following proposition,
because it is exactly what is proven in Propositions~5.3 and~5.4
of~\cite{CavallettiManini22a}.

\begin{proposition}\label{P:almost-rigidity-segment}
Fix $N>1$.
Then, for $w\to0$ and $\delta\to 0$ it holds that
\begin{align}
  &
    \label{eq:almost-rigidity-diameter-convex}
    D'
    \geq
    D
    (
    1-o(1)
    )
    ,
  \\&
  \label{eq:almost-rigidity-r-above}
  r_h(w)
  \leq
  D(w^{\frac{1}{N}}(1+o(1)))
  ,
  \\&
  \label{eq:almost-rigidity-r-below}
  r_h(w)
  \geq
  D(w^{\frac{1}{N}}(1+o(1)))
  ,
\end{align}
where $D\geq D'>0$ and $([0,D'],F,h\Leb^{1})$ is one-dimensional
measured Finsler
manifold satisfying the oriented $\CD(0,N)$ condition such that
$\Res_{h}^D(w)=\Res_{F,h}^D([0,r_h(w)])\leq
\delta$.
\end{proposition}


We now drop the assumption $E=[0,r]$.
Up to a negligible set, it holds that
$E=\bigcup_{i \in \N}(a_i,b_i)$, where the intervals $(a_i,b_i)$ are far away
from each other (i.e.\ $b_i<a_j$ or $b_j<a_i$, for $i\neq j$).
The boundedness of the original set of our isoperimetric problem,
implies that $E\subset [0,L]$, for some $L>0$.
Define $b(E):=\esssup E\leq L$.

In the next proposition we prove that $b(E)$ is in the essential
boundary of $E$.

\begin{lemma}\label{lem:increasing}
Fix $N>1$, $L>0$, and $\Lambda\geq 1$.
Then there exists two constants $\bar w>0$ and $\bar\delta>0$
(depending only on $N$, $L$, and $\Lambda$) such that the following happens.
For all $D\geq D'>0$ with
$D\geq 4L\Lambda$, for all $([0,D'],F,h\Leb^1)$ a one-dimensional
measured Finsler manifold satisfying the oriented $\CD(0,N)$ condition 
with $\Lambda_{F}\leq \Lambda$,
and for all $E\subset [0,L]$, such that $\mm_h(E)\leq\bar w$
and $\Res_{F,h}^D(E)\leq\bar\delta$,  there exists
$a\in[0,b(E))$ and an at-most-countable family of intervals
$((a_i,b_i))_i$ such that, up to a negligible set,
\begin{align}
&
E
  =
  \bigcup_{i}
  (a_i,b_i)
  \cup(a,b(E)),
\end{align}
with $a_i,b_i< a$, $\forall i$.

Moreover,  $h$ is strictly increasing on $[0,b(E)]$.
\end{lemma}
\begin{proof}
Taking into account the definition of residual and the isoperimetric inequality~\eqref{eq:isoperimetric-inequality-residual}, choosing $\bar\delta\leq 1$, we
can deduce that
\begin{equation}
  \frac{D'}{D}
  \geq
  \frac{1+\Res_{F,h}^{D'}(E)}{1+\Res_{F,h}^{D}(E)}
  \geq
  \frac{1+\Lambda_{F}^{-1}-1-O(w^{\frac{1}{N}})}{1+\bar\delta}
  \geq
  \frac{\Lambda^{-1}}{2}-O(w^{\frac{1}{N}}))
  .
\end{equation}
If we choose $\bar w$ small enough, taking into account the hypothesis
$D\geq 4L\Lambda$, we deduce $D'\geq 2L$

Since $E=\bigcup_{i} (a_i,b_i)$ (up to a negligible set), our aim is to prove that there
exists $j$ such that $a_i<a_j$, for all $i\neq j$.
In this case we set $a=a_j$.
Suppose on the contrary, that $\forall j,\, \exists i\neq j$ such that
$a_i> a_j$, hence there exists a sequence $(i_n)_n$, so that
$(a_{i_n})_n$ is increasing, thus converging to some $y\in(0,L]$.
Recalling that $F\geq\Lambda^{-1}$, we can compute the perimeter
\begin{equation*}
  \infty
  =
  \sum_{n\in\N}
  F(a_{i_n})h(a_{i_n})
  \leq
  \PP_{F,h}(E)
  =
  \frac{N}{D}(\mm_h(E))^{1-\frac{1}{N}}(1+\Res_{F,h}^D(E))
  <\infty,
\end{equation*}
which is a contradiction.

Finally, we prove that $h$ increases on $[0,b(E)]$.
In order to simplify the notation, let $b:=b(E)$.
Denote by
$t:=\lim_{z\searrow0}(h(b+z)^{\frac{1}{N-1}}-h(b)^{\frac{1}{N-1}})/z$
the right-derivative of $h^{\frac{1}{N-1}}$ in $b$ (whose existence is
guaranteed by concavity).
If $t>0$, then the concavity of $h^{\frac{1}{N-1}}$ yields that $h$ is
strictly increasing in $[0,b]$.
Suppose on the contrary that $t\leq 0$, then it holds that
\begin{equation}
  h(x)
  \leq
  h(b)\left(\frac{D'-x}{D'-b}\right)^{N-1},
  \quad
  \forall x\in[0,b],
  \quad\text{ and }\quad
  h(x)\leq h(b),
  \quad \forall x\in[b,D'].
\end{equation}
We integrate obtaining
\begin{equation}
  \label{eq:big-estimate-increasing}
  \begin{aligned}
    1
    &
    \leq
      \int_0^bh(b)\left(\frac{D'-x}{D'-b}\right)^{N-1}\,dx
    \\
    &
      \qquad
    +
    \int_b^{D'}
    h(b)\,dx
    =
    \frac{h(b)}{N}
    \left(
      \frac{D'^N-(D'-b)^N}{(D'-b)^{N-1}}
      +N(D'-b)
    \right)
    \\[2mm]&
    \leq
    \frac{\PP_{F,h}(E)}{N}
    \left(
      \frac{D'^N}{(D'-b)^{N-1}}
      +ND'
    \right)
    =
    \frac{\PP_{F,h}(E)D'}{N}
    \left(
      \left(1-\frac{b}{D'}\right)^{1-N}
      +N
    \right)
    \\[2mm]&
    =
    \frac{\PP_{F,h}(E)D'}{N}
    \left(
      1+(N-1)\frac{b}{D'}
      +
      o\left(\frac{b}{D'}\right)
      +N
    \right)
.
  \end{aligned}
\end{equation}
The first factor in the r.h.s.\ of the estimate above is controlled
just using the definition of residual
\begin{equation}
  \frac{\PP_{F,h}(E)D'}{N}
  \leq
  \frac{\PP_{F,h}(E)D}{N}
  =
  \mm_h(E)^{1-\frac{1}{N}}
  (1+\Res_{F,h}^D(E)),
\end{equation}
and, if $\mm_h(E)\to0$ and $\Res_{F,h}^D(E)$ is bounded, then the
term above goes to $0$.
Regarding the second factor, it sufficies to prove that $\frac{b}{D'}$
is bounded:
\begin{equation}
  \frac{b}{D'}
  \leq
  \frac{L}{D'}
  \leq
  \frac{L}{2L}
  =\frac{1}{2}.
\end{equation}
If we put together this last two estimates, we deduce that the r.h.s.\
of~\eqref{eq:big-estimate-increasing} is infinitesimal as
$\mm_h(E)\to0$ and $\Res_{F,h}^D(E)\to0$, obtaining a contradiction.
\end{proof}

This proposition guarantees the existence of a right-extremal
connected component of the set $E$; this component is precisely the
interval $(a,b(E))$.
We will denote by $a(E)$ the number $a$ given by 
Proposition~\ref{lem:increasing}.
Since our estimates are infinitesimal expansions in the limit as
$\mm_h(E)\to0$ and $\Res_{F,h}^D(E)\to0$, we will always assume that
$\mm_h(E)\leq\bar w$ and $\Res_{F,h}^D(E)\leq\bar\delta$, so that the
expression $a(E)$ makes sense.
For the same reason, we will always assume that $h$ is increasing in
the interval $[0,b(E)]$.

We now prove that this component $(a(E),b(E))$ tends to fill the set
$E$, that $b(E)$ converges as expected to
$D\mm_h(E)^{\frac{1}{N}}$, and that  the length of the ray tends to be
maximal.

\begin{proposition}\label{P:almost-rigidity-general}
Fix $N>1$, $L>0$, and $\Lambda\geq 1$.
Then, for $w\to0$ and $\delta\to 0$ it holds that
\begin{align}
  &
  \label{eq:almost-rigidity-diameter-non-convex}
  D'
  \geq
    D(1-o(1))
  \\&
  \label{eq:almost-rigidity-b-above}
  b(E)
  \leq
    Dw^{\frac{1}{N}}
    +
    Do(w^{\frac{1}{N}})
  \\&
  \label{eq:almost-rigidity-b-below}
  b(E)
  \geq
    Dw^{\frac{1}{N}}
    -
    Do(w^{\frac{1}{N}})
  \\&
  \label{eq:almost-rigidity-a}
  a(E)
  \leq
  Do(w^{\frac{1}{N}}),
\end{align}
where $D\geq 4L\Lambda$, $D'\in(0,D]$, $([0,D'],F,h\Leb^1)$ is a
one-dimensional measured Finsler manifold satisfying the oriented
$\CD(0,N)$ condition with $\Lambda_F\leq\Lambda$, and
the set $E\subset [0,L]$ satisfies $\mm_h(E)=w$ and $\Res_{F,h}^D(E)\leq \delta$.
\end{proposition}
\begin{proof}
  \,

\smallskip\noindent
{\bf Part 1 {\rm Inequality~\eqref{eq:almost-rigidity-diameter-non-convex}}.}\\%
Since $h$ is decreasing on $[0,b(E)]$, we have that $h(r_h(v))\leq
h(b(E)) \leq \PP_{F,h}(E)$, hence $\Res_h^D(v)\leq\Res_{F,h}^D(E)$.
The thesis follows from estimate~\eqref{eq:almost-rigidity-diameter-convex}.

\smallskip\noindent
{\bf Part 2 {\rm Inequality~\eqref{eq:almost-rigidity-b-below}}.}\\%
Since the density $h$ is strictly increasing on $[0,b(E)]$ and
$E\subset[0,b(E)]$ (up to a null measure set), it holds that
$r_h(w)\leq b(E)$ and

\begin{equation}
  \Res_h^D(w)
  =
  \frac{Dh(r_h(w))}{Nw^{1-\frac{1}{N}}}-1
  \leq
  \frac{Dh(b(E))}{Nw^{1-\frac{1}{N}}}-1
  \leq
  \frac{D\PP_{F,h}(E)}{Nw^{1-\frac{1}{N}}}-1
  =
  \Res_{F,h}^D(E)\leq \delta.
\end{equation}
Estimate~\eqref{eq:almost-rigidity-r-below} concludes this part
\begin{equation}
  D(w^{\frac{1}{N}}-o(w^\frac{1}{N}))
  \leq
  r_h(w)
  \leq
  b(E)
  .
\end{equation}

\smallskip
\noindent
{\bf Part 3 {\rm Inequality~\eqref{eq:almost-rigidity-a}}.}\\%
First we prove that $a(E)< r_h(w)$, for $w$ and $\delta$ small
enough.
Suppose on the contrary that $a(E)\geq r_h(w)$, implying that
$h(a(E))\geq h(r_h(w))$, hence
$\PP_{F,h}(E)\geq \Lambda^{-1} h(a(E))+h(b(E))\geq
(1+\Lambda^{-1})h(r_h(w))$.
We deduce that (compare
with~\eqref{eq:isoperimetric-inequality-residual-reversible})
\begin{align*}
  -O(w^{\frac{1}{N}})
  &
    \leq
    \Res_h^D(w)
    =
    \frac{D h(r_h(w))}{Nw^{1-\frac{1}{N}}} -1
    \leq
    \frac{D \PP_{F,h}(E)}{(1+\Lambda^{-1})Nw^{1-\frac{1}{N}}} -1
  \\
  &
    =
    \frac{1}{1+\Lambda^{-1}}(\Res_{F,h}^D(E)-\Lambda^{-1})
    \leq
    \frac{\delta-\Lambda^{-1}}{1+\Lambda^{-1}}.
\end{align*}
If we take the limit as $w\to0$ and $\delta\to0$ we obtain a contradiction.

Using the Bishop--Gromov inequality and the isoperimetric
inequality (respectively), we get
\begin{align}
  &
  h(a(E))
  \geq
  h(r_h(w))
  \left(
  \frac{a(E)}{r_h(w)}
  \right)^{N-1}
  \\&
  h(b(E))
    \geq
    h(r_h(w))
    \geq
  \frac{N}{D}w^{1-\frac{1}{N}}(1-O(w^{\frac{1}{N}}))
  .
\end{align}
We put together the inequalities above obtaining
\begin{equation*}
  \begin{aligned}
    \frac{N}{D}w^{1-\frac{1}{N}}
    (1+\Res_{F,h}^D(E))
    &
    =
    \PP_{F,h}(E)
    \geq
    h(b(E))+\Lambda^{-1}h(a(E))
    \\&
    \geq
    h(r_h(w))+\Lambda^{-1}h(a(E))
    \\&
    \geq
    h(r_h(w))
    \left(
      1
    +
    \Lambda^{-1}
      \left(
        \frac{a(E)}{r_h(w)}
      \right)
      ^{N-1}
    \right)
    \\&
    \geq
    \frac{N}{D}w^{1-\frac{1}{N}}(1-O(w^{\frac{1}{N}}))
    \left(
      1
      +
    \Lambda^{-1}
      \left(
        \frac{a(E)}{r_h(w)}
      \right)
      ^{N-1}
    \right),
  \end{aligned}
\end{equation*}
hence
\begin{equation}
  \label{eq:a-is-small}
  \begin{aligned}
    a(E)
    &
    \leq
      r_h(w)
      \Lambda^{\frac{1}{N-1}}
      \left(
      \frac{
        1+\Res_{F,h}^D(E)
      }{
        1+O(w^{\frac{1}{N}})
      }
      -1
    \right)
    ^{\frac{1}{N-1}}
      \\&
    \leq
    r_h(w)
      \Lambda^{\frac{1}{N-1}}
    \left(
      (1+\delta)(1-O(w^{\frac{1}{N}}))
      -1
    \right)
    ^{\frac{1}{N-1}}
    \leq
    r_h(w)\, o(1)
    \\&
    \leq
    Dw^{\frac{1}{N}}(1+o(1))
    o(1)
    =
    Do(w^{\frac{1}{N}}),
  \end{aligned}
\end{equation}
where the estimate~\eqref{eq:almost-rigidity-r-above} was taken
into account.

\smallskip\noindent
{\bf Part 4 {\rm Inequality~\eqref{eq:almost-rigidity-b-above}}.}\\%
Since
\begin{equation}
  \int_Eh=\int_0^{r_h(w)} h,
\end{equation}
we deduce (taking into account $a(E)\leq r_h(w)\leq b(E)$)
\begin{equation}
  \int_{E\cap[0,r_h(w)]} h
  +
  \int_{r_h(w)}^{b(E)} h
  =
  \int_{E\cap[0,r_h(w)]} h
  +
  \int_{[0,r_h(w)]\backslash E} h
  =
  \int_{E\cap[0,r_h(w)]} h
  +
  \int_{[0,a(E)]\backslash E} h,
\end{equation}
hence
\begin{equation}
  (b(E)-r_h(w))
  \,
  h(r_h(w))
  \leq
  \int_{r_h(w)}^{b(E)} h
  =
  \int_{[0,a(E)]\backslash E} h
  \leq
  \int_0^{a(E)} h
  \leq
  a(E)
  \,
  h(a(E)),
\end{equation}
yielding
\begin{equation}
  b(E)-r_h(w)
  \leq
  a(E)
  \,
  \frac{h(a(E))}{h(r_h(w))}
  \leq
  a(E).
\end{equation}
Combining the inequality above, the already-proven
estimate~\eqref{eq:almost-rigidity-b-below},
and the estimate~\eqref{eq:almost-rigidity-r-above}, we reach the conclusion.
\end{proof}

\subsection{Almost rigidity of the density \texorpdfstring{$h$}{h}}
In this section we prove that the density $h$ converges uniformly to
the density  $Nx^{N-1}/D^N$.
The bound from below is easy and follows from the Bishop--Gromov
inequality.

\begin{proposition}
  \label{P:rigidity-of-space-easy-part}
  Fix $N>1$, $L>0$, and $\Lambda\geq 1$.
Then, for $w\to0$ and $\delta\to 0$ it holds that
\begin{align}
  \label{eq:rigidity-of-h-below}
  &
    h(x)
    \geq
    \frac{N}{D^N}x^{N-1}(1-o(1)),
    \quad
    \text{ uniformly w.r.t.\ }
    x\in [0,b(E)],
\end{align}
where $D\geq 4L\Lambda$, $D'\in(0,D]$, $([0,D'],F,h\Leb^1)$ is a one
dimensional measured Finsler manifold satisfying the $\CD(0,N)$ condition, with $\Lambda_F\leq\Lambda$, and the set
$E\subset [0,L]$ satisfies $\mm_h(E)=w$ and $\Res_{F,h}^D(E)\leq \delta$.
\end{proposition}
\begin{proof}
Fix $x\in[0,b(E)]$.
The Bishop--Gromov inequality yields
\begin{equation}
  h(x)
  \geq
  h(b(E))
  \,
  \frac{x^{N-1}}{b(E)^{N-1}}
  \geq
  h(r_h(w))
  \,
  \frac{x^{N-1}}{b(E)^{N-1}}.
\end{equation}
The first factor is controlled using the isoperimetric
inequality~\eqref{eq:isoperimetric-1d}
\begin{equation}
  h(r_h(w))
  \geq
  \frac{N}{D}w^{1-\frac{1}{N}}(1-O(w^{\frac{1}{N}}))
  =
  \frac{N}{D}w^{1-\frac{1}{N}}(1-o(1))
  ,
\end{equation}
whereas the term $b(E)$ is controlled using
estimate~\eqref{eq:almost-rigidity-b-above}
\begin{equation*}
  b(E)
  \leq
  Dw^{\frac{1}{N}}(1+o(1)).
\end{equation*}
By combining these to estimates we reach the thesis
\end{proof}

The following corollary gives a lower boundary for the residual, under
the hypothesis that the (positive part of the) residual is bounded from above, improving
inequality~\eqref{eq:isoperimetric-inequality-residual}.

\begin{corollary}
  \label{C:self-improvement-residual}
  Fix $N>1$, $L>0$, and $\Lambda\geq 1$.
Then, for $w\to0$ and $\delta\to 0$ it holds that
\begin{align}
  \label{eq:self-improvement-residual-old}
  &
    \Res_{F,h}^{D}(E)
    \geq -o(1)
\end{align}
where $D\geq 4L\Lambda$, $D'\in(0,D]$, $([0,D'],F,h\Leb^1)$ is a
one-dimensional measured Finsler manifold satisfying the $\CD(0,N)$ condition,
with $\Lambda_F\leq\Lambda$, and the set
$E\subset [0,L]$ satisfies $\mm_h(E)=w$ and $\Res_{F,h}^D(E)\leq \delta$.
\end{corollary}
\begin{proof}
  By a direct computation, recalling
  estimates\eqref{eq:rigidity-of-h-below}
  and~\eqref{eq:almost-rigidity-b-below}, we obtain
  \begin{align*}
    \Res_{F,h}^D(E)
    &
      \geq
      \frac{D h(b(E))}{Nw^{1-\frac{1}{N}}}-1
      \geq
      \frac{b(E)^{N-1}(1-o(1))}{D^{N-1}w^{1-\frac{1}{N}}}-1
    \\
    &
      \geq
      \frac{(w^{\frac{1}{N}}(1-o(1)))^{N-1}}{w^{1-\frac{1}{N}}}-1
      \geq o(1)
      .
      \qedhere
  \end{align*}
\end{proof}
In order to prove an upper bound for the density, we present the
following, purely technical lemma.
\begin{lemma}
\label{lem:monotonia-f}
Fix $N>1$ and consider the function $f:[0,1)\times[0,\infty]\to\R$ given by
\begin{equation}
  f(t,\eta)=\frac{1+\eta-t^N}{1-t}.
\end{equation}
Define the function $g$ by
\begin{equation}
  \label{eq:definition-g-mononicity}
  g(\eta)=\sup\{t-s:f(t,0)\leq f(s,\eta)\}.
\end{equation}
Then $\lim_{\eta\to0} g(\eta)=0$.
\end{lemma}
\begin{proof}
The proof is by contradiction.
Suppose that there exists $\epsilon>0$ and three sequences in
$(\eta_n)_n$, $(t_n)_{n}$, and $(s_n)_{n}$, such that $\eta_n\to0$,
$f(t_n,0)\leq f(s_n,\eta_n)$, and $t_n-s_n>\epsilon$.
Up to a taking a sub-sequence, we can assume that $t_n\to t$ and
$s_n\to s$, hence $1\geq t\geq s+\epsilon$.
The functions $f(\cdot,\eta_n)$ converge to $f(\cdot,0)$,
uniformly in the interval $[0,1-\frac{\epsilon}{2}]$.
This implies $f(s_n,\eta_n)\to f(s,0)$, yielding $f(t,0)\leq
f(s,0)$.
Since $t\mapsto f(t,0)$ is strictly increasing, we obtain
$t\leq s\leq t-\epsilon$, which is a contradiction.
\end{proof}

We now obtain an upper bound for $h$ in the interval $[a(E),b(E)]$
going in the opposite direction of the Bishop--Gromov inequality.

\begin{proposition}
Fix $N>1$, $L>0$, and $\Lambda\geq 1$.
Then, for $w\to0$ and $\delta\to 0$, it holds that
\begin{equation}
  \label{eq:rigidity-of-h-above}
  h(x)
  \leq
  h(b(E))\left(\frac{x}{b(E)}+o(1)\right)^{N-1},
  \quad
  \text{ uniformly w.r.t. }
  x\in[a(E),b(E)],
\end{equation}
where $D\geq 4L\Lambda$, $D'\in(0,D]$, $([0,D'],F,h\Leb^1)$ is a
one-dimensional measured Finsler manifold satisfying the oriented $\CD(0,N)$
condition, with $\Lambda_F\leq\Lambda$, and the set $E\subset [0,L]$
satisfies $\mm_h(E)=w$ and $\Res_{F,h}^D(E)\leq \delta$.

\end{proposition}
\begin{proof}
Fix $x\in[a(E),b(E)]$.
In order to ease the notation, define
\begin{equation}
a:=a(E),
\quad
b:=b(E),
\quad
k:=h(x)^{\frac{1}{N-1}},
\quad
l:=h(b(E))^{\frac{1}{N-1}}.
\end{equation}
The concavity of $h^{\frac{1}{N-1}}$ yields
\begin{align}
  &
    h(y)\geq\left(\frac{y}{x}\right)^{N-1}k^{N-1},
    \quad
    \forall y\in[a,x],
  \\&
  h(y)\geq\left(l+(k-l)\frac{b-y}{b-x}\right)^{N-1},\quad\forall y\in[x,b].
\end{align}
If we integrate, we obtain
\begin{equation*}
  \begin{aligned}
    w
    &
    \geq
    \int_{a}^x\frac{y^{N-1}}{x^{N-1}}k^{N-1}
    \,
    dy
    +
    \int_x^{b}\left(l+(k-l)\frac{b-y}{b-x}\right)^{N-1}
    \,dy
    \\&
    =
    \frac{k^{N-1}\,(x^N-a^N)}{Nx^{N-1}}
    +
    \frac{b-x}{N}
    \,
    \frac{l^N-k^N}{l-k},
  \end{aligned}
\end{equation*}
yielding
\begin{equation*}
  \begin{aligned}
  \frac{
    1
    -
    \left(
      \frac{k}{l}
    \right)^N
  }{
    1-\frac{k}{l}}
  &
  \leq
  \frac{
    Nw
    -
    \frac{
      k^{N-1}(x^N-a^N)
    }{
      x^{N-1}
    }
  }{
    l^{N-1}(b-x)
  }
  =
  \frac{
    \frac{Nw}{bl^{N-1}}
    -
    \frac{k^{N-1}(x^N-a^N)}{b(lx)^{N-1}}
  }{
    1-\frac{x}{b}
  }
  \\&
  \leq
  \frac{
    \frac{Nw}{bl^{N-1}}
    -
    \frac{x^N-a^N}{b^N}
  }{
    1-\frac{x}{b}
  }
  =
  \frac{
    \frac{Nw}{bl^{N-1}}
    +
    \frac{a^N}{b^{N}}
    -
    \frac{x^N}{b^N}
  }{
    1-\frac{x}{b}
  }
  ,
  \end{aligned}
\end{equation*}
where in the last inequality we used the Bishop--Gromov inequality
written in the form
$\frac{k^{N-1}}{l^{N-1}}\geq\frac{x^{N-1}}{b^{N-1}}$.
We now estimate the terms $\frac{Nw}{bl^{N-1}}$ and
$\frac{a^N}{b^{N}}$.
Regarding the former, taking into
account~\eqref{eq:almost-rigidity-b-below} and the isoperimetric
inequality~\eqref{eq:isoperimetric-1d}, we deduce
\begin{align*}
  \frac{Nw}{bl^{N-1}}
  &
  =
  \frac{Nw}{b(E)\,h(b(E))}
  \leq
    \frac{Nw}{b(E)\,h(r_h(w))}
  \\
  &
  \leq
  \frac{
    Nw
  }{
    Dw^{\frac{1}{N}}(1-o(1))
    \,\,
    \frac{N}{D}w^{1-\frac{1}{N}}(1-O(w^{\frac{1}{N}})
  }
  =
  1+o(1).
\end{align*}
Conversely, we estimate the latter term
(recall~\eqref{eq:almost-rigidity-b-above}
and~\eqref{eq:almost-rigidity-a})
\begin{equation}
\frac{a^N}{b^{N}}
=
\frac{a(E)^N}{b(E)^{N}}
\leq
\frac
{
  D^No(w)
}{
  D^Nw(1-o(1))^N
}
=
o(1).
\end{equation}
Putting all the pieces together, we obtain
\begin{equation}
  f\left(\frac{k}{l},0\right)
  =
  \frac{1-\left(\frac{k}{l}\right)^N}{1-\frac{k}{l}}
  \leq
  \frac{
    \frac{Nw}{bl^{N-1}}
    +
    \frac{a^N}{b^{N}}
    -
    \frac{x^N}{b^N}
  }{
    1-\frac{x}{b}
  }
  \leq
  \frac{
    1+o(1)
    -
    \frac{x^N}{b^N}
  }{
    1-\frac{x}{b}
  }
  =f\left(\frac{x}{b},o(1)\right),
\end{equation}
where $f$ is the function of Lemma~\ref{lem:monotonia-f}.
Applying said Lemma we get
\begin{equation}
  \frac{k}{l}
  -
  \frac{x}{b}
  \leq
  g(o(1))
  =
  o(1).
\end{equation}
If we explicit the definitions of $k$, $l$, and $b$, it turns out that
the inequality above is precisely the thesis.
\end{proof}


\subsection{Rescaling the diameter and renormalizing the measure}
So far, we have obtained an estimate of the densities $h$ and the set $E$.
The presence of factor $\frac{1}{D^N}$ in the
estimate~\eqref{eq:rigidity-of-h-below} suggests
the need of a suitable rescaling to get a non-trivial limit estimate.
We rescale the space by $\frac{1}{b(E)}$ and renormalize the measure
by $\mm_h(E)$.

Fix $k>0$ and define the rescaling transformation $S_k(x)=x/k$.
Given a density $h:[0,D']\to\R$ and $E\subset [0,L]$, we
define
\begin{equation}
  \nu_{h,E}
  =
  (S_{b(E)})_\#
  \left(
    \frac{
      \mm_h\llcorner_E
    }{
      \mm_h(E)
    }
  \right)
  \in\ProbMeas([0,1]).
\end{equation}
Clearly $\nu_{h,E}\ll\Leb^1$, so we denote by $\tilde h_E:[0,1]\to\R$
the Radon--Nikodym derivative $\frac{d\nu_{h,e}}{d\Leb^1}$.
The density $\tilde h_E$ can be explicitly computed
\begin{equation}
  \label{eq:definition-of-tilde-h-E}
  \tilde h_E(t)
  =
  \indicator_E(b(E)t)
  \frac{b(E)}{\mm_h(E)}
  \,
  h(b(E)t).
\end{equation}
Notice that, since $E$ could be disconnected, the indicator function
in~\eqref{eq:definition-of-tilde-h-E} prevents
$\tilde h_E^{\frac{1}{N-1}}$ from being concave, i.e.,\ $\tilde h_E$
possibly fails the oriented $\CD(0,N)$ condition.
However, in the limit, the $\CD(0,N)$ condition reappears, as it is
explicated by the following proposition.

\begin{proposition}
  Fix $N>1$, $L>0$, and $\Lambda\geq 1$.
Then, for $w\to0$ and $\delta\to 0$ it holds that
\begin{equation}
  \norm{\tilde h_E- N t^{N-1}}_{L^\infty(0,1)}
  \leq
  o(1)
\end{equation}
where $D\geq 4L\Lambda$, $D'\in(0,D]$, $([0,D'],F,h\Leb^1)$ is a
one-dimensional measured Finsler manifold satisfying the $\CD(0,N)$ condition,
with $\Lambda_F\leq\Lambda$, and the set $E\subset [0,L]$
satisfies $\mm_h(E)=w$ and $\Res_{F,h}^D(E)\leq \delta$.
\end{proposition}
\begin{proof}
Fix $t\in[0,1]$.
The proof is divided in four parts.

\smallskip\noindent
{\bf Part 1 {\rm Estimate from below and $t>\frac{a(E)}{b(E)}$}.}\\%
Since $t>\frac{a(E)}{b(E)}$, then $t\,b(E)\in E$ (for a.e.\ $t$).
A direct computation, gives
\begin{equation}
  \begin{aligned}
    \tilde h_E(t)
    &
    =
    \frac{b(E)}{w}
    \,
    h(tb(E))
    \geq
    \frac{Nb(E)^N}{D^Nw}\,
    t^{N-1}(1-o(1))
    \\&
    \geq
    \frac{ND^Nw(1+o(1))^N}{D^Nw}\,
    t^{N-1}(1-o(1))
    =
    Nt^{N-1}
    -
 o(1),
\end{aligned}
\end{equation}
having used the estimate~\eqref{eq:rigidity-of-h-below}, with
$x=tb(E)$, in the first inequality
and~\eqref{eq:almost-rigidity-b-below} in the second inequality.

\smallskip\noindent
{\bf Part 2 {\rm Estimate from below and $t\leq\frac{a(E)}{b(E)}$}.}\\%
In this case it may happen that $t\,b(E)\notin E$, so the best
estimate from below is the non-negativity.
For this reason, here we exploit the fact that the interval
$[0,\frac{a(E)}{b(E)}]$ is ``short'' and that $t\leq \frac{a(E)}{b(E)}$.
A direct computation gives (recall~\eqref{eq:almost-rigidity-b-below}
and~\eqref{eq:almost-rigidity-a})
\begin{equation}
  \begin{aligned}
    \tilde h_E(t)
    &
    \geq
    0
    \geq
    Nt^{N-1}-Nt^{N-1}
    \geq
    Nt^{N-1}
    -N
    \frac{a(E)^{N-1}}{b(E)^{N-1}}
    \\&
    \geq
    Nt^{N-1}
    -N
    \frac{
      D^{N-1}o(w^{1-\frac{1}{N}})
    }{
      D^{N-1}w^{1-\frac{1}{N}}(1+o(1))^{N-1}
    }
    \geq
    Nt^{N-1}
    -o(1).
  \end{aligned}
\end{equation}
\smallskip\noindent
{\bf Part 3 {\rm Estimate from above and $t\geq\frac{a(E)}{b(E)}$}.}\\%
We use estimate~\eqref{eq:rigidity-of-h-above}, with $x=tb(E)$, deducing
\begin{equation}
  \begin{aligned}
    \tilde h_E(t)
    &
    =
    \frac{b(E)}{w}
    \,
    h(tb(E))
    \leq
    \frac{b(E)}{w}
    \,
    h(b(E))
    (t+o(1))^{N-1}
    \leq
    \frac{b(E)}{w}
    \,
    h(b(E))
    (t^{N-1}+o(1))
    \\&
    \leq
    \frac{
      D w^{\frac{1}{N}}(1+o(1))
    }{w}
    \,
    \PP_{F,h}(E)
    (t^{N-1}+o(1))
    \\&
    =
    \frac{
      D w^{\frac{1}{N}}(1+o(1))
    }{w}
    \,
    \frac{N}{D}
    w^{1-\frac{1}{N}}
    (1+\Res_{F,h}^D(E))
    (t^{N-1}+o(1))
    \\&
    \leq
    N
    (1+o(1))
    (1+\delta)
    (t^{N-1}+o(1))
    =
    Nt^{N-1}
    +o(1)
  \end{aligned}
\end{equation}
(in the second inequality we used the uniform continuity of
$t\in[0,1]\mapsto t^{N-1}$; in the third one,
estimate~\eqref{eq:almost-rigidity-b-above}).

\smallskip\noindent
{\bf Part 4 {\rm Estimate from above and $t\leq\frac{a(E)}{b(E)}$}.}\\%
Without loss of generality we can assume that $a(E)\in E$.
Using the previous part we compute
\begin{align*}
  \tilde h_E(t)
  &
    =
    b(E)
    \frac{
    \indicator_E(tb(E))
    }{
    \mm_h(E)}
    h(b(E)t)
    \leq
    \frac{
    b(E)
    }{
    \mm_h(E)}
    h(b(E)t)
    \leq
    \frac{
    b(E)
    }{
    \mm_h(E)}
    h(a(E))
  \\&
    =
    \tilde h_E\left(\frac{a(E)}{b(E)}\right),
  \leq
    N\left(
    \frac{a(E)}{b(E)}
    \right)^{N-1}
    +o(1)
  \leq
  o(1)
  \leq
  Nt^{N-1}
  +o(1).
\qedhere
\end{align*}
\end{proof}

The following theorem summerizes the contents of this section.
Notice that the function $\omega$ takes as argument the positive part
of the residual and not the residual itself.
\begin{theorem}\label{T:rigidity-1d}
  Fix $N>1$, $L>0$, and $\Lambda\geq 1$.
  Then there exists a function
  $\omega:(0,\infty)\times[0,\infty)\to\R$, infinitesimal in
  $0$, such that the following holds.
  For all $D\geq 4L\Lambda$, $D'\in(0,D)$, for all $([0,D'],F,h\Leb^1)$
  one-dimensional measured Finsler manifold satisfying the oriented $\CD(0,N)$
  condition with $\Lambda_F\leq\Lambda$, and for all $E\subset [0,L]$,
  it holds that
  \begin{align}
    &
      \label{eq:almost-rigidity-b}
      \left|
      b(E)-D\mm_h(E)^{\frac{1}{N}}
      \right|
      \leq
      D\mm_h(E)^{\frac{1}{N}}
      \omega(\mm_h(E),(\Res_{F,h}^D(E))^{+})
      ,
    \\[2mm]
    &
      \label{eq:almost-rigidity-h-tilde}
      \norm{\tilde{h}_E- N t^{N-1}}_{L^\infty}
      \leq
      \omega(\mm_h(E),(\Res_{F,h}^D(E))^{+})
      ,
    \\[2mm]
    &
      \label{eq:self-improvement-residual}
      \Res_{F,h}^D(E)
      \geq
      -\omega(\mm_h(E),(\Res_{F,h}^D(E))^{+})
      ,
  \end{align}
  where $b(E)=\esssup E$ and
  $\tilde h_E$ is the density of
  $\mm_h(E)^{-1}(S_{b(E)})_\#\mm_h\llcorner_E$, with
  $S_{b(E)}(x)=x/b(E)$.
\end{theorem}

\section{Passage to the limit as \texorpdfstring{$R\to\infty$}{R→∞}}
\label{S:limit}

We now go back to the studying the identity case of the isoperimetric
inequality.
Fix $E$ a bounded Borel with positive measure such that
\begin{equation}
  \PP(E)=N(\omega_N\AVR_X)^{\frac{1}{N}}\mm(E)^{1-\frac{1}{N}},
\end{equation}
where $(X,F,\mm)$ is a $\CD(0,N)$ measured Finsler manifold having $\AVR_X>0$.
We will use the notation introduced Section~\ref{sec:localization}.
Denote by $\phi_R$ the $1$-Lipschitz Kantorovich potential associated
to $f_{R}=\indicator_E-\frac{\mm(E)}{\mm(B_R)}\indicator_{B_R}$.
If we add a constant to $\phi_R$, we still get a Kantorovich
potential, so we can assume that the family $\phi_R$ is equibounded on
every bounded set.
The Ascoli--Arzel\`a theorem, together with a diagonal argument,
implies that that, up to subsequences, $\phi_R$ converges to a certain
$1$-Lipschitz function $\phi_\infty$, uniformly on every compact set.
%

We recall the disintegration given by Proposition~\ref{P:disintfinal}
\begin{align}
  \label{eq:dis-one-line}
  &
\mm\llcorner_{\widehat{\mathcal{T}}_{R}} 
    = \int_{Q_{R}} \widehat{\mm}_{\alpha,R}\, \widehat{\qq}_{R}(d\alpha),
    \quad
    \text{ and }
    \quad
  \PP(E;\,\cdot\,)
  \geq
  \int_{Q_R}
  \PP_{\widehat X_{\alpha,R}}(E;\,\cdot\,)
  \,\widehat\q_R(d\alpha)
    .
\end{align}

The effort of this section goes in the direction to
understand how the properties of the disintegration behave at the
limit, and to try to pass to the limit in the disintegration.
Throughout this section, we set
$\rho=(\frac{\mm(E)}{\omega_{N}\AVR_{X}})^{\frac{1}{N}}$.

Before going on, using the self-improvement estimate of the
residual~\eqref{eq:self-improvement-residual}, we prove the following
Proposition.
\begin{proposition}\label{cor:goodrays3}
Up to taking subsequences, it holds that
\begin{equation}
  \label{eq:residual-is-infinitesimal3}
  \lim_{R\to\infty}
  \Res_{\QQ_R(x),R}
  = 0
  ,
  \qquad
  \mm\llcorner_{E}\text{-a.e.}
  .
\end{equation}
\end{proposition}
\begin{proof}
  Corollary~\ref{cor:goodrays2} guarantees that
\begin{equation}
  \limsup_{R\to\infty}
  \int_{E} \Res_{\QQ_R(x),R}\,\mm(dx)
  \leq 0,
\end{equation}
Using estimate~\eqref{eq:self-improvement-residual}, we estimate the
negative part of the residual
\begin{equation}
  (\Res_{\QQ_R(x),R})^{-}
  \leq
  \omega  \left(\frac{\mm(E)}{\mm(B_R)},(\Res_{\QQ_R(x),R})^{+}\right)
  =
  \omega  \left(\frac{\mm(E)}{\mm(B_R)},0\right)
  ,
\end{equation}
where $\omega$ is a function, infinitesimal in $(0,0)$.
The $L^1$-norm of the residual is given by
\begin{equation}
  \norm{\Res_{\QQ_R(x),R}}_{L^1(E;\mm)}
  =
  2\int_E
  (\Res_{\QQ_R(x),R})^-\,d\mm
  +
  \int_E
  \Res_{\QQ_R(x),R}\,d\mm
  .
\end{equation}
Taking into account the previous inequality and, again,
Corollary~\ref{cor:goodrays2},
we deduce that $\Res_{\QQ_R(x),R}$, converges to $0$ in $L^1$.
By taking a subsequence, we
obtain~\eqref{eq:residual-is-infinitesimal3}.
\end{proof}

\subsection{Passage to the limit of the radius}
First of all we define the \emph{radius} function
$r_R:\overline{E}\to[0,\diam E]$.
Fix $x\in E\cap \widehat{\mathcal{T}}_R$ and let
$E_{x,R}:=(g_R(\QQ_R(x),\cdot))^{-1}(E)\subset[0,|\widehat{X}_{\QQ_R(x),R}|]$.
Define
\begin{align}
  \label{eq:definition-r}
  &
    r_R(x):=
    \begin{cases}
      \esssup E_{x,R}
      ,
      \quad
      &\text{ if } x\in E\cap\widehat\T_R,
      \\
      0
      ,
      \quad
      &\text{ otherwise.}
    \end{cases}
\end{align}
Notice that $r_{R}(x) = b(E_{x,E})$, where the notation $b(E)$
was introduced in Section~\ref{Ss:rigidity-non-convex}.

The radius function  is defined on $\overline E$ for two motivations:
we require a common domain not depending on $R$ and 
the domain must be compact.
\begin{remark}\label{rmrk:domain-radius}
The set $E\cap\widehat\T_R$ has full $\mm\llcorner_E$-measure in
$\overline E$, hence it does not really matter how $r_R$ is defined outside
$E\cap\widehat\T_R$.
This fact is relevant, because we will only take limits
in the $\mm\llcorner_E$-a.e.\ sense.
\end{remark}

The next proposition ensures that, in limit as $R\to\infty$, the
function $r_R$ converges to $\rho$, which is
precisely the radius that we expect.

\begin{proposition}
  \label{P:limit-of-function-r}
  Up to subsequences it holds true
  \begin{equation}
    \lim_{R\to\infty} r_R
    =\rho=
    \left(
      \frac{\mm(E)}{\omega_N \AVR_X}
    \right)^{\frac{1}{N}}, \qquad \mm\llcorner_E -\textrm{a.e.}.
  \end{equation}
\end{proposition}

\begin{proof}
  By Proposition~\ref{cor:goodrays3}, there exists a sequence $R_n$
  and a negligible subset $N\subset E$, such that
  $\lim_{n\to\infty} \Res_{\QQ_{R_n}(x),R_n}=0$,  for all $x\in E\backslash N$.

  Define $G:=\bigcap_n\widehat\T_{R_n}\backslash N$ and notice that
  $\mm(E\backslash G)=0$.
  Fix $n\in\N$ and $x\in G$ and let $\alpha:=\QQ_{R_n}(x)\in Q_{R_n}$.
  Clearly, it holds
  \begin{align*}
&  |r_{R_n}(x)-
\rho |
%
    \leq
  \left|r_{R_n}(x)-
    ({R_n}+\diam E)
    \left(
      \tfrac{\mm(E)}{\mm(B_{R_n})}
    \right)^\frac{1}{N}
        \right|
  +
  \left|
    ({R_n}+\diam E)
    \left(
      \tfrac{\mm(E)}{\mm(B_{R_n})}
    \right)^\frac{1}{N}
    -
      \rho
      \right|
      .
  \end{align*}

  The second term goes to $0$ by definition of $\AVR$, so we focus on
  the first term.
  Consider the ray $(\widehat X_{\alpha,{R_n}}, F,\widehat\mm_{\alpha,{R_n}})$.
  By definition, we have that
  \begin{equation}
    \Res_{h_{\alpha,{R_n}}}^{{R_n}+\diam E}(E_{x,{R_n}})
    =
    \Res_{\alpha,R_n}
  \end{equation}
  We can now use Theorem~\ref{T:rigidity-1d} (in
  particular estimate~\eqref{eq:almost-rigidity-b}), obtaining
  \begin{equation}
    \begin{aligned}
      &
      \left|
        r_{R_n}(x)
        -
        ({R_n}+\diam E)
        \left(
          \tfrac{\mm(E)}{\mm(B_{R_n})}
        \right)^{\frac{1}{N}}
      \right|
      =
      \left|
        r_{R_n}(x)
        -
        ({R_n}+\diam E)
        (
        \mm_{h_{\alpha,R_n}}(E_{x,{R_n}})
        )^{\frac{1}{N}}
      \right|
      \\
      &
      \quad\quad
      \leq
      ({R_n}+\diam E)
      \mm_{h_{\alpha,R_n}}(E)^{\frac{1}{N}}
      \omega(
        \mm_{h_{\alpha,R_n}}(E)
        ,
        (\Res_{F,h_{\alpha,{R_n}}}^{{R_n}+\diam E}(E_{x,{R_n}})^+
        )
        )
      \\
      &
      \quad\quad
      =
      ({R_n}+\diam E)
      \left(
        \frac{\mm(E)}{\mm(B_{R_n})}
      \right)^{\frac{1}{N}}
      \omega\left(
        \frac{\mm(E)}{\mm(B_{R_n})}
        ,
        (\Res_{\QQ_R(x),{R_n}})^+
      \right)
        .
    \end{aligned}
  \end{equation}
Taking the limit as $n\to\infty$, we conclude.
\end{proof}

\subsection{Passage to the limit of the rays}
Consider now a constant-speed parametrization of the rays inside the set $E$:
\begin{equation}
  \label{eq:definition-of-gamma}
  \gamma_s^{x,R}:=
  \begin{cases}
    g_R(\QQ_R(x),s\,r_R(x)),
    \quad
    &
    \text{ if } x\in E\cap\widehat\T_R,
    \\
    x,
    \quad
    \text{ otherwise,}
  \end{cases}
\end{equation}
where $x\in\overline E$ and $s\in[0,1]$.
Remark~\ref{rmrk:domain-radius} applies also to the map
$x\mapsto\gamma^{x,R}$.
A direct consequence of the definition of $\gamma^{x,R}$ and the
properties of the disintegration are
\begin{align}
  &
    \label{eq:gamma-along-rays}
    \sfd(\gamma_t^{x,R},\gamma_s^{x,R})
    =
    \phi_R(\gamma_t^{x,R})
    -
    \phi_R(\gamma_s^{x,R})
    ,
    \quad
    \forall \,0\leq t\leq s\leq 1,
    \text{ for $\mm$-a.e.\ }x\in E,
  \\
  &
    \label{eq:speed-of-gamma}
    \sfd(\gamma_0^{x,R},\gamma_1^{x,R})
    =r_R(x),
    \quad
    \text{ for $\mm$-a.e.\ }x\in E
    ,
  \\
  &
    \label{eq:x-belongs-to-gamma}
    x\in \gamma^{x,R}
    ,
    \quad
    \text{ for $\mm$-a.e.\ }x\in E
    .
\end{align}
Please notice the order of the quantifiers in~\eqref{eq:gamma-along-rays}:
said equation means that $\exists N\subset E$ negligible such that
$\forall t\leq s$, $\forall x\in E\backslash N$,
\eqref{eq:gamma-along-rays} holds true.
In equation~\eqref{eq:x-belongs-to-gamma}, the expression
$x\in\gamma^{x,R}$ means that $\exists t\in[0,1]$ such that
$x=\gamma^{x,R}_t$, or, equivalently,
$\min_{t\in[0,1]}\sfd(x,\gamma_t^{x,R})=0$.

In order to compute the limit behaviour of $\gamma^{x,R}$ as $R\to\infty$
we proceed as follows.
Define the set
$K:=\{\gamma\in\Geo(X): \gamma_0,\gamma_1\in\overline E\}$; this set is
compact by Ascoli--Arzel\'a Theorem.
Define the measure
(having mass $\mm(E)$)
\begin{equation}
  \tau_R:=
  (\mathrm{Id}\times \gamma^{\,\cdot\,,R})_\#\mm\llcorner_{E}
  \,\in \M(\overline E\times K)
  .
\end{equation}
The measures $\tau_R$ enjoy the following immediate properties,
\begin{align}
  &
    (P_{1})_\#\tau_R=\mm\llcorner_E,
    \quad
    \text{ and }
    \quad
    \gamma=\gamma^{x,R},
    \quad\text{ for $\tau_R$-a.e.\ }(x,\gamma)\in \overline E\times K.
\end{align}
The
properties~\eqref{eq:gamma-along-rays}--\eqref{eq:x-belongs-to-gamma}
can be restated using a more measure-theoretic language
\begin{align}
  &
    \label{eq:gamma-along-rays2}
    \begin{aligned}
      &
    \sfd(e_t(\gamma),e_s(\gamma))
    -
    \phi_R(e_t(\gamma))
    +
    \phi_R(e_s(\gamma))
    =0
    ,
      \\
      &
    \qquad
    \forall \,0\leq t\leq s\leq 1,
        \quad
    \text{ for $\tau_R$-a.e.\ }(x,\gamma)\in \overline E\times K,
    \end{aligned}
  \\
  &
    \label{eq:speed-of-gamma2}
    \sfd(e_0(\gamma),e_1(\gamma))
    -r_R(x)=0,
    \quad
    \text{ for $\tau_R$-a.e.\ }(x,\gamma)\in \overline E\times K
    ,
  \\
  &
    \label{eq:x-belongs-to-gamma2}
    x\in\gamma
    ,
    \quad
    \text{ for $\tau_R$-a.e.\ }(x,\gamma)\in \overline E\times K
\end{align}
Clearly, the family of measures $(\tau_R)_{R>0}$ is tight, thus, by
Prokhorov Theorem, we can extract a sub-sequence such that
$\tau_R\rightharpoonup\tau$ weakly, i.e.,
$\int_{\overline E\times K} \psi\,d\tau_R\to\int_{\overline E\times K}
\psi\,d\tau$, for all $\psi\in C_b(\overline E\times K)$.
The next proposition guarantees that the
properties~\eqref{eq:gamma-along-rays2}--\eqref{eq:x-belongs-to-gamma2}
pass to the limit as $R\to\infty$.
\begin{proposition}
  For $\tau$-a.e.\ $(x,\gamma)\in\overline E\times K$, it holds that
  \begin{align}
    \label{eq:gamma-along-rays-weak}
    &
      \sfd(e_t(\gamma),e_s(\gamma))
      =
      \phi_\infty(e_t(\gamma))
      -
      \phi_\infty(e_s(\gamma))
      ,
      \quad
      \forall\, 0\leq t\leq s\leq1,
    \\
    &
      \label{eq:speed-of-gamma-weak}
      \sfd(e_0(\gamma),e_1(\gamma))
      =
      \rho
      ,
    \\
    &
      \label{eq:x-belongs-to-gamma-weak}
      x\in\gamma
      .
\end{align}
\end{proposition}
\begin{proof}
Fix $t\leq s$ and integrate~\eqref{eq:gamma-along-rays2} in
$\overline E\times K$, obtaining
\begin{align*}
  0
  &
    =
    \int_{\overline E\times K}
    (
    \sfd(e_t(\gamma),e_s(\gamma))
    -
    \phi_R(e_t(\gamma))
    +
    \phi_R(e_s(\gamma))
    )\, \tau_R(dx\, d\gamma)
    =
  \\
  &
    \int_{\overline E\times K}
    L_{\phi_R}^{t,s}(\gamma)
    \, \tau_R(dx\, d\gamma)
    ,
\end{align*}
having set
$L_\psi^{t,s}(\gamma):= \sfd(e_t(\gamma),e_s(\gamma)) - \psi(e_t(\gamma)) +
\psi(e_s(\gamma))$.
The map $L_{\phi_R}^{t,s}:K\to\R$ is clearly continuous and converges uniformly
(recall that $\phi_R\to\phi_\infty$ uniformly on every compact) to
$L_{\phi_\infty}^{t,s}$.
Therefore, we can take the limit in the equation above obtaining
\begin{align*}
  0
  &
    =
    \int_{\overline E\times K} L_{\phi_\infty}^{t,s}(\gamma)
    \, \tau(dx\, d\gamma)
    =
  \\
  &
    \int_{\overline E\times K}
    (
    \sfd(e_t(\gamma),e_s(\gamma))
    -
    \phi_\infty(e_t(\gamma))
    +
    \phi_\infty(e_s(\gamma))
    )\, \tau(dx\, d\gamma).
\end{align*}
The $1$-lipschitzianity of $\phi_\infty$, yields
$L_{\phi_\infty}^{t,s}(\gamma)\geq0$, $\forall\gamma\in K$, hence
\begin{equation}
  \sfd(e_t(\gamma),e_s(\gamma))
  =
  \phi_\infty(e_t(\gamma))
  -
  \phi_\infty(e_s(\gamma))
  \quad
  \text{ for $\tau$-a.e.\ } (x,\gamma)\in \overline E\times K.
\end{equation}
In order to conclude, fix $P\subset[0,1]$ a countable dense subset,
and find a $\tau$-negligible set $N\subset\overline E\times K$ such that
the equation above is true outside $N$ for all $t\leq s$ in $P$.
We conclude by approximating $t$ and $s$ in $P$
obtaining~\eqref{eq:gamma-along-rays-weak}.

Now we prove~\eqref{eq:speed-of-gamma-weak}.
The idea is similar, but now we need to be more careful,
for the function $r_R$ fails to be continuous.
We  integrate Equation~\eqref{eq:speed-of-gamma2}
obtaining
\begin{align*}
  0=
  &
    \int_{\overline E\times X}
    |\sfd(e_0(\gamma),e_1(\gamma))-r_R(x)|
    \,\tau_R(dx\,d\gamma).
\end{align*}
Are in position to apply Lusin's and Egorov's theorems.
Fix $\epsilon>0$ and find a compact $M\subset E$, such that:
1) the restrictions $r_R|_{M}$ are continuous;
2) the restricted maps $r_R|_{M}$ converge uniformly to
$\rho$;
3) $\mm(E\backslash {M})\leq\epsilon$.
We now compute the limit
\begin{align*}
  0
  &
    =
    \lim_{R\to\infty}
    \int_{\overline E\times K}
    |\sfd(e_0(\gamma),e_1(\gamma))-r_R(x)|
    \,\tau_R(dx\,d\gamma)
  \\ &
    \geq
    \liminf_{R\to\infty}
    \int_{M\times K}
    |\sfd(e_0(\gamma),e_1(\gamma))-r_R(x)|
    \,\tau_R(dx\,d\gamma)
  \\
  &
    \geq
    \int_{M\times K}
    |
    \sfd(e_0(\gamma),e_1(\gamma))
    -
    \rho
    |
    \,\tau(dx\,d\gamma)
    \geq 0
,
\end{align*}
therefore
\begin{equation}
    \sfd(e_0(\gamma),e_1(\gamma))
    =
    \rho
  ,
  \quad
  \text{ for $\tau$-a.e.\ }
  (x,\gamma)\in M\times K
  .
\end{equation}
This means that the equation above holds true except for a set of
measure at most $\epsilon$, and by letting $\epsilon\to0$, we
conclude.

Finally we prove~\eqref{eq:x-belongs-to-gamma-weak}.
In this case, consider the continuous, non-negative function
$L(x,\gamma):=\inf_{t\in[0,1]}\sfd(x,e_t(\gamma))$.
Equation~\eqref{eq:x-belongs-to-gamma2} implies
\begin{equation}
  0=\int_{\overline E\times K} L(x,\gamma)\,\tau_R(dx\,d\gamma).
\end{equation}
The equation above passes to the limit as $R\to\infty$, so the
conclusion immediately follows.
\end{proof}

\subsection{Disintegration of the measure and the perimeter}

Having in mind the disintegration formula~\eqref{eq:dis-one-line}, we
define the map $\overline E\ni x\mapsto\mu_{x,R}\in\ProbMeas(\overline E)$ as
\begin{equation}
  \mu_{x,R}:=
  \begin{cases}
  \frac{\mm(B_R)}{\mm(E)}
  \,
  (\widehat{\mm}_{\QQ_R(x),R})\llcorner_E
  ,\quad
  &
  \text{ if }x\in E\cap\widehat\T_R,
  \\
  \delta_x
  ,\quad
  &
  \text{ otherwise.}
\end{cases}
\end{equation}
A direct computation
(recall~\eqref{E:disintfinal}--\eqref{E:disintfinal2}) gives
\begin{align*}
  \mm(A\cap E)
  &
    =
    \int_{Q_R}
    \widehat\mm_{\alpha,R}(A\cap E)
    \,\widehat\q_R(d\alpha)
    =
    \frac{\mm(B_R)}{\mm(E)}
    \int_{Q_R}
    \widehat\mm_{\alpha,R}(A\cap E)
    \,(\QQ_R)_\#(\mm\llcorner_E)(d\alpha)
  \\
  &
    =
    \frac{\mm(B_R)}{\mm(E)}
    \int_X
    \widehat\mm_{\QQ_R(x),R}(A\cap E)
    \,\mm\llcorner_E(dx)
    =
    \int_X
    \mu_{x,R}(A)
    \,\mm\llcorner_E(dx),
\end{align*}
thus the following disintegration formula holds,
\begin{equation}
  \label{eq:disintegration-mu}
  \mm\llcorner_E
  =
  \int_{\overline E}\mu_{x,R}\,\mm\llcorner_E(dx)
  .
\end{equation}
\begin{remark}
\label{rmrk:caveat-disintegration}
We briefly discuss the measurablity of the integrand
function in Eq.~\eqref{eq:disintegration-mu}.
It holds that the map $x\mapsto\mu_{x,R}(A)$ is
measurable and the formula~\eqref{eq:disintegration-mu} holds.
Indeed, the map $x\mapsto\mu_{x,R}(A)$ is (up to excluding the
a negligible set) the composition of the maps
$Q_R\ni\alpha\mapsto\frac{\mm(B_R)}{\mm(E)}\widehat\mm_{\alpha,R}(A\cap
E)$ and the projection $\QQ_R$.
The former map is $\widehat\q_R$-measurable, whereas
the map $\QQ_R$ is $\mm$-measurable, w.r.t.\ the
$\sigma$-algebra of $Q_R$, thus the composition is measurable.
\end{remark}

Since
$\widehat\mm_{\alpha,R}=(g_R(\alpha,\cdot))_\#(h_{\alpha,R}\Leb^1\llcorner_{[0,|\widehat
  X_{\alpha,R}|]})$,
we can compute explicitly the measure $\mu_{x,R}$
(recall that by~\eqref{eq:definition-r} $r_R(x)=\esssup E_{x,R}$,
for $\mm\llcorner_E$-a.e.\ $x$)
\begin{equation}
  \begin{aligned}
  \mu_{x,R}
  &
  =
  \frac{\mm(B_R)}{\mm(E)}
  (g_R(\QQ_R(x),\cdot))_\#
  \left(
    (g_R(\QQ_R(x),\cdot))^{-1}(E)
    h_{\QQ_R(x),R}\Leb^1\llcorner_{[0,r_R(r)]}
  \right)
  \\
  &
  =
  (g_R(\QQ_R(x),\cdot))_\#
  \left(
    \indicator_{E_{x,R}}
    \frac{\mm(B_R)}{\mm(E)}
    h_{\QQ_R(x),R}\Leb^1\llcorner_{[0,r_R(x)]}
  \right)
  \\
  &
  =
  (\gamma^{x,R})_\#(\tilde h_E^{x,R}\Leb^1\llcorner_{[0,1]}),
  \quad
  \text{ for $\mm\llcorner_E$-a.e.\ } x\in \overline E
  \end{aligned}
\end{equation}
where
\begin{equation}
  \tilde h_E^{x,R}(t)
  =
  \indicator_{E_{x,R}}(r_R(x) t)
  \,
  r_R(x)
  \frac{\mm(B_R)}{\mm(E)}
  \,
  h_{\QQ_R(x),R}(r_R(x)t).
\end{equation}

Having in mind~\eqref{E:disintfinalper}, we can perform a similar operation
for the perimeter.
Indeed, in the natural parametrization of the rays, if we consider
only the ``right extremal'' of $E_{x,R}$ and the fact that
$F(\partial_t)=1$, it holds that
\begin{equation}
  h_{R,\QQ_R(x)}(r_R(x))\delta_{r_R(x)}
  \leq
  \PP_{F,h_{R,\QQ_R(x)}}(E_{x,R};\,\cdot\,)
  .
\end{equation}
This observation, naturally leads to the definition
\begin{equation*}
  \begin{aligned}
    p_{x,R}
    :&
    =
    \begin{cases}
      \min
      \left\{
        \frac{\mm(B_R)}{\mm(E)}
        h_{R,\QQ_R(x)}(r_R(x))
        ,
      \frac{N}{\rho}
      \right\}
      \delta_{g_R(\QQ_R(x),r_R(x))}
      ,
      \quad
      &
      \text{ if }x\in E\cap\widehat\T_R
      ,
      \\
      \frac{N}{\rho}
      \delta_x
      ,
      \quad
      &
      \text{ if } x\in\overline E\backslash(E\cap\T_R).
    \end{cases}
  \end{aligned}
\end{equation*}
Using the maps $\gamma^{x,R}$ and $\tilde h_{x,R}$, we rewrite
$p_{x,R}$
\begin{equation}
  p_{x,R}
  =
    \begin{cases}
      \min
      \left\{
        \dfrac{
          \tilde h_{x,R}(1)
        }{
          \sfd(\gamma^{x,R}_0,\gamma^{x,R}_1)
        }
        ,
        \dfrac{N}{\rho}
      \right\}
      \delta_{\gamma^{x,R}_1}
      ,
      \quad
      &
      \text{ if }x\in E\cap\widehat\T_R
      ,
      \\
      \dfrac{N}{\rho}
      \delta_x
      ,
      \quad
      &
      \text{ if } x\in\overline E\backslash(E\cap\T_R).
    \end{cases}
\end{equation}
By definition of $p_{x,R}$ we have that
\begin{equation}
  \begin{aligned}
    p_{x,R}
  &
  \leq
  \frac{\mm(B_R)}{\mm(E)}
  \PP_{X_{R,\QQ_R(x)}}(E;\,\cdot\,),
  \quad
  \text{ for $\mm\llcorner_E$-a.e.\ }x\in \overline E,
  \end{aligned}
\end{equation}
deducing the following ``disintegration'' formula
(equations~\eqref{E:disintfinalper} and~\eqref{E:disintfinal2} are
taken into account),
\begin{equation}
  \label{eq:disintegration-perimeter-p}
  \begin{aligned}
    \PP(E;A)
    &
    \geq
    \int_{Q_R}  \PP_{X_{\alpha,R}}(E;A)\,\widehat\q_R(d\alpha)
    =
    \frac{\mm(B_R)}{\mm(E)}
    \int_{\overline E}  \PP_{X_{R,\QQ_R(x)}}(E;A)\,\mm\llcorner_E(dx)
    \\
    &
    \geq
    \int_{\overline E}
    \,p_{x,R}(A)
    \,\mm\llcorner_{E}(dx)
    ,
    \quad
    \forall A\subset \overline E \text{ Borel}.
  \end{aligned}
\end{equation}

Define now the compact set
$F:=e_{(0,1)}(K)=\{\gamma_t:\gamma\in K,t\in[0,1]\}$ and let
$S\subset\M^+(F)$ be the subset of the non-negative measures on $F$
with mass at most $N/\rho$.
The sets $\ProbMeas(F)$ and $S$ are naturally endowed with the weak topology
of measures.
Since $K$ and $F$ are compact Hausdorff spaces, the Riesz--Markov
Representation Theorem, implies that the weak topology on $\ProbMeas(F)$ and
$S$, coincides with the weak* topology induced by the duality against
continuous functions $C(K)$ and $C(F)$, respectively.
The weak* convergence can be metrized on bounded sets, if the primal
space is separable; here we chose as a metric
\begin{equation}
  \label{eq:distance-weak-convergence}
  d(\mu,\nu)
  =
  \sum_{k=1}^{\infty}
  \frac{1}{2^k \norm{f_k}_{\infty}}
  \left|
    \int_{X}
    f_k
    \, d\mu
    -
    \int_{X}
    f_k
    \, d\nu
  \right|
  ,
\end{equation}
where $\{f_k\}_{k}$ is a dense set in $C(X)$.
We endow the spaces $\ProbMeas(F)$ and $S$ with the distance defined
in~\eqref{eq:distance-weak-convergence}, making them two compact
metric spaces.

Define now the map $G_R:\overline E\times K\to\ProbMeas(F)\times S$, as
\begin{equation}
  G_R(x,\gamma)
  :=
  \left(
    \gamma_\#(\tilde h_E^{x,R}\Leb^1\llcorner_{[0,1]})
    ,
    \min
    \left\{
      \frac{\tilde h^E_{x,R}(1)}{\sfd(e_0(\gamma),e_1(\gamma))}
      ,
      \frac{N}{\rho}
    \right\}
    \delta_{e_1(\gamma)}
  \right)
  .
\end{equation}
Clearly, the function $G_R$ is measurable w.r.t.\ the variable $x$ and
continuous w.r.t.\ the variable $\gamma$.
Define the measure (having mass $\mm(E)$)
\begin{equation}
  \sigma_R:=
  (\mathrm{Id}\times G_R)_\# \tau_R
  \in
  \M^+(\overline E\times K\times \ProbMeas(F)\times S).
\end{equation}
In order to ease the notation, we set $Z=\overline E\times K\times
\ProbMeas(F)\times S$.
\begin{proposition}
The measure $\sigma_R$ enjoys the following properties,
\begin{align}
  &
    \label{eq:disintegration-measure-weak}
    \int_E \psi\, d\mm
    =
    \int_Z\int_E \psi(y)\,\mu(dy)\,\sigma_R(dx\,d\gamma\,d\mu\,dp),
    \quad
    \forall\psi\in C^0_b(\overline E)
    ,
  \\
  &
    \label{eq:disintegration-perimeter-weak}
    \int_{\overline E} \psi(y)\,\PP(E,dy)
    \geq
    \int_Z \int_{\overline E} \psi(y) \,p(dy)\,\sigma_R(dx\,d\gamma\,d\mu\,dp),
    \quad
    \forall\psi\in C^0_b(\overline E),\psi\geq0
    .
\end{align}
\end{proposition}
\begin{proof}
Fix a test function $\psi\in C_b^0(\overline E)$.
Notice that for $\sigma_R$-a.e.\ $(x,\gamma,\mu,p)\in Z$, we
have that $\mu=\mu_{x,R}$, because
\begin{equation}
  \mu
  =
  \gamma_\#(\tilde h_E^{x,R}\Leb^1\llcorner_{[0,1]})
  =
  (\gamma^{x,R})_\#(\tilde h_E^{x,R}\Leb^1\llcorner_{[0,1]})
  =
  \mu_{x,R},
  \quad
  \text{ for $\sigma_R$-a.e.\ }
  (x,\gamma,\mu,p)\in Z,
\end{equation}
and we used the fact that $\gamma=\gamma_{x,R}$ for $\tau_R$-a.e.\
$(x,\gamma)\in \overline E\times K$.
We conclude the proof of~\eqref{eq:disintegration-measure-weak} by a direct computation
\begin{align*}
  \int_E \psi\,d\mm
  &
  =
    \int_E \int_E \psi(y)\,\mu_{x,R}\,\mm(dx)
  =
    \int_Z \int_E \psi(y)\,\mu_{x,R}(dy)\,\sigma_R(dx\,d\gamma\,d\mu\,dp)
  \\
  &
    =
    \int_Z \int_E \psi(y)\,\mu(dy)\,\sigma_R(dx\,d\gamma\,d\mu\,dp)
    .
\end{align*}

Now fix an open set $\Omega\subset X$ and compute
having in mind~\eqref{eq:disintegration-perimeter-p}
\begin{align*}
  \PP(E;\Omega)
  &
    \geq
    \int_{E}
    \min\left\{
    \frac{\tilde h^E_{x,R}(1)}{\sfd(\gamma_0^{x,R},\gamma_1^{x,R})}
    ,
    \frac{N}{\rho}
    \right\}
    \delta_{\gamma_1^{x,R}}(\Omega)
    \,d\mm(dx)
  \\
  &
    =
    \int_{Z}
    \min\left\{
    \frac{\tilde h^E_{x,R}(1)}{\sfd(e_0(\gamma^{x,R}),e_1(\gamma^{x,R}))}
    ,
    \frac{N}{\rho}
    \right\}
    \delta_{e_1(\gamma^{x,R})}(\Omega)
    \,d\sigma_R(dx\,d\gamma\,d\mu\,dp)
  \\
  &
    =
    \int_{Z}
    \min\left\{
    \frac{\tilde h^E_{x,R}(1)}{\sfd(e_0(\gamma),e_1(\gamma))}
    ,
    \frac{N}{\rho}
    \right\}
    \delta_{e_1(\gamma)}(\Omega)
    \,d\sigma_R(dx\,d\gamma\,d\mu\,dp).
\end{align*}
Taking into account
\begin{equation}
  p=
  \min\left\{
    \frac{\tilde h^E_{x,R}(1)}{\sfd(e_0(\gamma),e_1(\gamma))}
    ,
    \frac{N(\omega_N\AVR_X)^{\frac{1}{N}}}{\mm(E)^{\frac{1}{N}}}
  \right\}
  \delta_{e_1(\gamma)}(\Omega)
  ,
  \quad\text{ for $\sigma_R$-a.e.\ }
  (x,\gamma,\mu,p)\in Z,
\end{equation}
we continue this chain of inequalities
\begin{align*}
  \PP(E;\Omega)
  &
  \geq
  \int_{Z}
  \min\left\{
    \frac{\tilde h^E_{x,R}(1)}{\sfd(e_0(\gamma),e_1(\gamma))}
    ,
    \frac{N}{\rho}
  \right\}
  \delta_{e_1(\gamma)}(\Omega)
    \,d\sigma_R(dx\,d\gamma\,d\mu\,dp)
    \\&
  =
  \int_Z p(\Omega)  \,d\sigma_R(dx\,d\gamma\,d\mu\,dp).
\end{align*}
Since $\PP(E;A)=\inf\{\PP(E;\Omega):\Omega\supset A\text{ is open}\}$,
for any Borel set $A$, we can conclude.
\end{proof}

Before taking the limit as $R\to\infty$, we state a useful lemma.

\begin{lemma}
Let $X$ be a Polish space, let $Y$, $Z$ be
two compact metric spaces, and let $\mm$ be a finite Radon
measure on $X$.
Consider a sequence of functions $f_n:X\times Y\to Z$ and $f:X\times
Y\to Z$, such that $f$ and $f_n$ are Borel-measurable in the first variable
and continuous in the second.
Suppose that for $\mm$-a.e.\ $x\in X$ the sequence $f_n(x,\cdot)$
converges uniformly to $f(x,\cdot)$.
Consider a sequence of measures $\mu_n\in\M^{+}(X\times Y)$ such that
$\mu_n\rightharpoonup\mu$ weakly in $\M^{+}(X\times Y)$ and
$(\pi_X)_\#\mu_n=\mm$.

Then we have that
\begin{equation}
  (\mathrm{Id}\times f_n)_\#\mu_n
  \rightharpoonup
  (\mathrm{Id}\times f)_\#\mu,
  \quad\text{ weakly in }
  \M(X\times Y\times Z).
\end{equation}
\end{lemma}
\begin{proof}
In order to ease the notation, set
$\nu_n=(\mathrm{Id}\times f_n)_\#\mu_n$ and
$\nu=(\mathrm{Id}\times f)_\#\mu$.
Fix $\epsilon>0$.
We make use an extension of the Egorov's and Lusin's Theorems for
functions taking values in separable metric spaces
(see~\cite[Theorem~7.5.1]{Dudley02} and~\cite[Appendix~D]{Dudley99}).
In this setting, we deal with maps taking value in $C(Y,Z)$, the space
of continuous functions between the compact spaces $Y$ and $Z$, which
is separable.

Therefore, there exists a
compact $K\subset X$ such that:
1) the maps $x\in K\mapsto f_n(x,\cdot)\in C(Y,Z)$ are continuous;
2) the restrictions $x\in K\mapsto f_n(x,\cdot)$ converge to $x\in
K\mapsto f(x,\cdot)$, uniformly in the space $C(K, C(Y,Z))$;
3) $\mm(X\backslash K)\leq\epsilon$.
Regarding point 2), this implies that the restrictions
$f_n|_{K\times Y}\to f|_{K\times Y}$ converge uniformly in $K\times Y$.

We test the convergence of $\nu_n$ against a function
$\phi\in C_b^0(X\times Y\times Z)$
\begin{align*}
  \left|
    \int_{X\times Y\times Z}
    \phi
    \,d\nu_n
    -
    \int_{X\times Y\times Z}
    \phi
    \,d\nu
  \right|
  &
  \leq
  \norm{\phi}_{C^0}
  (\nu_n((X\backslash K)\times Y\times Z)
  +
    \nu((X\backslash K)\times Y\times Z))
  \\&
  \quad\quad
  +
  \left|
    \int_{K\times Y\times Z}
    \phi
    \,d\nu_n
    -
    \int_{K\times Y\times Z}
    \phi
    \,d\nu
  \right|
  \\[2mm]
  &
  =
  \norm{\phi}_{C^0}
  (\mm(X\backslash K)
  +
    \mm(X\backslash K))
  \\&
  \quad\quad
  +
  \left|
    \int_{K\times Y\times Z}
    \phi
    \,d\nu_n
    -
    \int_{K\times Y\times Z}
    \phi
    \,d\nu
  \right|
  \\[2mm]
  &
  \leq
  2\epsilon\norm{\phi}_{C^0}
  +
  \left|
    \int_{K\times Y\times Z}
    \phi
    \,d\nu_n
    -
    \int_{K\times Y\times Z}
    \phi
    \,d\nu
  \right|.
\end{align*}
Regarding the second term, we compute the integral
\begin{equation}
  \int_{K\times Y\times Z}
  \phi
  \,d\nu_n
  =
  \int_{K\times Y}
  \phi(x,y,f_n(x,y))
  \,\mu_n  (dx\,dy).
\end{equation}
Using compactness, one easily checks that $\phi(x,y,f_n(x,y))$
converges to $\phi(x,y,f(x,y))$ uniformly in $K\times Y$.
For this reason, together with the fact that $\mu_n\rightharpoonup
\mu$ weakly we take the limit in the equation above concluding the proof
\begin{equation*}
    \lim_{n\to\infty}
    \int_{K\times Y}
    \phi(x,y,f_n(x,y))
    \,\mu_n  (dx\,dy)
    =
    \int_{K\times Y}
    \phi(x,y,f(x,y))
    \,\mu  (dx\,dy)
    =
    \int_{K\times Y\times Z}
    \phi
      \,d\nu,
      \qedhere
\end{equation*}
\end{proof}

\begin{corollary}
  Consider the function $G:\overline E\times K\to\ProbMeas(F)\times S$
  defined as
\begin{equation}
  G(x,\gamma)=
  \left(
    \gamma_\#(Nt^{N-1}\Leb^1\llcorner_{[0,1]}),
    \max\left\{
      \frac{N}{\sfd(e_0(\gamma),e_1(\gamma))},
      \frac{N}{\rho}
      \right\}
    \delta_{e_1(\gamma)}
  \right),
\end{equation}
and let $\sigma:=(\id\times G)_\#\tau$.
Then it holds that $\sigma_R\rightharpoonup\sigma$ in the weak topology
of measures.
\end{corollary}
\begin{proof}
We need only to check the hypotheses of the previous Lemma.
Due to Remark~\ref{rmrk:reassurement}, the irreversiblity of the
distance is not harmful.
The set $\overline E$ is compact, hence Polish.
The set $K$ is compact and so is $\ProbMeas(F)\times S$ (w.r.t.\ the distance
given by~\eqref{eq:distance-weak-convergence}).
The maps $G_R$ are measurable and continuous in
the first and second variable, respectively.
Finally, we need to see that for $\mm\llcorner_E$-a.e.\ $x$, the limit
$G_R(x,\gamma)\to G(x,\gamma)$ holds uniformly in $\gamma$.
Fix $x$ and $\gamma$ and pick $\psi\in C_b(F)$ a test function and
compute
\begin{align*}
  &
  \left|
  \int_{F}
  \psi(y)
  \,
  \gamma_\#(\tilde h_E^{x,R} \Leb^1\llcorner_{[0,1]})(dy)
  -
  \int_{F}
  \psi(y)
  \,
  \gamma_\#(Nt^{N-1}\Leb^1\llcorner_{[0,1]})(dy)
  \right|
  \\
  &
    \qquad
  =
  \left|
  \int_0^1
  \psi(\gamma_t)
  (\tilde h_E^{x,R}-Nt^{N-1})
  \, dt
  \right|
  \\
  &
    \qquad
    \leq
    \norm{\psi}_{C(F)}
    \norm{\tilde h_E^{x,R}-Nt^{N-1}}_{L^\infty}.
\end{align*}
The r.h.s.\ of the inequality is independent of $\gamma$ (but depends
only on $x$ and $\psi$) and converges to $0$ by
Theorem~\ref{T:rigidity-1d},
(see~particular~\eqref{eq:almost-rigidity-h-tilde}).
Therefore, the first component of $G_R(x,\gamma)$ converges (in the
weak topology of $\ProbMeas(F)$), uniformly w.r.t.\ $\gamma$ (compare
with~\eqref{eq:distance-weak-convergence}).
For the other component the proof is analogous, so we omit it.
\end{proof}

We conclude this section with a proposition reporting all the relevant
properties of the limit measure $\sigma$.

\begin{proposition}\label{P:limit}
  The measure $\sigma$ satisfies the following disintegration
  formulae
  \begin{align}
    &
      \label{eq:disintegration-measure-with-sigma}
      \begin{aligned}
        \int_E\psi(y)\,\mm(dy)
        &
        =
        \int_Z
        \int_0^1
        \psi(e_t(\gamma))Nt^{N-1}
        \,dt
        \,\sigma(dx\,d\gamma\,d\mu\,dp),
        \quad \forall \psi\in L^1(E;\mm\llcorner_E),
      \end{aligned}
    \\
    &
      \label{eq:disintegration-perimeter-with-sigma}
      \begin{aligned}
        \int_{\overline E} \psi(y) \PP(E;dy)
        &
        =
        \frac{N}{\rho}
        \int_Z
        \psi(e_1(\gamma))
        \psi\,\sigma(dx\,d\gamma\,d\mu\,dp),
        \quad
        \forall \psi\in L^1(\overline{E};\PP(E;\,\cdot\,))
        .
      \end{aligned}
  \end{align}
  Furthermore, for $\sigma$-a.e.\ $(x,\gamma,\mu,p)\in Z$ it holds
  \begin{align}
    &
      \label{eq:gamma-along-rays-weak2}
      \sfd(e_t(\gamma),e_s(\gamma))
      =
      \phi_\infty(e_t(\gamma))
      -
      \phi_\infty(e_s(\gamma))
      ,
      \quad
      \forall 0\leq t\leq s\leq 1,
    \\
    &
      \label{eq:speed-of-gamma-weak2}
      \sfd(e_0(\gamma),e_1(\gamma))
      =
      \rho
      ,
  \\
  &
    \label{eq:x-belongs-to-gamma-weak2}
    x\in\gamma
    ,
    \\
    &
      \label{eq:characterization-mu}
      \mu=\gamma_\#(Nt^{N-1}\Leb^1\llcorner_{[0,1]}),
    \\
    &
      \label{eq:characterization-p}
      p
      =
      \frac{N}{\rho}
      \delta_{e_1(\gamma)}.
  \end{align}
\end{proposition}
\begin{proof}
Equations~\eqref{eq:gamma-along-rays-weak2}--\eqref{eq:x-belongs-to-gamma-weak2}
have been already proven (see~of~\eqref{eq:gamma-along-rays-weak}--\eqref{eq:x-belongs-to-gamma-weak}).
Equation~\eqref{eq:characterization-mu} follows from the definition of
$G$.
Similarly, the defintion of $G$ and Equation~\eqref{eq:speed-of-gamma-weak2} imply~\eqref{eq:characterization-p}:
\begin{align*}
  p
  &
  =
  \min
  \left\{
  \frac{N}{\sfd(e_0(\gamma),e_1(\gamma))}
  ,
  \frac{N}{\rho}
  \right\}
  \delta_{e_1(\gamma)}
    =
  \frac{N}{\rho}
    \delta_{e_1(\gamma)}
    .
\end{align*}

We prove now~\eqref{eq:disintegration-measure-with-sigma}.
Given a function $\psi\in C_b^0(F)=C_b^0(e_{(0,1)}(K))$ we define
$L_\psi:\ProbMeas(F)\to\R$ as $L_\psi(\mu)=\int_{F}\psi\,d\mu$.
This latter function is bounded and continuous w.r.t.\ the weak
topology of $\ProbMeas(F)$, thus we can compute the limit
using~\eqref{eq:disintegration-measure-weak}
and~\eqref{eq:characterization-mu}
\begin{align*}
  \int_E \psi \, d\mm
  &
    =
    \lim_{R\to\infty}
    \int_Z
    \int_{F}\psi(y)\,\mu(dy)
    \,\sigma_R(dx\,d\gamma\,d\mu\,dp)
  \\
  &
    =
    \lim_{R\to\infty}
    \int_Z
    L_\psi(\mu)
    \,\sigma_R(dx\,d\gamma\,d\mu\,dp)
  \\
  &
    =
    \int_Z
    \int_{F}\psi(y)\,\mu(dy)
    \,\sigma(dx\,d\gamma\,d\mu\,dp)
    =
  \\
  &
    \int_Z
    \int_0^1\psi(e_t(\gamma))Nt^{N-1}\,dt
    \,\sigma(dx\,d\gamma\,d\mu\,dp)
    .
\end{align*}
Using standard approximation arguments, we see that the equation above
holds true also for any $\psi\in L^1(E;\mm\llcorner_E)$.

Regarding~\eqref{eq:disintegration-perimeter-with-sigma}, one can
analogously deduce that
\begin{equation*}
  \begin{aligned}
    \int_{\overline{E}}\psi(y)\,\PP(E;dy)
    &
      \geq
      \frac{N}{\rho}
    \int_Z
      \psi(e_1(\gamma))
    \,\sigma(dx\,d\gamma\,d\mu\,dp),
    \quad
  \forall\psi\in L^1(\overline E;\PP(E;\,\cdot\,)),\, \psi\geq0.
  \end{aligned}
\end{equation*}
If we test the inequality above with $\psi=1$, the inequality is
saturated thus the two measures have the same mass, so the
inequality improves to an equality.
\end{proof}

\subsection{Back to the classical localization notation}
\label{Ss:localization-classical}
We are now in position to re-obtain a ``classical'' disintegration
formula for the measure $\mm$, as well as for the relative perimeter
of $E$.

We recall the definition of some of the objects that were introduced
in Section~\ref{S:needle}.
For instance, let $\Gamma_\infty:=\{(x,y):
\phi_\infty(x)-\phi_\infty(y)=\sfd(x,y)\}$
and let $\T_\infty$ be the transport set, i.e., the family of
points passing through only one non-degenerate transport curve.
Let $\A_\infty$ the set of branching points (i.e.\ points where two of
more non-degenerate transport curves pass).
The sets of forward and backward branching points are defined as
\begin{align}
  \A_\infty^{+}
  : =
  &~
    \{x\in \A_\infty:\exists y\neq x\text{ such that }(x,y)\in\Gamma_\infty\}
    ,
  \\
  \A_\infty^{-}
  : =
  &~
    \{x\in \A_\infty:\exists y\neq x\text{ such that }(y,x)\in\Gamma_\infty\}
    .
\end{align}
We recall that $\A_\infty=\A_\infty^+\cup\A_\infty^-$ and that $\A_\infty$
is negligible.
Let $Q_\infty$ be the quotient set and let
$\QQ_\infty:\T_\infty\to Q_\infty$ be the quotient map; denote by
$X_{\alpha,\infty}:=\QQ^{-1}(\alpha)$ the disintegration rays and let
$g_\infty:\Dom (g_\infty)\subset Q_\infty\times [0,\infty)\to X$ be the
standard parametrization of the rays.

We introduce the function
$t_\alpha:\overline{X_{\alpha,\infty}}\to[0,\infty)$ defined as
\begin{equation}
  t_\alpha(x)
  :=
  (g_\infty(\alpha,\,\cdot\,))^{-1}
  =
  \sfd(g_\infty(\QQ_\infty(x),0),x)
  ;
\end{equation}
the function $t_\alpha$ measures how much a point is translates from
the starting point of the ray $X_{\alpha,\infty}$.

The following proposition guarantees that the geodesic on which the
measure $\sigma$ is supported lay on the transport set $\T_{\infty}$.

\begin{proposition}
  For $\sigma$-a.e.\ $(x,\gamma,\mu,p)\in Z$, it holds that
  $e_t(\gamma)\in \T_\infty$, for all $t\in(0,1)$.
\end{proposition}
\begin{proof}
  Clearly, for $\sigma$-a.e.\ $(x,\gamma,\mu,p)\in Z$, $\gamma$ is
  non-degenerate, hence $e_t(\gamma)\notin \D$, where $\D$ is the set
  where no non-degenerate transport curve pass.
  Therefore we need only to check that $e_t(\gamma)\not\in\A^\infty$.
  We will prove only that $e_t(\gamma)\neq\A_\infty^+$, for the case
  $e_t(\gamma)\neq\A_\infty^-$ is analogous.
  Fix $\epsilon>0$ and let
  \begin{equation}
    P:=\{(x,\gamma,\mu,p)\in Z: e_\epsilon(\gamma)\in \A^+_\infty
    \text{ and
      conditions~\eqref{eq:disintegration-measure-with-sigma}--\eqref{eq:characterization-p}
      holds}
    \}
  \end{equation}
  Notice that by definition of $\A^+_\infty$, if $(x,\gamma,\mu,p)\in
  P$, then $\gamma_t\in \A^+_\infty$, for all $t\in [0,\epsilon]$,
  thus we can compute
  \begin{align*}
    0
    &
    =
      \mm(\A^+_\infty)
      =
      \int_Z
      \int_0^1
      \indicator_{\A^+_\infty}(e_t(\gamma))
      Nt^{N-1}
      \,dt
      \,\sigma(dx\,d\gamma\,d\mu\,dp)
    \\
    &
      \geq
      \int_P
      \int_0^\epsilon
      \indicator_{\A^+_\infty}(e_t(\gamma))
      Nt^{N-1}
      \,dt
      \,\sigma(dx\,d\gamma\,d\mu\,dp)
      \geq
      \epsilon^N
      \sigma(P),
  \end{align*}
  thus $P$ is negligible.
  Fix now $(x,\gamma,\mu,p)\notin P$.
  By definition of $\A^+_\infty$ and $P$, we have that
  $\gamma_t\not\in \A^+_\infty$, for all $t\in [\epsilon,1]$.
  By arbitrariness of $\epsilon$, we deduce that for
  $\sigma$-a.e\ $(x,\gamma,\mu,p)\in Z$, it holds that
  $e_t(\gamma)\notin \A^+_\infty$, for all $t\in (0,1]$.
\end{proof}
\begin{corollary}
  It holds that $E\subset\T_\infty$ and
  for $\sigma$-a.e.\ $(x,\gamma,\mu,p)\in Z$, we have that
  $e_t(\gamma)\in\overline{X_{\QQ(x),\infty }}$ and
  \begin{equation}
    \label{eq:gamma-parametrization-g}
    e_t(\gamma)
    =g_\infty(\QQ(x),t_{\QQ(x)}(e_0(\gamma))+\rho t).
  \end{equation}
\end{corollary}

Define $\hat{\q}:=\frac{1}{\mm(E)}(\QQ_\infty)_\#(\mm\llcorner_E)\ll(\QQ_\infty)_\#\mm\llcorner_{\T_\infty}$ and let
$\tilde\q$ be a probability measure such that $(\QQ_\infty)_\#\mm\llcorner_{\T_\infty}\ll\tilde\q$.
The disintegration theorem gives the following two formulae,
\begin{equation}
  \label{eq:disintegration-ugly}
  \mm\llcorner_E
  =
  \int_{Q_\infty}
  \hat \mm_{\alpha,\infty}
  \,
  \hat{\q}(d\alpha),
  \quad
  \text{ and }
  \quad
  \mm\llcorner_{\T_\infty}
  =
  \int_{Q_\infty}
  \tilde \mm_{\alpha,\infty}
  \,
  \tilde\q(d\alpha),
\end{equation}
where the measures $\hat{\mm}_{\alpha,\infty}$ and $\tilde\mm_{\alpha,\infty}$
are supported on $X_{\alpha,\infty}$.
By comparing the two expressions above, it turns out that
$\frac{d\hat{\q}}{d\tilde{\q}}(\alpha)\,\hat{\mm}_{\alpha,\infty}
= \indicator_E\tilde\mm_{\alpha,\infty}$.
The Localization Theorem~\ref{T:locMCP} (see also
Remark~\ref{R:disintegration}), ensures that the transport rays
$(X_{\alpha,\infty}, F, \tilde\mm_{\alpha,\infty})$ satisfies the
oriented $\CD(0,N)$ condition.
On the contrary, we cannot deduce the same condition for the other
disintegration, because the reference measure is restricted to the set
$E$ and not the transport set.
Consider the densities $\hat h_\alpha$ and $\tilde h_\alpha$ given by
\begin{equation}
  \hat{\mm}_{\alpha,\infty}
  =
  (g_\infty(\alpha,\,\cdot\,))_\#(\hat h_\alpha
  \Leb^1_{(0,|X_{\alpha,\infty}|)})
  ,
  \quad
  \text{ and }
  \quad
  \tilde\mm_{\alpha,\infty}
  =
  (g_\infty(\alpha,\,\cdot\,))_\#(\tilde h_\alpha
  \Leb^1_{(0,|X_{\alpha,\infty}|)})
  .
\end{equation}
Clearly, it holds that
$\frac{d\hat{\q}}{d\tilde\q}(\alpha)\hat h_\alpha(t) =
\indicator_E(g(\alpha,t))\tilde h_\alpha(t)$, thus we can derive a
somehow weaker concavity condition for the function
$\hat h_\alpha^{\frac{1}{N-1}}$: for all
$x_0,x_1\in(0,|X_{\alpha,\infty}|)$ and for all $t\in[0,1]$, it holds
that
\begin{equation}
  \begin{aligned}
    &
  \hat h_\alpha((1-t) x_0 + t x_1)^{\frac{1}{N-1}}
  \geq
  (1-t) \hat h_\alpha(x_0)^{\frac{1}{N-1}}
      + t   \hat h_\alpha(x_1)^{\frac{1}{N-1}},
    \\
    &
      \qquad
  \text{ if }  \hat h_\alpha((1-t) x_0 + tx_1)>0.
  \end{aligned}
\end{equation}
A natural consequence is the following ``Bishop--Gromov inequality'',
\begin{equation}
  \label{eq:weak-mcp-condition}
  \text{the map
    $r\mapsto \frac{\hat h_\alpha(r)}{r^{N-1}}$ is decreasing on the set
    $\{r\in(0,|X_{\alpha,\infty}|):\hat h_\alpha(r)>0\}$.}
\end{equation}
Define the full-measure set $\hat Z\subset Z$ as
\begin{equation*}
  \begin{aligned}
    \hat Z
    :=
    \{
    &
    (x,\gamma,\mu,p)\in Z
    :
    x\in E\cap\T_{\infty}
    ,
      \text{ and the properties given by}
    \\
    &\qquad
    \text{ Equations~\eqref{eq:disintegration-measure-with-sigma}--\eqref{eq:disintegration-perimeter-with-sigma}
    and~\eqref{eq:gamma-parametrization-g} holds}
    \}
    .
  \end{aligned}
\end{equation*}
We partitionate $\hat Z$ as follows,
\begin{equation}
  \hat Z_\alpha
  :=
  \{
  (x,\gamma,\mu,p)\in\hat Z
  :\QQ_\infty(x)=\alpha
  \}
  ,
\end{equation}
and we disintegrate the measure $\sigma$ according to the partition $(\hat
Z_\alpha)_{\alpha\in Q_\infty}$
\begin{equation}
  \label{eq:disintegration-sigma}
  \sigma
  =
  \int_{Q_\infty}
  \sigma_\alpha
  \,\q(d\alpha),
\end{equation}
where the probability measures $\sigma_\alpha$ are supported on $\hat Z_\alpha$.
Moreover, let $\nu_\alpha\in\ProbMeas([0,\infty))$ be the measure given by
\begin{equation}
  \nu_\alpha
  :=
  \frac{1}{\mm(E)}(t_\alpha \circ e_0 \circ \pi_K)_\#(\sigma_\alpha)
\end{equation}
(we recall that $t_\alpha=(g_\infty(\alpha,\,\cdot\,))^{-1}$ and
$\pi_K(x,\gamma,\mu,p)=\gamma$).

The following proposition shows that the density $\hat{h}_{\alpha}$
can be seen as a convolution of the model density and the measure
$\nu_{\alpha}$.

\begin{proposition}
  For $\hat{\q}$-a.e.\ $\alpha\in {Q_{\infty}}$, it holds that
  \begin{equation}
    \hat h_\alpha(r)
    =
    N\omega_N\AVR_X
    \int_{[0,\infty)}
    (r-t)^{N-1}
    \indicator_{(t,t+\rho)}(r)
    \,
    \nu_\alpha(dt)
    ,
    \quad
    \forall r\in(0,|X_{\alpha,\infty}|)
    .
  \end{equation}
\end{proposition}
\begin{proof}
Fix $\psi\in L^1(\mm\llcorner_E)$ and compute its integral using
the Equations~\eqref{eq:disintegration-measure-with-sigma}
and~\eqref{eq:disintegration-sigma}
\begin{equation}
  \begin{aligned}
    \int_E\psi(x)\,\mm(dx)
    &
      =
      \int_{\hat Z}
      \int_0^1\psi(e_t(\gamma)) N t^{N-1}
      \, dt
      \,
      \sigma(dx\, d\gamma \, d\mu \, dp)
    \\
    &
      =
      \int_{Q_\infty}
      \int_{\hat Z_\alpha}
      \int_0^1\psi(e_t(\gamma)) N t^{N-1}
      \, dt
      \,
      \sigma_\alpha(dx\, d\gamma \, d\mu \, dp)
      \,
      \q(d\alpha)
      .
  \end{aligned}
\end{equation}
Fix now $\alpha\in {Q_{\infty}}$ and compute
(recall~\eqref{eq:gamma-parametrization-g} and the definition of
$\hat Z$)
\begin{equation}
  \begin{aligned}
      &
      \int_{\hat Z_\alpha}
      \int_0^1\psi(e_t(\gamma)) N t^{N-1}
      \, dt
      \,
      \sigma_\alpha(dx\, d\gamma \, d\mu \, dp)
    \\
      &
        \qquad
      =
      \int_{\hat Z_\alpha}
      \int_0^\rho\psi(e_{s/\rho}(\gamma)) N \frac{s^{N-1}}{\rho^N}
      \, ds
      \,
      \sigma_\alpha(dx\, d\gamma \, d\mu \, dp)
    \\
    &
        \qquad
      =
      \int_{\hat Z_\alpha}
      \int_0^\rho
      \psi(g_\infty(\QQ(x),t(\alpha,\gamma_0)+s))
      N \frac{s^{N-1}}{\rho^N}
      \, ds
      \,
      \sigma_\alpha(dx\, d\gamma \, d\mu \, dp)
    \\
    &
        \qquad
      =
      \int_{\hat Z_\alpha}
      \int_0^{|X_{\alpha,\infty}|}
      \psi(g_\infty(\alpha,r))
      N \frac{(r-t(\alpha,\gamma_0))^{N-1}}{\rho^N}
      \times
    \\
      &
        \qquad\qquad \times
      \indicator_{(t(\alpha,\gamma_0),t(\alpha,\gamma_0)+\rho)}(r)
      \,
      dr
      \,
      \sigma_\alpha(dx\, d\gamma \, d\mu \, dp)
    \\
      &
        \qquad
      =
      \int_0^{|X_{\alpha,\infty}|}
      \psi(g_\infty(\alpha,r))
      \int_{\hat Z_\alpha}
      N \frac{(r-t(\alpha,\gamma_0))^{N-1}}{\rho^N}
      \times
    \\
      &
        \qquad\qquad \times
      \indicator_{(t(\alpha,\gamma_0),t(\alpha,\gamma_0)+\rho)}(r)
      \,
      \sigma_\alpha(dx\, d\gamma \, d\mu \, dp)
      \,
      dr
      ,
  \end{aligned}
\end{equation}
therefore, by the uniqueness of the disintegration, we can conclude
\begin{align*}
    \hat{h}_{\alpha}(r)
    &
  =
  \int_{\hat Z_\alpha}
  N \frac{(r-t(\alpha,\gamma_0))^{N-1}}{\rho^N}
  \indicator_{(t(\alpha,\gamma_0),t(\alpha,\gamma_0)+\rho)}(r)
  \,
  \sigma_\alpha(dx\, d\gamma \, d\mu \, dp)
    \\
    &
  =
  N\omega_N\AVR_X \int_{[0,\infty)}
  (r-t)^{N-1}
  \indicator_{(t,t+\rho)}(r)
  \,
  \nu_\alpha(dt)
  .
      \qedhere
  \end{align*}
\end{proof}

Using the fact that $\hat h_\alpha$ is a convolution, we deduce that
$\nu_\alpha$ is indeed the Dirac delta.
\begin{proposition}
For $\hat{\q}$-a.e.\ $\alpha\in {Q_{\infty}}$, it holds that $\nu_\alpha=\delta_0$.
\end{proposition}
\begin{proof}
  Let $T:=\inf \supp\nu_\alpha$.
  If we set $r\in(T,T+\rho)$, we can compute
  \begin{equation}
    \label{eq:hat-h-is-positive}
  \begin{aligned}
    \frac{\hat h_{\alpha,\infty}(r)}{N\omega_N\AVR_X}
    &
  =
   \int_{[0,\infty)}
  (r-t)^{N-1}
  \indicator_{(t,t+\rho)}(r)
  \,
  \nu_\alpha(dt)
      =
  \int_{[T,r)}
  (r-t)^{N-1}
  \,
      \nu_\alpha(dt)
    \\
    &
      \geq
  \int_{[T,r)}
      \left(
      \frac{r-T}{2}
      \indicator_{[T,(r+T)/2]}(t)
      \right)^{N-1}
  \,
      \nu_\alpha(dt)
      =
      \frac{(r-T)^{N-1}}{2^{N-1}}
      \nu_\alpha([T,\tfrac{r+T}{2}]).
  \end{aligned}
\end{equation}
  By definition of $T$, we have that
  $\nu_\alpha([T,\frac{r+T}{2}])>0$, hence $\hat h_\alpha(r)>0$, for all
  $r\in(T,T+\rho)$.
  On the other hand
  \begin{equation}
    \label{eq:hat-h-goes-to-zero}
  \begin{aligned}
    \hat h_{\alpha,\infty}(r)
    &
  =
      N\omega_N\AVR_X
  \int_{[T,r)}
  (r-t)^{N-1}
  \,
      \nu_\alpha(dt)
    \\
    &
      \leq
      N\omega_N\AVR_X
      (r-T)^{N-1}
  \,
      \nu_\alpha([T,r))
      \to
      0
      .
      \quad
      \text{ as }
      r\to T^+
      .
  \end{aligned}
\end{equation}
  We claim that $T=0$.
  Indeed, if $T>0$, then $\lim_{r\to T^{+}}\hat h_\alpha(r)/r^{N-1}=0$
  contradicting~\eqref{eq:weak-mcp-condition}.
  We derive now the non-increasing function
  \begin{equation}
    (0,\rho)\ni r
    \mapsto
    \frac{\hat h_\alpha(r)}{r^{N-1}}
    =
    \frac{N\omega_N\AVR_X}{r^{N-1}}
    \int_{[0,r)}
    (r-t)^{N-1}
    \,
    \nu_\alpha(dt)
    ,
  \end{equation}
  obtaining
  \begin{align*}
    0
    &
      \geq
      N\omega_N\AVR_X
      \left(
      \frac{1-N}{r^{N}}
    \int_{[0,r)}
    (r-t)^{N-1}
    \,
      \nu_\alpha(dt)
      +
      \frac{1}{r^{N-1}}
      \frac{d}{dr}
          \int_{[0,r)}
    (r-t)^{N-1}
    \,
    \nu_\alpha(dt)
      \right)
      .
  \end{align*}
  We compute the second term
  \begin{align*}
    \frac{d}{dr}
    \int_{[0,r)}
    (r-t)^{N-1}
    \,
    \nu_\alpha(dt)
&=
      \lim_{h\to 0}
          \int_{[r,r+h)}
    \frac{(r+h-t)^{N-1}}{h}
    \,
      \nu_\alpha(dt)
    \\
    &
      \qquad
      +
      \lim_{h\to 0}
      \int_{[0,r)}
      \frac{(r+h-t)^{N-1}-(r-t)^{N-1}}{h}
    \,
      \nu_\alpha(dt)
    \\
    &
      \geq
      0 + 
      \int_{[0,r)}
      \lim_{h\to 0}
      \frac{(r+h-t)^{N-1}-(r-t)^{N-1}}{h}
    \,
      \nu_\alpha(dt)
    \\[2mm]
    &
      =
      (N-1)\int_{[0,r)}
      (r-t)^{N-2}
    \,
      \nu_\alpha(dt)
      ,
  \end{align*}
  yielding
  \begin{align*}
    0
    &
      \geq
      (1-N)
    \int_{[0,r)}
    (r-t)^{N-1}
    \,
      \nu_\alpha(dt)
      +
      r
      \frac{d}{dr}
          \int_{[0,r)}
    (r-t)^{N-1}
    \,
    \nu_\alpha(dt)
    \\
    &
      \geq
      (N-1)
    \int_{[0,r)}
    (r(r-t)^{N-2}-(r-t)^{N-1})
    \,
      \nu_\alpha(dt)
    \\
    &
      =(N-1)
      \int_{[0,r)}
      t(r-t)^{N-2}
    \,
      \nu_\alpha(dt).
  \end{align*}
The inequality above gives $\nu_\alpha((0,r))=0$, for all
$r\in(0,\rho)$, hence $\nu_\alpha(0,\rho)=0$.
We deduce that
\begin{equation}
\hat  h_\alpha(r)
  =
  N\omega_N\AVR_X
  \int_{[0,r)}
  (r-t)^{N-1}
  \, \nu_\alpha(dt)
  =
  N\omega_N\AVR_X
  \,
  r^{N-1} \nu_\alpha(\{0\})
  ,
  \quad
  \forall r\in(0,\rho)
  .
\end{equation}
If $\nu_\alpha([\rho,\infty))=0$, then $\nu_\alpha=\delta_0$ (because
$\nu_\alpha$ has mass $1$) completing the proof.
Assume on the contrary that $\nu_\alpha([\rho,\infty))>0$, and
let $S:=\inf \supp (\nu_\alpha\llcorner_{[\rho,\infty)})\geq\rho$.
In this case, following the computations~\eqref{eq:hat-h-is-positive}
and~\eqref{eq:hat-h-goes-to-zero}, with $S$ in place of $T$, we deduce
$\lim_{r\to S^+} \hat h_\alpha(r)=0$,
contradicting~\eqref{eq:weak-mcp-condition}.
\end{proof}
\begin{corollary}
  \label{cor:gamma-parametrization-g-improved}
  For $\hat{\q}$-a.e.\ $\alpha\in {Q_{\infty}}$, for
  $\sigma_\alpha$-a.e. $(x,\gamma,\mu,p)\in Z_\alpha$, it holds that
  $e_t(\gamma)=g(\alpha,\rho t)$, $\forall t\in[0,1]$.
\end{corollary}
\begin{proof}
  The fact that $\nu_\alpha=\delta_0$, implies $t_\alpha(\gamma_0)=0$
  for $\sigma_\alpha$-a.e.\ $(x,\gamma\,\mu,p)\in\hat Z_\alpha$,
  hence, recalling the disintegration
  formula~\eqref{eq:gamma-parametrization-g} and the definition of
  $\hat{Z}$, we deduce that
  $ e_t(\gamma) = g(\alpha,t_\alpha(e_0)+\rho t) = g(\alpha,\rho t)$.
\end{proof}

The next corollary concludes the discussion of the limiting procedures
of the disintegration.

\begin{corollary}
  \label{cor:disintegration-classical}
  For $\hat{\q}$-a.e.\ $\alpha\in {Q_{\infty}}$, it holds that
  \begin{equation}
    \hat h_\alpha(r)
    =
    N\omega_N\AVR_X
    \indicator_{(0,\rho)}(r)
    r^{N-1}
    .
  \end{equation}
  Moreover, the following disintegration formulae hold true,
  \begin{align}
    \label{eq:disintegration-measure-classical}
    &
      \mm\llcorner_E
    =
    N\omega_N\AVR_X
    \int_{{Q_{\infty}}}
      (g_\infty(\alpha,\,\cdot\,))_{\#}
      (
      r^{N-1}
      \,\Leb^1\llcorner_{(0,\rho)})
    \,\hat{\q}(d\alpha)
    ,
    \\
    \label{eq:disintegration-perimeter-classical}
    &
      \PP(E;\,\cdot\,)
      =
      \PP(E)
      \int_{Q_{\infty}}
      \delta_{g_\infty(\alpha,\rho)}
      \,
      \hat{\q}(d\alpha)
      .
  \end{align}
\end{corollary}
\begin{proof}
  We need only to prove
  Equation~\eqref{eq:disintegration-perimeter-classical}.
  Equation~\eqref{eq:disintegration-perimeter-with-sigma} and
  Corollary~\ref{cor:gamma-parametrization-g-improved} yield
  \begin{align*}
    \int_{\overline E}
    \psi(x)
    \,
    \PP(E;dx)
    &
      =
      \frac{N}{\rho}
      \int_{\hat Z}
      \psi(e_1(\gamma))
      \psi\,\sigma(dx\,d\gamma\,d\mu\,dp)
    \\
    &
      =
      \frac{N}{\rho}
      \int_{Q_{\infty}}
      \int_{\hat Z_\alpha}
      \psi(e_1(\gamma))
      \,\sigma_\alpha(dx\,d\gamma\,d\mu\,dp)
      \,\hat{\q}(d\alpha)
    \\
    &
      =
      \frac{N}{\rho}
      \int_{Q_{\infty}}
      \psi(g_\infty(\alpha,\rho))
      \int_{\hat Z_\alpha}
      \,\sigma_\alpha(dx\,d\gamma\,d\mu\,dp)
      \,\hat{\q}(d\alpha)
      ,
    \\
    &
      \qquad\qquad
      \forall\psi\in L^{1}(\overline{E};\PP(E;\,\cdot\,))
      .
      \qedhere
  \end{align*}
\end{proof}

\section{\texorpdfstring{$E$}{E} is a ball}
\label{eq:is-a-ball}

The aim of this section is to prove that $E$ coincides with a ball of
radius $\rho$ and to extend the disintegration formula to the whole
manifold.
Before starting the proof, we give a topological technical lemma.
This lemma is, in some sense, a weak formulation of the  statement: let
$\Omega$ be an open connected subset of a topological space $X$ and
let $E\subset X$ be any set; if $\Omega\cap E\neq\emptyset$ and
$\Omega\backslash E\neq\emptyset$, then we have that
$\partial E\cap\Omega\neq\emptyset$.

\begin{lemma}
  \label{lem:boundary-non-empty}
  Let $(X,F,\mm)$ be measured Finsler manifold (with possible infinite
  reversibility).
Let $E\subset X$ be a Borel set and let $\Omega\subset X$ be an open
connected set with finite measure.
If $\mm(E\cap \Omega)>0$ and $\mm(\Omega\backslash E)>0$, then
$\PP(E;\Omega)>0$.
\end{lemma}
\begin{proof}
  Assume first that the manifold is riemannian.
  In this case, we can assume by contradiction that
  $\PP(E;\Omega)=0$, yielding that the BV function $\indicator_E$ is
  constant in $\Omega$, but this contradicts the hypotheses
  $\mm(E\cap \Omega)>0$ and $\mm(\Omega\backslash E)>0$.

  We now drop the riemannianity hypothesis.
  As we stressed out (see Remark~\ref{rmrk:reassurement}), there
  exists a riemannian metric $g$, such that its dual metric $g^{-1}$
  in $T^*X$ satisfies
  $\sqrt{g^{-1}(\omega,\omega)}\leq F^*(\omega)$, for all $\omega\in T^{*}X$.
  By definition of perimeter, there exists a sequence
  $u_n\in\Lip_{loc}(\Omega)$ such that $u_n\to\indicator_E$ in
  $L^1_{loc}$ and $\int_\Omega |\partial u_n|\,d\mm\to \PP_{(X,F,\mm)}(E;\Omega)$.
  Since $g^{-1}(du_n,d u_n)\leq F^*(-du_n) = |\partial u_n|$ a.e.\ in
  $\Omega$, we conclude that
  $\PP_{(X,g,\mm)}(E;\Omega)\leq\PP_{(X,F,\mm)}(E;\Omega)$.
\end{proof}

\begin{proposition}
  \label{P:min-max-phi}
  For $\hat{\q}$-a.e.\ $\alpha\in {Q_{\infty}}$, it holds
  that
  \begin{equation}
    \phi_\infty(g_\infty(\alpha,0))\leq \esssup_E\phi_\infty
    ,
    \quad\text{ and }\quad
    \phi_\infty(g_\infty(\alpha,\rho))\geq \essinf_E\phi_\infty
    .
  \end{equation}
\end{proposition}
\begin{proof}
  We prove only the former inequality; the latter has the same proof.
In order to ease the notation define $M:=\esssup_E\phi_\infty$.
Let
$H:=\{\alpha\in {Q_{\infty}}: \phi_\infty(g_\infty(\alpha,0))\geq
M+2\epsilon\}$.
Consider the following measure on $E$,
\begin{equation}
  \mathfrak{n}(T)
  :=
  N\omega_N\AVR_X
  \int_H
  \int_0^\epsilon
    \indicator_T(g_\infty(\alpha,r)) r^{N-1}\,dr
    \,\hat{\q}(d\alpha)
    ,
    \quad
    \forall T\subset E \text{ Borel}
    .
\end{equation}
Clearly, $\mathfrak{n}\ll\mm$ (compare
with~\eqref{eq:disintegration-measure-classical}), thus
$\phi_\infty(x)\leq M$, for $\mathfrak{n}$-a.e.\ $x\in E$.
If we compute the integral
\begin{align*}
  0&
     \geq
    \int_{E}
  \left(
  \phi_\infty(x)-M
  \right)
  \,\mathfrak{n}(dx)
     =
     N\omega_N\AVR_X
    \int_H
    \int_0^{\epsilon}
    \left(
    \phi_\infty(g_\infty(\alpha,t))
    -M
    \right)
    t^{N-1}\,dt
    \,\hat{\q}(d\alpha)
  \\
  &
    =
     N\omega_N\AVR_X
    \int_H
    \int_0^{\epsilon}
    \left(
    \phi_\infty(g_\infty(\alpha,0))-t
    -M
    \right)
    t^{N-1}\,dt
    \,\hat{\q}(d\alpha)
  \\
  &
    \geq
     N\omega_N\AVR_X
    \int_H
    \int_0^{\epsilon}
    \epsilon
    t^{N-1}\,dt
    \,\hat{\q}(d\alpha)
    =\epsilon^N\hat{\q}(H).
\end{align*}
we can deduce that $\hat{\q}(H)=0$ and, by arbitrariness of $\epsilon$,
we conclude.
\end{proof}

\begin{theorem}\label{T:Ball}
  There exists a (unique) point $o\in X$, such that, up to a
  negligible set, $E=B^{+}(o,\rho)$, where
  $\rho=(\frac{\mm(E)}{\omega_N\AVR_X})^\frac{1}{N}$.
  Moreover, it holds that
  \begin{equation}
    \label{eq:phi-has-max-in-o}
    \phi_\infty(o)
    =\esssup_{E}\phi_\infty
    =
    \max_{B^{+}(o,\rho)}\phi_\infty.
  \end{equation}
\end{theorem}
\begin{proof}
Define
$\tilde E:=\supp\indicator_{E}$.
Recall that by definition of support, $\tilde{E}=\bigcup_C
C$, where the intersection is taken among all closed sets
$C$ such that $\mm(E\backslash C)=0$;
and in particular $\mm(E\backslash\tilde E)=0$.
Let $o\in\argmax_{\tilde E}\phi_\infty$.
By definition of $\tilde E$, we have that $\max_{\tilde
  E}\phi_\infty=\esssup_E\phi_\infty$, deducing the first equality
of~\eqref{eq:phi-has-max-in-o}.
The other equality in~\eqref{eq:phi-has-max-in-o} will follow from the
fact $E=B^{+}(o,\rho)$ (up to a negligible set).

It is sufficient to prove only that $B^{+}(o,\rho)\subset E$, for
the other inclusion is automatic
Indeed, the Bishop--Gromov inequality, together with the definition of
a.v.r.\ yields
\begin{equation}
  \mm(E)
  \geq
  \mm(B^{+}(o,\rho))
  \geq
  \omega_N \AVR_X \rho^{N}
  =
  \mm(E),
\end{equation}
and the equality of measures improves to an equality of sets.

Fix now $\epsilon>0$ and define $A=B^{+}(o,\rho-\epsilon)$.
If $\mm(A\backslash E)=0$, then we deduce that
$B^{+}(o,\rho-\epsilon)\subset E$ and, by
arbitrariness of $\epsilon$, we can conclude.

Suppose on the contrary that $\mm(A\backslash E)>0$.
Clearly $A$ is connected and $\mm(A\cap E)>0$ (otherwise
$o\notin\tilde E$), so we can apply Lemma~\ref{lem:boundary-non-empty}
obtaining $\PP(E;A)>0$.
Define $H=\{\alpha\in {Q_{\infty}}: g_\infty(\alpha,\rho)\in A\}$.
The set $H$ is non-negligible because
(recall~\eqref{eq:disintegration-perimeter-classical})
\begin{align*}
  0<
  \frac{\PP(E;A)}{\PP(E)}
  &
    =
    \int_{Q_{\infty}} \indicator_A(g_\infty(\alpha,\rho))\,\hat{\q}(d\alpha)
    =
    \int_H \indicator_A(g_\infty(\alpha,\rho))\,\hat{\q}(d\alpha)
    =
    \hat{\q}(H)
    .
\end{align*}
By lipschitz-continuity of $\phi_\infty$ we deduce
\begin{equation}
  \phi_\infty(x)
  \geq
  \phi_\infty(o)-\rho+\epsilon
  \geq
  M-\rho+\epsilon
  ,
  \quad
  \forall x\in A = B^{+}(o,\rho-\epsilon)
\end{equation}
hence
\begin{equation}
  \phi_\infty(g_\infty(\alpha,\rho))
  \geq
  M-\rho+\epsilon,
  \quad
  \forall \alpha\in H
  .
\end{equation}
Continuing the chain of inequalities, we arrive at
\begin{equation}
  \phi_\infty(g_\infty(\alpha,0))
  =
  \phi_\infty(g_\infty(\alpha,\rho))
  +\rho
  \geq
  M+\epsilon,
  \quad
  \forall \alpha\in H
  .
\end{equation}
The line above, together with the fact that $\hat{\q}(H)>0$,
contradicts Proposition~\ref{P:min-max-phi}.
\end{proof}

\subsection{\texorpdfstring{$\phi_\infty(x)$}{φ∞(x)} coincides with
  \texorpdfstring{$-\sfd(o,x)$}{-d(o,x)}}

The present section is devoted in proving that,
$\phi_\infty(x)=-\sfd(o,x)+\phi_\infty(o)$.

\begin{proposition}
  For $\hat{\q}$-a.e.\ $\alpha\in {Q_{\infty}}$, it holds that
  \begin{align}
    \label{eq:gamma-are-radial}
    \sfd(o,g_\infty(\alpha,t))= t,
    \quad
    \forall t\in[0,\rho]
    .
  \end{align}
\end{proposition}
\begin{proof}
  By the $1$-lipschitzianity of $\phi_\infty$ and the fact that
  $E=B^{+}(o,\rho)$ (up to a negligible set) we deduce that
  $\phi_\infty(x)\geq\phi_\infty(o)-\rho$, for $\mm$-a.e.\ $x\in E$.
  Henceforth, Proposition~\ref{P:min-max-phi} and
  Equation~\eqref{eq:phi-has-max-in-o} yield
  \begin{equation}
    \phi_\infty(g_\infty(\alpha,0))\leq\phi_\infty(o),
    \quad\text{ and }\quad
    \phi_\infty(g_\infty(\alpha,\rho))\geq\phi_\infty(o)-\rho
    .
  \end{equation}
  Since $\frac{d}{dt}\phi_\infty( g_\infty(\alpha,t))=-1$, $t\in (o,\rho)$, the
  inequalities above are saturated, i.e., it holds that
  \begin{equation}
    \phi_\infty(g_\infty(\alpha,t))=\phi_\infty(o)-t
    ,
    \quad
    \forall t\in[0,\rho]
    ,
    \text{ for $\hat{\q}$-a.e.\ }
    \alpha\in Q_{\infty}
    .
  \end{equation}
  Using again the $1$-lipschitzianity of $\phi_\infty$, we arrive at
  \begin{equation}
    \label{eq:gamma-is-far-away}
    \begin{aligned}
    \sfd(o,g_\infty(\alpha,t))
      &
        \geq
        \phi_\infty(o)-\phi_\infty(g_\infty(\alpha,t))
        =
        t
    ,
    \quad
    \forall t\in[0,\rho]
    ,
    \text{ for $\hat{\q}$-a.e.\ }
    \alpha\in Q_{\infty}
        .
    \end{aligned}
  \end{equation}
  Now fix $\epsilon>0$ and let
  $C=\{\alpha\in {Q_{\infty}}:
  \sfd(o,g_\infty(\alpha,0))>(1+\Lambda_{F})\epsilon\}$, where
  $\Lambda_F$ is the reversibility constant.
  Define the function $f(t):=\inf\{\sfd(o,g_\infty(\alpha,t)): \alpha\in
  C\}$.
  Clearly, $f$ is $\Lambda_F$-Lipschitz and satisfies
  $f(0)\geq (1+\Lambda_F)\epsilon$, hence
  $f(t)\geq(1+\Lambda_F)\epsilon- \Lambda_Ft$, yielding
  (cfr.~\eqref{eq:gamma-is-far-away})
  \begin{equation}
    f(t)\geq
    \max\{((1+\Lambda_F)\epsilon- \Lambda_F t),t\}
    \geq
    \epsilon.
  \end{equation}
  The inequality above implies that $g_\infty(\alpha,t)\notin B^{+}(o,\epsilon)$
  for all $t\in [0,1]$, for all $\alpha\in C$.
  We compute $\mm(B^+(0,\epsilon))$ using the disintegration
  formula~\eqref{eq:disintegration-measure-classical}
  \begin{align*}
    \frac{\mm(B^{+}(o,\epsilon))}{N\omega_N\AVR_X}
    &
      =
      \int_{Q_{\infty}}
      \int_0^\rho
      \indicator_{B^{+}(o,\epsilon)}(g_\infty(\alpha,t))
      \,
      t^{N-1}
      \,
      dt
      \,
      \hat{\q}(d\alpha)
    \\
    &
      =
      \int_{{Q_{\infty}}\backslash C}
      \int_0^\rho
      \indicator_{B^{+}(o,\epsilon)}(g_\infty(\alpha,t))
      \,
      t^{N-1}
      \,
      dt
      \,
      \hat{\q}(d\alpha)
.
  \end{align*}
  If $\indicator_{B^{+}(o,\epsilon)}(g_\infty(\alpha,t))=1$, then
  inequality~\eqref{eq:gamma-is-far-away} yields
  $t\leq\epsilon$, so we continue the computation
  \begin{align*}
    \frac{\mm(B^{+}(o,\epsilon))}{N\omega_N\AVR_X}
    &
      =
      \int_{{Q_{\infty}}\backslash C}
      \int_0^\rho
      \indicator_{B^{+}(o,\epsilon)}(g_\infty(\alpha,t))
      \,
      t^{N-1}
      \,
      dt
      \,
      \hat{\q}(d\alpha)
    \\
    &
      =
      \int_{{Q_{\infty}}\backslash C}
      \int_0^\epsilon
      \indicator_{B^{+}(o,\epsilon)}(g_\infty(\alpha,t))
      \,
      t^{N-1}
      \,
      dt
      \,
      \hat{\q}(d\alpha)
    \\
    &
      \leq
      \int_{{Q_{\infty}}\backslash C}
      \int_0^\epsilon
      \,
      t^{N-1}
      \,
      dt
      \,
      \hat{\q}(d\alpha)
      =
      (\hat{\q}({Q_{\infty}})-\hat{\q}(C))\frac{\epsilon^N}{N}
      .
  \end{align*}
  On the other hand, the Bishop--Gromov inequality yields
  \begin{equation}
    \mm(B^{+}(o,\epsilon))
    \geq
    \frac{\epsilon^N}{\rho^N}
    \mm(B^{+}(o,\rho))
    =
    \frac{\epsilon^N}{\rho^N}
    \mm(E)
    =\epsilon^N
    \omega_N\AVR_X
    .
  \end{equation}
  The comparison of the two previous inequality gives
  $\hat{\q}(C)=0$.
  By arbitrariness of $\epsilon$, we deduce that $g_\infty(\alpha,0)=o$ for
  $\hat{\q}$-a.e.\ $\alpha\in {Q_{\infty}}$.

  Finally, using again~\eqref{eq:gamma-is-far-away}, we conclude
  \begin{align*}
    &
    t
    \leq
    \sfd(o,g_\infty(\alpha,t))
    \leq
    \sfd(o,g_\infty(\alpha,0))
    +
    \sfd(g_\infty(\alpha,0),g_\infty(\alpha,t))
    =
    t
      ,
    \\
    &
    \qquad
    \forall t\in[0,\rho]
    ,
    \text{ for $\hat{\q}$-a.e\ }
    \alpha\in {Q_{\infty}}.
    \qedhere
  \end{align*}
\end{proof}

\begin{corollary}
  It holds that for all $x\in B^{+}(o,\rho)$,
  $\phi_\infty(x)=\phi_\infty(o)-\sfd(o,x)$.
\end{corollary}
\begin{proof}
  If $x\in E\cap\T_\infty$, then $x=g(\alpha,t)$, for some $t$, with
  $\alpha=\QQ_\infty(x)$.
  By the previous proposition we may assume that $g_\infty(\alpha,0)=o$,
  hence we have that
  \begin{equation*}
    \phi_\infty(x)
    -
    \phi_\infty(o)
    =
    \phi_\infty(g_\infty(\alpha,t))
    -
    \phi_\infty(g_\infty(\alpha,0))
    =
    -
    \sfd(g_\infty(\alpha,0),g_\infty(\alpha,t))
    =
    -
    \sfd(o,x)
    .
  \end{equation*}
  Since $\T_\infty\cap E$ has full measure in $B^{+}(o,\rho)$, we conclude.
\end{proof}

\subsection{Localization of the whole space}

We can now extend the localization given in
Section~\ref{Ss:localization-classical} to the whole space $X$.
Since we do not know the behaviour of $\phi_{\infty}$ outside
$B^{+}(o,\rho)$, we take as reference $1$-Lipschitz function $-\sfd(o,\,\cdot\,)$,
which coincides with $\phi_{\infty}$ on $B^{+}(o,\rho)$:
we disintegrate using $-\sfd(o,\,\cdot\,)$ and we see that this
disintegration coincides with the one given
Section~\ref{Ss:localization-classical} in the set $E$.
From this fact, and the geometric properties of the space, we will
conclude.

We recall some of the concepts introduced in Subsection~\ref{S:needle},
applied to the $1$-Lipschitz function $-\sfd(o,\,\cdot\,)$.
The set $\D$ where no non degenerate transport curve pass is empty,
for we can connect $o$ to any point with a minimal geodesic.
The set of branching points, $\A$, contains only $o$ and elements of
the boundary; this follows
from the uniqueness of the geodesics.
For this reason, the transport set $\T$ coincides with
$X\backslash \{o\}$.
Let $Q\subset\T$ be a measurable section and let $\QQ:\T\to{Q}$ be the
quotient map; let $X_\alpha:=\QQ^{-1}(\alpha)$ be the disintegration
rays and let $g:\Dom (g)\subset Q\times\R \to X$ be the standard
parametrization.
The map $t\mapsto g(\alpha,t)$ is the unitary speed
parametrization of the geodesic connecting $o$ to $\alpha$ and then
maximally extended.
Define $\q:=\frac{1}{\mm(E)}\QQ_\#(\mm\llcorner_{E})$.
Using the $\CD(0,N)$ condition, one immediately sees that $\QQ_{\#}(\mm)\ll\q$.

We are in position to use Theorem~\ref{T:locMCP} (compare with
Remark~\ref{R:disintegration}), hence there exists a unique
disintegration for the measure $\mm$
\begin{equation}
  \label{eq:disintegration-final}
  \mm
  =
  \int_{{Q}}\mm_\alpha \,\q(d\alpha)
  ,
\end{equation}
where the measures $\mm_\alpha$ are supported on $X_\alpha$ and the
transport ray $(X_\alpha,F,\mm_\alpha)$ satisfy the oriented
$\CD(0,N)$ condition.
We denote by $h_\alpha:(0,|X_\alpha|)\to\R$ the density function satisfying
$\mm_\alpha=(g(\alpha,\cdot))_\#(h_\alpha \Leb^1\llcorner_{(0,|X_\alpha|)})$.

The next two propositions bind together the disintegration obtained
in Section~\ref{Ss:localization-classical} (in particular
Corollary~\eqref{cor:disintegration-classical}) with the
disintegration given by~\eqref{eq:disintegration-final}.

\begin{proposition}
  There exists a (unique) measurable map
  $L:\Dom(L)\subset Q_\infty\to{Q}$ such that $\D(L)$ has full
  $\hat{\q}$-measure in $Q_{\infty}$ and it holds
\begin{equation}
  L(\QQ_\infty(x))
  =
  \QQ(x)
  ,
  \quad
  \forall
  x\in B^{+}(o,\rho)\cap\T_\infty\cap  \T
  ,
  \qquad
  \text{ and }
  \qquad
  \q=L_\# \hat{\q}
  .
\end{equation}
\end{proposition}
\begin{proof}
  Since $\phi_{\infty}=\phi_{\infty}(o)-\sfd(o,\,\cdot\,)$ on
  $B^+(o,\rho)$, the partitions
  $(X_{\alpha,\infty})_{\alpha\in Q_{\infty}}$ and
  $(X_{\alpha})_{\alpha\in Q}$ agree on the set
  $B^{+}(o,\rho)\cap\T_{\infty}\cap\T$.
  Consider the set
  \begin{equation}
    H
    :=
    \{
    (x,\alpha,\beta)
    \in (B^+(o,\rho)\cap\T_{\infty}\cap\T)\times Q_{\infty} \times Q
    : \QQ_{\infty}(x)=\alpha
    \text{ and }
    \QQ(x)=\beta
    \}
    ,
  \end{equation}
  and let $G:=\pi_{Q_{\infty} \times Q}(H)$ be the projection of $H$
  on the second and third variable.
  For what we have said $G$ is the graph of a map
  $L:\Dom(L)\subset Q_{\infty}\to Q$.
  The other properties easily follow.
\end{proof}

\begin{proposition}
  For $q$-a.e.\ $\alpha\in {Q}$, it holds that
  $|X_\alpha|\geq\rho$ and
  \begin{equation}
    h_\alpha(r)=N\omega_N\AVR_X r^{N-1}
    ,
    \quad
    \forall r\in[0,\rho]
    .
  \end{equation}
\end{proposition}
\begin{proof}
  Using Equation~\eqref{eq:gamma-are-radial},
  we deduce that
  for $\hat{\q}$-a.e.\ $\alpha\in {Q_{\infty}}$, it holds that 
  \begin{equation}
    g_\infty(\alpha,t)
    =
    g(L(\alpha),t)
    ,
    \quad
    \forall t\in(0,\min\{\rho,|X_\alpha|\})
    .
  \end{equation}
  Since in the
  disintegration~\eqref{eq:disintegration-measure-classical}, all rays
  have length $\rho$, we deduce that $|X_\alpha|\geq \rho$.
  Moreover, we obtain $\hat{\mm}_{\alpha,\infty}=(\mm_{L(\alpha)})\llcorner_E$, concluding.
\end{proof}

\begin{theorem}
  \label{th:disintegration-final}
  For $\q$-a.e.\ $\alpha\in {Q}$, it holds that $|X_\alpha|=\infty$ and
  \begin{equation}
    h_\alpha(r)
    =
    N\omega_N\AVR_X
    r^{N-1}
    ,
    \quad
    \forall r>0
    .
  \end{equation}
\end{theorem}
\begin{proof}
Fix $\epsilon>0$ and let
\begin{equation}
  C
  :=
  \left\{
    \alpha\in {Q}:
    \lim_{R\to\infty}\int_0^R h_\alpha/R^N
    < \omega_N \AVR_X(1-\epsilon)
  \right\}
  ,
\end{equation}
with the convention that the limit above is $0$ if
$|X_\alpha|<\infty$.  The limit always exists and it is not larger
than $\omega_N\AVR_X$ by the Bishop--Gromov inequality applied to each
transport ray.
We compute $\AVR_X$ using the disintegration
\begin{align*}
  \omega_N\AVR_X
  &
    =
    \lim_{R\to\infty}
    \frac{\mm(B^{+}(o,R))}{R^N}
    =
    \lim_{R\to\infty}
    \int_{{Q}}\int_0^R\frac{ h_\alpha(t)}{R^N}\,dt \,\q(d\alpha)
  \\
  &
    =
    \int_{{Q}}
    \lim_{R\to\infty}
    \int_0^R\frac{ h_\alpha(t)}{R^N}
    \,dt \,\q(d\alpha)
  \\
  &
    =
    \int_C
    \lim_{R\to\infty}
    \int_0^R\frac{ h_\alpha(t)}{R^N}
    \,dt \,\q(d\alpha)
    +
    \int_{{Q}\backslash C}
    \lim_{R\to\infty}
    \int_0^R\frac{ h_\alpha(t)}{R^N}
    \,dt \,\q(d\alpha)
  \\
  &
    \leq
    \int_C
    \omega_N\AVR_X(1-\epsilon)
     \,\q(d\alpha)
    +
    \int_{{Q}\backslash C}
    \omega_N\AVR_X
    \,\q(d\alpha)
  \\
  &
    =
    \omega_N\AVR_X(1-\epsilon\q(C)),
\end{align*}
thus $\q(C)=0$.
By arbitrariness of $\epsilon$ we deduce that
$\lim_{R\to\infty}\int_0^R h_\alpha/R^N= \omega_N \AVR_X$, hence
$h_\alpha(t)=N\omega_N\AVR_X t^{N-1}$, for $\q$-a.e.\
$\alpha\in\tilde{Q}$.
\end{proof}

The proof of Theorem \ref{T:main1} is therefore concluded.
As described in the introduction, Theorem \ref{T:Euclid-application}
is an immediate consequence.

\appendix

\section{The relative perimeter as a Borel measure}
\label{S:perimeter-measure}

This appendix is devoted in proving that the relative perimeter can
be extended uniquely to a Borel measure.
Notice that in the result that follow, it is not needed the fact that
$\Lambda_F<\infty$, the forward-completeness, and local forward convexity.
We follow the line traced in~\cite{Mirandajr.03}.

We recall the definition of relative perimeter: fixed a Borel set
$E\subset \Omega$ of a measured Finsler manifold $(X,F,\mm)$, and fixed
$\Omega\subset X$, we define the perimeter of $E$ relative to $\Omega$
as
\begin{equation}
  \PP(E;\Omega)
  :=
  \inf\{
  \liminf_{n\to\infty}
  \int_{\Omega} |\partial u_n|\,d\mm
  : u_n\in\Lip_{loc}(\Omega)
  \text{ and }
  u_n\to \indicator_E \text{ in } L^1_{loc}(\Omega)
  \}
  .
\end{equation}
The infimum is clearly realized by a certain sequence $u_n$.
Using a truncation argument we may assume that $u_n$ takes values in
$[0,1]$; moreover, by passing to subsequences, we may also assume that
$u_n$ converges also in the $\mm$-a.e.\ sense.
If, in addition, $\Omega$ has finite measure, we may also assume (by
dominated convergence theorem) that $u_n\to\indicator_E$ in
$L^1(\Omega)$.
These assumptions will always be assumed tacitly, when dealing with a
sequence realizing the minimum in the definition of perimeter.

The slope satisfies calculus rules, in the $\mm$-a.e. sense
\begin{align}
    &
      |\partial(f+g)|\leq |\partial f|+|\partial g|
      ,
      \qquad
      |\partial(-f)|\leq\Lambda_F |\partial f|
      ,
    \\
    &
      |\partial(fg)|\leq f|\partial g|+g|\partial f|
      ,\qquad
      \text{ if }f,g\geq 0,
    \\
    &
      |\partial(fg)|\leq \Lambda_F (|f||\partial g|+|g||\partial f|)
      .
\end{align}
The proof is straightforward, once we know that $|\partial
f|(x)=F^*(-df(x))$, for $\mm$-a.e. $x\in X$.

The next lemma permits us to join two Lipschitz functions defined on
overlapping domains.
\begin{lemma}
  \label{lem:lipschitz-join}
  Let $(X,F,\mm)$ be a measured Finsler manifold.
  Let $N,M\subset X$ be two open sets such that $\partial
  M\cap\partial N=\emptyset$ and $\Lambda_{F,M\cap N}<\infty$.
  Then there exist an open set $H$ such that $\overline{H}\subset N\cap M$ and a constant
  $c=c(M,N)$ such that the following happen.
  For all $u\in\Lip_{loc}(M), v\in\Lip_{loc}(N)$, for all
  $\epsilon>0$, there exists a function $w\in\Lip_{loc}(M\cup N)$, such that
  \begin{equation}
    w=u\text{ in }M\backslash N,
    \qquad
    w=v\text{ in }N\backslash M,
    \qquad
    \min\{u,v\}\leq w\leq \max\{u,v\}
    \text{ in }
    M\cap N
    ,
  \end{equation}
   and it holds that
  \begin{equation}
    \label{eq:lipschitz-join}
    \int_{M\cup N}
    |\partial w|
    \,d\mm
    \leq
    \int_{M}
    |\partial u|
    \,d\mm
    +
    \int_{N}
    |\partial v|
    \,d\mm
    +
    c
    \int_H
    |v-w|
    \,d\mm
    +\epsilon
    .
  \end{equation}
\end{lemma}
\begin{proof}
  The hypothesis $\partial M\cap\partial N=\emptyset$ yields
  $\overline{M\backslash N}\cap\overline{N\backslash M}=\emptyset$.
  Define $d:=\inf\{\sfd(x,y),x\in M\backslash N, y\in N\backslash
  M\}$. and consider the function $\phi:M\cup N\to\R$ defined as
  \begin{equation}
    \phi(x)
    :=
    \max\left\{
      1-\frac{3}{d}\sup_{y\in B^{+}(M\backslash N,d/3)} \sfd(y,x)
      ,0
    \right\}
    .
  \end{equation}
  The function $\phi$ is $(3/d)$-Lipschitz and attains the values
  $1$ and $0$ in a neighborhood of $M\backslash N$ and
  $N\backslash M$, respectively.
  Define $H=\phi^{-1}((0,1))$.
  Clearly it holds $\overline{H}\subset M\cap N$.
  Fix now $\epsilon>0$ and find $k\in\N$ such that
  \begin{equation}
    \int_H (|\partial u|+|\partial v|)\,d\mm
    \leq
    \Lambda_{F,M\cap N}^{-2}\epsilon k
    .
  \end{equation}
  Define $H_i$ and $\psi_i$ ($i=1\dots k$) as
  \begin{equation}
    H_i=\phi^{-1}\left(\left(\frac{i-1}{k},\frac{i}{k}\right)\right)
    ,\qquad
    \psi_i= \min \left\{3\left(k\phi
        -i+\frac{2}{3}\right)^{+},1\right\}
    .
  \end{equation}
  Clearly, $\psi_i$ is $(9k/d)$-Lipschitz and it is locally
  constant outside $H_i$.
  Define $w_i=\psi_i u+(1-\psi_i) v$.
  We compute the slope of $w_i$ in $H_i$ using the calculus rules for
  the slope
  \begin{align*}
    |\partial w_i|
    &
    =
    |\partial (v+\psi_i(u-v))|
    \leq
      |\partial v|+|\partial (\psi_i(u-v))|
      \\&
    \leq
    |\partial v|+
    \Lambda_{F,M\cap N}|\partial\psi_i| |u-v|+
    \Lambda_{F,M\cap N}|\partial(u-v)|\psi_i
      \\&
    \leq
    |\partial v|
    +
    \frac{9k}{d}\Lambda_{F,M\cap N} |u-v|
    +
    \Lambda_{F,M\cap N}^2(|\partial u|+|\partial v|)
    .
  \end{align*}
  Outside $H_i$ the slope of $w_i$ is either $|\partial u$ or
  $\partial v$.
  Integrating over $M\cup N$, we obtain
  \begin{equation}
    \begin{aligned}
      \int_{M\cup N}
      |\partial w_i|
      \,d\mm
      &
      \leq
      \int_{M}
      |\partial u|
      \,d\mm
      +
      \int_{N}
      |\partial v|
      \,d\mm
      +
      \frac{9k\Lambda_{F,M\cap N}}{d}
      \int_{H_i}
        |u-v|\,d\mm
      \\
      &
        \qquad
      +
      \Lambda^2_{F,M\cap N}
      \int_{H_i}
      (
      |\partial u|+|\partial v|
      )
        \,d\mm
        .
    \end{aligned}
  \end{equation}
  Summing over $i$ and dividing by $k$, we deduce that
  \begin{equation*}
    \begin{aligned}
      \frac{1}{k}
      \sum_{i=1}^k
      \int_{M\cup N}
      |\partial w_i|
      \,d\mm
      &
      \leq
      \int_{M}
      |\partial u|
      \,d\mm
      +
      \int_{N}
      |\partial v|
      \,d\mm
      +
      \frac{9\Lambda_{F,M\cap N}}{d}
      \int_{H}
        |u-v|\,d\mm
        +
        \epsilon
        ,
    \end{aligned}
  \end{equation*}
  hence there exists an index $i_0$ such that $w=w_{i_0}$
  satisfies~\eqref{eq:lipschitz-join}, with
  $c=9\Lambda_{F,M\cap N}/d$.
\end{proof}

\begin{theorem}
  Let $(X,F,\mm)$ be a measured Finsler manifold, and let $E\subset X$ be a
  Borel set.
  Then it holds that
  \begin{enumerate}
  \item (Monotonicity)
    $\PP(E;A)\leq \PP(E;B)$, if $A\subset B$,
  \item (Superadditivity)
    $\PP(E;A\cup B)\geq \PP(E;A) +\PP(E;B)$, if $A\cap B=\emptyset$,
  \item (Inner regularity)
    $\PP(E;A)=\sup\{\PP(E;B):B\subset A\text{ has compact
      closure in } A\}$,
  \item (Subadditivity)
    $\PP(E;A\cup B)\leq \PP(E;A) + \PP(E;B)$,
  \end{enumerate}
  for all open sets $A$, $B$.

  Moreover, if for any Borel set $A$ we define
  $\PP(E;A):=\inf\{\PP(E;B):B\supset A\text{ is open}\}$, then the map
    $A\mapsto\PP(E;A)$ is a Borel measure.
\end{theorem}
\begin{proof}
The monotonicity and superadditivity are immediate consequences of the
definition of perimeter.
Let's consider the inner regularity.
Fix an open set $A$, such that $\sup\{\PP(E;B):B\subset A\}<\infty$
(otherwise the proof is trivial).
Find $(A_j)_{j}$ a sequence of open sets with compact closure such
that $\overline{A_j}\subset A_{j+1}$, and
$\bigcup_{j} A_j= A$, and define
$C_j=A_{2j}\backslash \overline{A_{2j-3}}$.
Since $C_{2j}\cap C_{2k}=\emptyset$, if $j\neq k$, by superadditivity,
we have that $\sum_j \PP(E; C_{2j})<\infty$, and analogously $\sum_j
\PP(E; C_{2j+1})<\infty$.
Fix $\epsilon>0$; there exists $J$, such that
\begin{equation}
  \sum_{j=J}^\infty \PP(E;C_j)
  \leq
  \epsilon 2^{-4}
  .
\end{equation}
Let $A:=C_{J+2}$, $B':=A_{J+1}$, $F_h:=C_{J+h-1}$, and
$G_h:=\bigcup_{i=1}^h F_i$; all these sets have compact closure, thus
the irreversibility constant is finite on these sets.

By definition of perimeter, there exists a sequence $\psi_{m,h}\in
\Lip_{loc}(F_h)$ such that $\psi_{m,h}\to \indicator_E$ in
$L^1(F_h)$ and
\begin{equation}
  \int_{F_h}|\partial \psi_{m,h}|\,d\mm
  \leq
  \PP(E;F_h)
  +2^{-2-m-h}
  .
\end{equation}
Notice that $G_h$ has compact closure, hence
$\Lambda_{F,G_n\cap F_{h+1}}<\infty$, thus we are in position to use
Lemma~\ref{lem:lipschitz-join} applied to the sets $G_h$ and
$F_{h+1}$.
Said Lemma gives a set $H_h\subset G_h\cap F_{h+1}$ and a constant
$c_h$, that will be used soon.
Clearly, up to passing to subsequences, we can assume that
\begin{equation}
  c_h
  \int_{H_h}
  |\psi_{m,h+1}-\psi_{m,h}|\,d\mm
  \leq \epsilon 2^{-10-h}
  .
\end{equation}

We define inductively on $h$ a sequence of functions
$u_{m,h}:G_h\to\R$ as follows.
For the initial step, take $u_{m,1}=\psi_{m,1}$.
For the inductive step, apply Lemma~\ref{lem:lipschitz-join} to the
functions $u_{m,h}$ and $\psi_{m,h+1}$ obtaining a function
$u_{m,h+1}$ such that
\begin{equation}
  \begin{aligned}
    \int_{G_{h+1}} |\partial u_{m,h+1}| \,d\mm
    &
  \leq
  \int_{G_{h}} |\partial u_{m,h}| \,d\mm
  +
      \int_{F_{h+1}} |\partial \psi_{m,h+1}| \,d\mm
    \\
    &
      \quad
  + c_h
  \int_{H_h}
  |u_{m,h}-\psi_{m,h+1}| \, d\mm
  +
  \epsilon 2^{-10-h}
  .
  \end{aligned}
\end{equation}
Since  $u_{m,h+1}=\psi_{m,h+1}$ on $F_{h+1}\backslash G_h$ and
$u_{m,h+1}=u_{m,h}$ on $G_h\backslash F_{h+1}$, we can deduce by
induction that
\begin{equation*}
  \begin{aligned}
    \int_{G_{h+1}} |\partial u_{m,h+1}| \,d\mm
    &
      \leq
      \sum_{i=1}^{h+1}
  \int_{F_{i}} |\partial \psi_{m,i}| \,d\mm
  +
      \sum_{i=1}^{h}
      \left(
      c_i
  \int_{H_i}
  |\psi_{m,i}-\psi_{m,i+1}| \, d\mm
  +
      \epsilon 2^{-10-i}
      \right)
    \\
    &
      \leq
      \sum_{i=1}^{h+1}
  \int_{F_{i}} |\partial \psi_{m,i}| \,d\mm
      +
      \epsilon 2^{-8}
      \leq
      \sum_{i=1}^{h+1}
      \PP(E;F_i)
      +2^{-m}
      +
      \epsilon 2^{-8}
  .
  \end{aligned}
\end{equation*}
We define $u_m(x)= u_{m,h}(x)$, whenever $x\in G_{h-1}$ (the
definition is well-posed) and we integrate its slope
\begin{align}
  \int_{A\backslash \overline{B'}}
  |\partial u_m|
  \,d\mm
  &
  \leq
    \lim_{h\to\infty}
    \int_{G_h} |\partial u_{m,h-1}|
    \,d\mm
    \leq
      \sum_{h=1}^{\infty}
    \PP(E;F_h)
    +
    2^{-m}
    +
    \epsilon 2^{-8}
  \\
  &
    =
    \sum_{h=1}^{\infty}
    \PP(E;C_{J+h-1})
    +
    2^{-m}
    +
    \epsilon 2^{-8}
    \leq
    \epsilon 2^{-3}+2^{-m}
    .
\end{align}
The sequence $u_m$ converges to $\indicator_E$ in $L^1(G_h)$ for all
$h$, hence it converges in $L^{1}_{loc}(A\backslash\overline{B'})$.

We take now $v_m\in \Lip_{loc}(B)$ converging to $\indicator_E$ in
$L^1(B)$ such that
\begin{equation}
  \PP(E;B)
  \leq
  \int_{B} |\partial v_m|\,d\mm
  +2^{-m}
  .
\end{equation}
We are in position to use Lemma~\ref{lem:lipschitz-join} again with
the sets $A\backslash\overline{B'}$ and $B$ and find an open set $H$
and
a constant $c$, such that for all $m$ there exists a function
$w_m:A\to\R$ such that
\begin{align*}
  \int_A
  |\partial w_m|
  \,d\mm
  &
  \leq
  \int_{A\backslash\overline{B'}}
  |\partial u_m|
  \,d\mm
  +
  \int_{B}
  |\partial v_m|
  \,d\mm
  +
  \int_{H}
  |u_m-v_m|
  \,d\mm
  +
    \epsilon 2^{-3}
  \\
  &
    \leq
    2^{-m}
    +\epsilon 2^{-3}
    +2^{-m}
  +
  \int_{H}
  |u_m-\indicator_E|
  \,d\mm
    +
  \int_{H}
  |\indicator_E-v_m|
    \,d\mm
    +
    \epsilon2^{-3}
    \\&
    \leq
  2^{1-m}
  +\epsilon
  +
  \int_{G_3}
  |u_m-\indicator_E|
  \,d\mm
    +
  \int_{B}
  |\indicator_E-v_m|
    \,d\mm
  .
\end{align*}
By taking the limit as $m\to\infty$, we deduce that
$\PP(E;A)\leq\PP(E;B)+\epsilon$, concluding the proof of the inner
regularity.

\smallskip

We prove now the subadditivity.
Fix $A$ and $B$ two open sets and let $A'$ and $B'$
compactly included in $A$ and $B$, respectively.
We will prove that $\PP(E;A'\cup B')\leq\PP(E;A')+\PP(E;B')$.
From this fact and the inner regularity, the subadditivity will
follow.
Consider $u_n\in \Lip_{loc}(A')$ and $v_n\in\Lip_{loc}(B')$ converging
in $L^1$ to $\indicator_E$, such that
\begin{equation}
  \int_{A'}|\partial u_n|\,d\mm
  \leq
  \PP(E;A')+\frac{1}{n}
  ,\qquad\text{ and }\qquad
  \int_{B'}|\partial v_n|\,d\mm
  \leq
  \PP(E;B')+\frac{1}{n}
  .
\end{equation}
Apply Lemma~\ref{lem:lipschitz-join} to the sets $A'$ and $B'$, and find
$H\subset A'\cap B'$ and $c>0$ such that, for all $n>0$, there exists a
function $w_n$ satisfying
\begin{equation}
  \int_{A'\cup B'}
  |\partial_n w_n|\,d\mm
  \leq
  \int_{A'}
  |\partial_n u_n|\,d\mm
  +
  \int_{B'}
  |\partial_n v_n|\,d\mm
  +
  c
  \int_H
  |u_n-v_n|\,d\mm
  +\frac{1}{n}
  .
\end{equation}
We conclude by taking the limit as $n\to\infty$.

\smallskip

The fact that the relative perimeter can be extended to a Borel
measure, is a consequence of a well-known Theorem of De Giorgi and
Letta~\cite{DeGiorgiLetta77}, that states that the conditions we have
just proven are sufficient to obtain such a measure.
\end{proof}

\section{Relaxation of the Minkowski content}
\label{S:minkowski}

In this appendix we give a proof of the fact that the perimeter can be
seen as the l.s.c.\ relaxation on the Minkowski content.
The proof follows the line of~\cite{AmDiGi17}, with some extra
attention to the irreversibility of the space.
In the case $X=\R^d$, this was already proven
in~\cite{ChambolleLisiniLussardi14}, with a different technique.

\begin{proposition}
  Let $(X,F,\mm)$ be a measured Finsler manifold and $E\subset X$ be a Borel set.
  Then it holds that
  \begin{equation}
    \mm^+(E)\geq \PP(E)
    .
  \end{equation}
\end{proposition}
\begin{proof}
  We consider the case $\mm^+(E)<\infty$ (the other is trivial).
  This implies that $\mm(\overline E\backslash E)=0$, hence, without loss of
  generality, we may assume that $E$ is closed.
  Consider the ${\epsilon}^{-1}$-Lipschitz function
\begin{equation}
  f_\epsilon(x) :=
  \max\left\{1-\frac{1}{\epsilon}\sup_{y\in B^{+}(E,\epsilon^2)}\sfd(y,x),0\right\}
  .
\end{equation}
Clearly $f_\epsilon\to\indicator_E$ in $L^1(\mm)$.
In $B^+(E,\epsilon^2)$ it is equal to $1$, hence $|\partial
f_\epsilon|(x)=0$, for all $x\in E$.
Conversely, in $X\backslash B^+(E,\epsilon+\epsilon^2)$ it attains its minimum, hence
$|\partial f_\epsilon|(x)=0$ for all $x\in X\backslash B^+(E,\epsilon+\epsilon^2)$.
We compute the integral
\begin{align*}
  \int_X |\partial f_\epsilon|(x)\,\mm(dx)
  &
  =
  \int_{B^+(E,\epsilon+\epsilon^2)\backslash E} |\partial f_\epsilon|(x)\,\mm(dx)
  \leq
  \int_{B^+(E,\epsilon+\epsilon^2)\backslash E}
    \frac{1}{\epsilon}\,\mm(dx)
  \\
  &
  =
    \frac{\mm(B^+(E,\epsilon+\epsilon^2)\backslash E)}{\epsilon}
    =
    (1+\epsilon)
    \frac{\mm(B^+(E,\epsilon+\epsilon^2))- \mm(E)}{\epsilon+\epsilon^2}
  .
\end{align*}
By taking the inferior limit as $\epsilon\to0$, we conclude.
\end{proof}

The previous proposition guarantees that the l.s.c.\ envelope of the
Minkowski content is not smaller than the perimeter.
The reverse is a bit more difficult and, at a certain point, we will
require forward-completeness.

We consider the ``semigroup'' $(T_t)_{t\geq0}$ given by the formula
\begin{equation}
  T_t f(x):=\sup_{y\in B^-(x,t)} f(y)
  ,\qquad
  T_0 f =f
  .
\end{equation}
Note that the ball in the supremum is backward.
The semigroup $T_t$ enjoys the following immediate property.
\begin{lemma}
  \label{lem:semigroup}
  It holds that $T_{t+s} f\geq T_t(T_s f)$ and, if $f$ is locally Lipschitz
  \begin{equation}
    \limsup_{t\to 0^+}
    \frac{T_t f-f}{t}
    \leq
    |\partial f|
    ,
    \qquad \mm\text{-a.e.\ in }X
    .
  \end{equation}
\end{lemma}
\begin{proof}
  Regarding the first part, fix $x\in X$, and $\epsilon>0$.
  By definition there exists $y$ such that $\sfd(y,x)<t$ and $(T_t(T_s
  f))(x)\leq (T_s f) (y)+\epsilon$.
  Similarly, there exists $z$ such that $\sfd(z,y)< s$ and $(T_s
  f)(y)\leq f(z)+\epsilon$.
  By triangular inequality, we have that $\sfd(z,x)<t+s$, thus
  \begin{equation}
    (T_{t+s}f)(x)
    \geq f(z)
    \geq
    (T_s f)(y)-\epsilon
    \geq
    (T_t(T_s f))(x)-2\epsilon
    .
  \end{equation}
  By arbitrariness of $\epsilon$, we conclude the first part.

  Regarding the second part, fix $x\in X$.
  By a direct computation we deduce
  \begin{align*}
    \limsup_{t\to 0^+}
    \frac{(T_t f)(x)-f(x)}{t}
    &
    =
    \inf_{r>0}
    \sup_{t\in(0,r)}
    \frac{
    \sup_{y\in B^-(x,t)} f(y)
    -f(x)
    }{t}
    \\
    &
      =
    \inf_{r>0}
    \sup_{t\in(0,r)}
    \sup_{y\in B^-(x,t)}
    \frac{
    (f(y)
    -f(x))^+
    }{t}
    \\
    &
      \leq
    \inf_{r>0}
    \sup_{t\in(0,r)}
    \sup_{y\in B^-(x,t)}
    \frac{
    (f(y)
    -f(x))^+
      }{\sfd(y,x)}
    \\
    &
      =
      \limsup_{y\to x}
    \frac{
    (f(y)
    -f(x))^+
      }{\sfd(y,x)}
      .
  \end{align*}
  If $x$ is a point where $f$ is differentiable, then the last term of
  the inequality above is equal to $F^*(-df)=|\partial f|(x)$,
  concluding the proof.
\end{proof}

We prove now a sort of coarea formula.

\begin{lemma}
  \label{lem:coarea-ugly}
  Consider $(X,F,\mm)$ a measured Finsler manifold.
  If $f:X\to [0,\R)$ is a Lipschitz function with compact support, it holds that
  \begin{equation}
    \int_0^\infty \mm^+(\{f\geq t\})\,dt
    \leq
    \int_X |\partial f|(x)\,\mm(dx)
    .
  \end{equation}
\end{lemma}
\begin{proof}
  In the first place, we notice that
  $\int_0^\infty \indicator_{\{f\geq t\}}(x)\,dt=f(x)$.
  Fix $t\geq 0$ and $h>0$.
  If $x\in B^{+}(\{f\geq t\},h)$, then $(T_h f)(x)\geq t$, or in other
  words $\indicator_{B^{+}(\{f\geq t\},h)}\leq\indicator_{\{(T_h f)\geq t\}}$.
  By integrating over $t$ we obtain
  \begin{equation}
    \int_0^\infty
    \indicator_{B^{+}(\{f\geq t\},h)}(x)
    \,dt
    \leq
    \int_0^\infty\indicator_{\{(T_h f)\geq t\}}(x)
    \,dt
    \leq
    (T_h f)(x)
    .
  \end{equation}
  By subtracting the first equation to the inequality above,
  integrating over $x$ and using Fubini's theorem, we obtain
  \begin{equation}
    \int_0^\infty
    \frac{\mm(B^{+}(\{f\geq t\},h))-\mm(\{f\geq t\})}{h}
    \,dt
    \leq
    \int_X \frac{(T_h f)(x)-f(x)}{h}
    \,\mm(dx)
    .
  \end{equation}
  The set $\{f\geq0\}$ is compact, hence for $h>0$ sufficiently small
  $B^+(\{f\geq0\},h)$ is compact.
  Moreover, $\frac{T_hf-f}{h}$ is smaller than the Lipschitz constant
  of $f$, hence the integrand in the r.h.s.\ is dominated by an $L^1$
  function.
  We take the inferior and superior limit in the l.h.s.\ and r.h.s.\
  (respectively) of the inequality above; the Fatou's Lemma brings us
  to the conclusion.
\end{proof}

We now prove that we can, without loss of generality, assume that the
functions of a sequence attaining the minimum in the definition of the
perimeter have compact support.
\begin{proposition}
  \label{prop:de-loc}
  Let $(X,F,\mm)$ be a forward-complete measured Finsler manifold, and let $E\subset X$ be a Borel
  set with finite measure.
  Then there exists a sequence of Lipschitz functions with compact
  support, $(w_n)_n$, such that $w_n\to \indicator_E$ in $L^1$ and
  $\PP(E)=\lim_{n\to\infty}\int_X |\partial w_n|\,d\mm$.
\end{proposition}
\begin{proof}
  Fix $E\subset X$ with finite measure, such that $\PP(E)<\infty$
  (otherwise the proof is trivial).
  Let $A_n:=B^+(o,n)$ for some $o$, fixed once and for all.
  Up to taking subsequences, we can assume that $\mm(E\backslash
  A_n)\leq 2^{-n}$.
  Let $\phi_n$ be the $3$-Lipschitz function given by
  \begin{equation}
    \phi_n(x)
    :=
    \left(
      1-3\inf_{y\in B^{+}(A_n,\frac{1}{3})} \sfd(y,x)
      \right)^+
    .
  \end{equation}
  This function takes value $1$ and $0$ in a neighborhood of
  $\overline{A_n}$ and $X\backslash A_{n+1}$, respectively.
  By definition of perimeter, there exist a sequence $u_n:A_n\to[0,1]$ of locally
  Lipschitz function such that
  \begin{equation}
  \PP(A_n)\geq 
  \int_{A_n} |\partial u_n|\,d\mm
  -2^{-n}
  ,\qquad
  \text{ and }
  \qquad
  \norm{u_n -\indicator_E}_{L^1(A_n)}\leq 2^{-n}
  .
\end{equation}
Define the function $w_n:=\phi_n u_{n+1}$.
This function is Lipschitz with compact support.
We compute its distance to $\indicator_E$
\begin{equation*}
  \int_X |w_n-\indicator_E|\,d\mm
  \leq
  \int_{A_n} |u_{n+1} -\indicator_E|\,d\mm
  +2\,\mm(E\cap A_{n+1}\backslash A_n)
  +\mm(E\backslash A_{n+1})
  \leq
  2^{3-n}
  ,
\end{equation*}
thus $w_n\to\indicator_E$ in $L^1(X)$.
Using the fact $|\partial w_n|\leq \phi_n |\partial u_{n+1}|+|\partial
\phi_n| u_{n+1}$, we deduce
\begin{align*}
  \int_{X}
  |\partial w_n|\,d\mm
  &
  \leq
  \int_{A_n}
  |\partial u_{n+1}|\,d\mm
  +
  \int_{A_{n+1}\backslash A_n}
  \phi_n|\partial u_{n+1}|\,d\mm
  +
  \int_{A_{n+1}\backslash A_n}
    u_{n+1}\,d\mm
  \\
  &
    \leq
  \int_{A_{n+1}}
    |\partial u_{n+1}|\,d\mm
    +2^{-1-n}
    \leq
    \PP(E;A_n)
    +2^{-n}
    \leq
    \PP(E)
    +2^{-n}
    .
    \qedhere
\end{align*}
\end{proof}

\begin{theorem}
  \label{T:lsc-envelope-minkowski}
  Consider $(X,F,\mm)$ a forward-complete measured Finsler manifold.
  Let $E\subset X$ be a Borel set with finite measure.
  Then there exists $(E_n)_{n}$, a sequence of compact sets, such
  that $\mm(E_n\bigtriangleup E)\to0$ and
  \begin{equation}
  \PP(E)\geq
  \limsup_{n\to\infty}
  \mm^+(E_n)
  .
\end{equation}
\end{theorem}
\begin{proof}
  Proposition~\ref{prop:de-loc} guarantees the existence of a sequence
  $(f_n)_n$ of
  Lipschitz functions with compact support, such that $f_n\to \indicator_E$ in
  $L^1(\mm)$ and
  \begin{equation}
    \PP(E)
    =
    \lim_{n\to\infty}
    \int_X |\partial f_n|(x)\,\mm(dx)
    .
  \end{equation}
  Clearly we may assume that $0\leq f_n\leq 1$.
  Fix $\epsilon\in(0,\frac{1}{2})$.
  By Lemma~\ref{lem:coarea-ugly}, there exists
  $t_n^\epsilon\in(\epsilon,1-\epsilon)$ such that
  \begin{equation}
    \mm^+(\{f_n\geq t_n^\epsilon\})
    \leq
    \frac{1}{1-2\epsilon}
    \int_{X}|\partial f_n|(x)\,\mm(dx)
    .
  \end{equation}
  Define $E_n^\epsilon:=\{f_n\geq t_n^\epsilon\}$.
  Since $\mm(E_n^\epsilon\bigtriangleup E)\to 0$, by taking an
  appropriate choice of
  $\epsilon=\epsilon_n$, we conclude.
\end{proof}
\bibliographystyle{acm}
\bibliography{literature.bib}

\end{document}